\RequirePackage{ifpdf}          
\ifpdf 
    \documentclass[10pt,pdftex]{article}
    \usepackage{hyperref} 
    \hypersetup{colorlinks, citecolor=MyDarkRed, filecolor=MyDarkBlue,
      linkcolor=MyDarkBlue, urlcolor=MyDarkBlue}
\else 
    \documentclass[10pt,dvips]{article}
    \usepackage{hyperref} 
    \hypersetup{colorlinks, citecolor=MyDarkRed, filecolor=MyDarkBlue,
      linkcolor=MyDarkRed, urlcolor=MyDarkBlue}
    \usepackage{breakurl}     
\fi

\usepackage[in]{fullpage}       
\usepackage{amsmath,amssymb}	
\usepackage{pslatex}		 
\usepackage[usenames]{color}   
\definecolor{MyDarkBlue}{rgb}{0, 0.0, 0.45} 
\definecolor{MyDarkRed}{rgb}{0.45, 0.0, 0} 
\definecolor{MyDarkGreen}{rgb}{0, 0.45, 0} 
\definecolor{MyLightGray}{gray}{.90}
\definecolor{MyLightGreen}{rgb}{0.5, 0.99, 0.5}
\DefineNamedColor{named}{Peach}         {cmyk}{0,0.50,0.70,0}
\DefineNamedColor{named}{Cyan}          {cmyk}{1,0,0,0}

\usepackage{graphicx} 
\usepackage{floatflt}
\usepackage{wrapfig}
\usepackage[parfill]{parskip}    
\usepackage{amsfonts, amscd, amssymb, amsthm, amsmath}
\usepackage{mathrsfs}
\usepackage{urwchancal}
\usepackage{xcolor}
\usepackage{MnSymbol}
\usepackage{pdfsync}
\usepackage{enumitem} 
\usepackage[mathscr]{euscript}
\theoremstyle{plain}
\newtheorem{thm}{Theorem}[section]
\newtheorem*{theorem-non}{Theorem}
\newtheorem{lem}[thm]{Lemma}
\newtheorem{prop}[thm]{Proposition}

\theoremstyle{definition}

\theoremstyle{remark}

\usepackage{color}
\definecolor{dred}{rgb}{.65, 0, 0.15}

\def\<{\langle} \def\>{\rangle}

\def\Tr{\mathrm{Tr}}

\usepackage{listings}           
\pagecolor{white}		

\begin{document}

\title{On the third moment of $L\big(\tfrac{1}{2}, \chi_{\scriptscriptstyle d}\big)$ {I}: \\ 
the rational function field case}	
\author{Adrian Diaconu \\         
}
\date{}	                
\maketitle

\begin{abstract} 
\noindent
In this note, we prove the existence of a secondary term in the asymptotic formula of the cubic moment of quadratic 
{D}irichlet {$L$}-functions 
$$ 
\sum_{\substack{d_{\scriptscriptstyle 0} - \mathrm{monic \; \& \; sq. \; free} \\ \deg \, d_{\scriptscriptstyle 0} \, = \, D}} \, 
L\big(\tfrac{1}{2}, \chi_{\scriptscriptstyle d_{\scriptscriptstyle 0}}\!\big)^{\scriptscriptstyle 3} 
$$
over rational function fields on the order of $q^{\scriptscriptstyle \frac{3}{4} D}.$ This term is in perfect analogy 
with the $x^{\scriptscriptstyle \frac{3}{4}}$-term indicated in our joint work \cite{DGH} for the corresponding 
asymptotic formula over the rationals. 
\end{abstract}

\tableofcontents 

\line(1,0){100}
\vskip-5pt
\noindent
{\small{School of Mathematics, University of Minnesota, Minneapolis, MN 55455\\
E-mail: cad@umn.edu}}

\section{Introduction} 

{\bf Statement of the main results.} Let $\mathbb{F}$ be a finite field with $q$ elements. \!For simplicity, we will assume throughout that 
$q\equiv 1 \!\pmod 4.$ For monic polynomials $d_{\scriptscriptstyle 0},\, m\in \mathbb{F}[x]$ with $d_{\scriptscriptstyle 0}$ square-free, 
let $\chi_{d_{\scriptscriptstyle 0}}\!(m) = (d_{\scriptscriptstyle 0}\slash m)$ be the usual quadratic residue symbol, and consider the generating series of the cubic moments of the central values of quadratic {D}irichlet {$L$}-functions 
\begin{equation*}
\mathscr{W}(\xi)\, = \sum_{D \ge 0}
\; \Bigg(\, \sum_{\substack{d_{\scriptscriptstyle 0} - \mathrm{monic \; \& \; sq. \; free} \\ \deg \, d_{\scriptscriptstyle 0} \, = \, D}} \, 
L\big(\tfrac{1}{2}, \chi_{\scriptscriptstyle d_{\scriptscriptstyle 0}}\!\big)^{\!\scriptscriptstyle 3} \Bigg)\, \xi^{\scriptscriptstyle D}.
\end{equation*} 
This series is absolutely convergent for complex $\xi$ with sufficiently small (depending upon the size $q$ of $\mathbb{F}$) absolute value.  

The main result of this note is the following 

\vskip10pt
\begin{thm} \label{Main Theorem A} --- The function $\mathscr{W}(\xi)$ has meromorphic continuation to the open disk 
$|\xi| < q^{\scriptscriptstyle - 2\slash 3}.$ It is analytic in this region, except for poles of order seven at $\xi = \pm \, q^{\scriptscriptstyle - 1}$ 
\!and simple poles at $\xi = \pm \, q^{\scriptscriptstyle - 3\slash 4}\!, \, \pm\,  i \, q^{\scriptscriptstyle - 3\slash 4},$ and the principal part at each 
of these poles is explicitly computable.
\end{thm} 

The principal parts of $\mathscr{W}(\xi)$ at the poles $\xi = \pm \, q^{\scriptscriptstyle - 1}$ can be computed as in \cite[Section~3.2]{DGH}, and accordingly will not be discussed any further here. \!The residues at the 
remaining (simple) poles will be computed in Section \ref{proof Thm A}, see \eqref{eq: generic-residue-3/4}.

Let $\zeta(s) = \zeta_{\scriptscriptstyle \mathbb{F}(x)}(s)$ denote the zeta function of the field $\mathbb{F}(x).$ 
As a consequence of Theorem \ref{Main Theorem A}, we have the following asymptotic formula for the cubic 
moments of the central values of quadratic {D}irichlet {$L$}-functions. 
\vskip10pt
\begin{thm} \label{Main Theorem B} --- For every small $\delta > 0$ and $D \in \mathbb{N},$ we have    
\begin{equation*} 
\sum_{\substack{d_{\scriptscriptstyle 0} - \mathrm{monic \; \& \; sq. \; free} \\ \deg \, d_{\scriptscriptstyle 0} \, = \, D}} \, 
L\big(\tfrac{1}{2}, \chi_{\scriptscriptstyle d_{\scriptscriptstyle 0}}\!\big)^{\!\scriptscriptstyle 3}
= \, \frac{q^{\scriptscriptstyle D}}{\zeta(2)} Q(D, q)\, + \, q^{\scriptscriptstyle \frac{3}{4} D} R(D, q)
\, + \, O_{{\scriptstyle \delta}, \, q}\Big(q^{\scriptscriptstyle D \left(\! \frac{2}{3} + \delta \!\right)}\Big)
\end{equation*} 
for explicitly computable $Q(D, q)$ and $R(D, q).$
\end{thm}

An expression for $Q(D, q)$ can be easily obtained from the principal parts of $\mathscr{W}(\xi)$ at 
$\xi = \pm \, q^{\scriptscriptstyle - 1}\!.$ We will not pursue this calculation here, as there are alternative ways to 
compute $Q(D, q)$ (see \cite[Section~8 (a)]{RW} and \cite[Section~5.3]{AK}). 
The computation of $R(D, q)$ will be given in Section \ref{proof Thm B}. 
However, for the convenience of the reader, we give here the expression of $R(D, q);$ if   
\begin{equation*}
P(x) : = (1 - x)^{\scriptscriptstyle 5} (1 + x)  (1 + 4 \, x + 11 x^{\scriptscriptstyle 2} + 10 \, x^{\scriptscriptstyle 3} 
- 11 x^{\scriptscriptstyle 4} + 11 x^{\scriptscriptstyle 6} - 4 \, x^{\scriptscriptstyle 7} - x^{\scriptscriptstyle 8}) 
= 1 - 14 \, x^{\scriptscriptstyle 3} - x^{\scriptscriptstyle 4} + 78 \, x^{\scriptscriptstyle 5} + \cdots
\end{equation*} 
is the polynomial appearing in \cite{Zha}, then 
\begin{equation*} 
\begin{split}
R(D, q) = & \; \tfrac{1}{4} \small{(1 + q^{\scriptscriptstyle 1\slash 4} + 10 \, q^{\scriptscriptstyle 1\slash 2} 
+ 7  q^{\scriptscriptstyle 3 \slash 4} 
\, + 20 \, q + 7  q^{\scriptscriptstyle 5\slash 4} + 10 \, q^{\scriptscriptstyle 3\slash 2} + q^{\scriptscriptstyle 7\slash 4} + q^{\scriptscriptstyle 2})} \,  \zeta\big(\tfrac{1}{2}\big)^{\! \scriptscriptstyle 7}
\cdot \, \prod_{p}  P\, \Big(1\slash \sqrt{|p|} \, \Big) \\
& + \tfrac{(- 1)^{\scriptscriptstyle D}}{4} \small{(1 - q^{\scriptscriptstyle 1\slash 4} + 10 \, q^{\scriptscriptstyle 1\slash 2} 
- 7  q^{\scriptscriptstyle 3 \slash 4} \, + 20 \, q - 7  q^{\scriptscriptstyle 5\slash 4} + 10 \, q^{\scriptscriptstyle 3\slash 2} - q^{\scriptscriptstyle 7\slash 4} + q^{\scriptscriptstyle 2})} \,  \zeta\big(\tfrac{1}{2}\big)^{\! \scriptscriptstyle 7}
\cdot \, \prod_{p}  P\, \Big(1 \slash \sqrt{|p|} \, \Big) \\
& + \tfrac{1}{2}\, \Re\big(i^{\, \scriptscriptstyle D} \small{(1 - i \, q^{\scriptscriptstyle 1\slash 4}\! 
- 4 \, q^{\scriptscriptstyle 1\slash 2} 
+ 7 i  \, q^{\scriptscriptstyle 3 \slash 4} + 6 \, q - 7 i \, q^{\scriptscriptstyle 5\slash 4}\! 
- 4 \, q^{\scriptscriptstyle 3\slash 2} + i \, q^{\scriptscriptstyle 7\slash 4} + q^{\scriptscriptstyle 2})}\big)\, 
L\big(\tfrac{1}{2}, \chi_{{\scriptscriptstyle \theta_{\scriptscriptstyle 0}}}\big)^{\! \scriptscriptstyle 7}
\cdot \prod_{p}  P\, \Big((- 1)^{\deg \, p}\slash \sqrt{|p|} \, \Big)
\end{split}
\end{equation*} 
the products in the right-hand side being over all monic irreducibles of $\mathbb{F}[x],$ and where 
\begin{equation*} 
\zeta\big(\tfrac{1}{2}\big) = \frac{1}{1 - \sqrt{q}} \;\;\; \text{and} \;\;\;  
L\big(\tfrac{1}{2}, \chi_{{\scriptscriptstyle \theta_{\scriptscriptstyle 0}}}\big) = \frac{1}{1 + \sqrt{q}}. 
\end{equation*} 
Similar results over the rationals will appear in the forthcoming manuscript \cite{DW}. 

\vskip5pt
{\bf Relation to previous work.} Understanding the moments of various families of {$L$}-functions at the center 
of symmetry is a very important problem in analytic number theory. A classical example is the family of quadratic 
{D}irichlet {$L$}-series\\ whose moments attracted considerable attention over the years. Heuristics to determine 
the main terms in the asymptotic formula for the moments of this family were given in \cite{CFKRS} and \cite{DGH}. 
\!Besides the conjectural main terms, this asymptotic formula should also have a finer part consisting generically of 
infinitely many lower order terms. The first\footnote{The lower order terms we are referring to are all of magnitude 
larger than the threshold $x^{\scriptscriptstyle \frac{1}{2}}.$ \!The additional term noticed in \cite{Florea1}, besides 
being certainly of different origin, it is a special feature of the first moment.} instance when such a lower 
order term occurs is the cubic moment, and to justify this assertion is the subject of this note and \cite{DW}; 
so far, the evidence supporting the existence of this additional term was limited to the conditional result in 
\cite{Zha}, and the extensive computations and experiments in \cite{AR}. This particular moment is the 
highest over the rationals for which an asymptotic formula has been established, see \cite{Sound}, \cite{DGH} and \cite{Young}. In the rational function field case, the corresponding asymptotics for the third and fourth moments 
have been established in \cite{Florea2} and \cite{Florea3}\footnote{In this case, the asymptotic formula proved 
contains only the leading three terms.}, respectively. \!It is by no means a coincidence that the error terms in both \cite{Young} and \cite{Florea2} are of size comparable to the size of the corresponding secondary terms asserted.

The approach we take is based on Weyl group multiple {D}irichlet series. \!\!These are series associated to 
root systems over global fields (containing sufficient roots of unity) of the form
\begin{equation*} 
Z(\mathrm{s}; \mathrm{m}, \Psi) \;\; = \sum_{\mathrm{n} = (n_{\scriptscriptstyle 1}\!, \ldots, n_{r})}\,  
\frac{H(\mathrm{n}; \mathrm{m}) \Psi(\mathrm{n})}
{\prod |n_{\scriptscriptstyle i}|^{s_{\scriptscriptstyle i}}}
\end{equation*} 
$\mathrm{m} = (m_{\scriptscriptstyle 1}, \ldots, m_{r})$ being a twisting parameter, satisfying (Weyl) groups 
of functional equations; see for details \cite{Bump}, \cite{BBFH} and \cite{CG2}. 
If $m_{\scriptscriptstyle i} = 1$ for all $i,$ the series is said to be {\it untwisted}. 
The most important part of $Z(\mathrm{s}; \mathrm{m}, \Psi)$ is the function $H,$ 
giving the structure of the multiple {D}irichlet series. 
Via a twisted multiplicativity (see, for example, \cite{BBFH} and \cite{CG2}), 
this function is determined by its values on prime powers. \!Equivalently, the multiple {D}irichlet series is determined 
by its $p$-parts\footnote{In \cite[Corollary~5.8]{CFG} it is shown that the $p$-parts of untwisted Weyl group multiple 
{D}irichlet series constructed from quadratic characters are uniquely determined. \!This implies the remarkable fact 
that untwisted quadratic Weyl group multiple {D}irichlet series over rational function fields coincide, 
after a simple change of variables, with their own $p$-parts.}, i.e., 
\begin{equation*} 
\sum_{k_{\scriptscriptstyle 1}\!, \ldots, \, k_{r} \, \ge \, 0}  
H\big(p^{k_{\scriptscriptstyle 1}}\!, \ldots, p^{k_{r}}\!;  p^{l_{\scriptscriptstyle 1}}\!, \ldots, p^{l_{r}}\big) \,
|p|^{- \, k_{\scriptscriptstyle 1} s_{\scriptscriptstyle 1} - \, \cdots \, - \, k_{r} s_{r}}
\; \qquad \; {(\text{with $p^{l_{\scriptscriptstyle i}} \!\parallel m_{\scriptscriptstyle i},$ $p$ prime})}.
\end{equation*} 
There are several different methods of representing the $p$-parts of multiple {D}irichlet series, namely, 
\begin{itemize}
  \item Definition by the ``averaging method", sometimes known as the Chinta-Gunnells method, see \cite{CG1} and  
  \cite{CG2}.  
  \item Definition as spherical $p$-adic {W}hittaker functions, see \cite{BBF2} and \cite{BBF3}.
 \item Definition as sums over crystal bases, see \cite{BBF2} and \cite{McN1}. 
 \item Definition as partition functions of statistical-mechanical lattice models, see \cite{BBB} and \cite{BBBF}.
\end{itemize}
The equivalence of the Chinta-Gunnells method with the Whittaker definition was established by McNamara 
\cite{McN2}. 

The Chinta-Gunnells method and the Whittaker definition were recently extended to infinite root systems 
in \cite{Lee-Zh} and \cite{PP}, respectively, and in \cite{DV} the author, in joint work with Pa\c sol, applied 
Deligne's theory of weights in the context of moduli spaces of admissible double covers to express the 
coefficients of the $p$-parts of untwisted multiple {D}irichlet series associated to arbitrary moments of 
quadratic {D}irichlet {$L$}-series in terms of $q$-Weil numbers, where $q = |p|.$ The axiomatic approach 
introduced in \cite{DV} has also been applied in \cite{White1} and \cite{White2} to construct untwisted 
Weyl group multiple {D}irichlet series associated to affine root systems.

\vskip5pt
{\bf Overview of the argument.} \!The main ideas involved in the proof of Theorem \ref{Main Theorem A} can be summarized 
as follows. \!As in \cite{DGH}, we first write  
\begin{equation*} 
\mathscr{W}(q^{- w}) \, = 
\sum_{h - \mathrm{monic}}\; \underbrace{\mu(h)}_{\textrm{\textcolor{MyDarkBlue}
{${\scriptstyle{\text{M\"obius function on $\mathbb{F}[x]$}}}$}}} \hskip-15pt Z\big(\tfrac{1}{2}, \tfrac{1}{2}, \tfrac{1}{2}, w, 1\!; h\big) 
\qquad \text{(for $\Re(w) > 1$)}
\end{equation*} 
where $Z$ is a multiple {D}irichlet series with a certain congruence condition. For every monic and square-free polynomial $h,$ 
this function will be expressed in terms of twisted (in the sense of \cite{DGH}) multiple {D}irichlet series. \!Unlike \cite{DGH}, 
the formula we use (see Section \ref{different-form}) is a finite sum of terms of the form 
\begin{equation*}
|h|^{- 2 w}\,  \chi_{c_{\scriptscriptstyle 2}}\!(c_{\scriptscriptstyle 1}) 
\, \underbrace{Z^{(h)}\big(\tfrac{1}{2}, \tfrac{1}{2}, \tfrac{1}{2}, w; 
\chi_{c_{\scriptscriptstyle 2}}, \chi_{c_{\scriptscriptstyle 1}}\big)}_{\textrm{\textcolor{MyDarkBlue}
{${\scriptstyle{\text{multiple {D}irichlet series}}}$}}}\,  
\prod_{p \, \mid \, c_{\scriptscriptstyle 1}} F(|p|^{\,\scriptscriptstyle - \frac{1}{2}}\!, \ldots, \, |p|^{\, - w}; \, |p|)\,  |p|^{\, - w}
\, \cdot \, \prod_{p \, \mid \, c_{\scriptscriptstyle 2} c_{\scriptscriptstyle 3}} 
G^{\scriptscriptstyle (\varepsilon_{\scriptscriptstyle p})}(|p|^{\,\scriptscriptstyle - \frac{1}{2}}\!, \ldots, \, |p|^{\, - w}; \, |p|)
\end{equation*} 
with $c_{\scriptscriptstyle i}$ monic, $h = c_{\scriptscriptstyle 1}  c_{\scriptscriptstyle 2}  c_{\scriptscriptstyle 3},$ and for each 
monic irreducible $p \, \mid \, c_{\scriptscriptstyle 2} c_{\scriptscriptstyle 3},$ the quantity 
$\varepsilon_{\scriptscriptstyle p} = 0$ or $1$ according as $p$ divides $c_{\scriptscriptstyle 3}$ 
or $p$ divides $c_{\scriptscriptstyle 2}.$ \!The functions $F$ and $G^{(\varepsilon_{\scriptscriptstyle p})}$ represent 
a (normalized) partition of the local $p$-part of the $D_{4}$-untwisted {W}eyl group multiple {D}irichlet series 
(associated to the cubic moment) corresponding to odd and even weighted monomials, and with negative degree terms 
in $|p|^{- w}$ removed. 

We will prove that the above series representation of $\mathscr{W}(q^{- w})$ converges absolutely and uniformly on every 
compact subset of the half-plane $\Re(w) > 2 \slash 3,$ away from the points $w \in \mathbb{C}$ for which 
$q^{- w} = \pm \, q^{\scriptscriptstyle - 1}\!,$ or 
$q^{- w} = \pm \, q^{\scriptscriptstyle - 3\slash 4}\!, \, \pm\,  i \, q^{\scriptscriptstyle - 3\slash 4}.$ To show this, 
we will exhibit additional decay of the function $Z\big(\tfrac{1}{2}, \tfrac{1}{2}, \tfrac{1}{2}, w, 1\!; h\big)$ in $|h|.$ This will 
be done in two steps: 
\begin{itemize}
  \item[\textasteriskcentered] We first show that, for $\Re(w) \ge \tfrac{1}{2},$ the functions   
$F$ and $G^{\scriptscriptstyle (0)}$ are bounded, independent of $w,\,  |p|,$ and 
\begin{equation*}
G^{\scriptscriptstyle (1)}(|p|^{\,\scriptscriptstyle - \frac{1}{2}}\!, \, |p|^{\,\scriptscriptstyle - \frac{1}{2}}\!,  
\, |p|^{\,\scriptscriptstyle - \frac{1}{2}}\!, \, |p|^{\, - w}; \, |p|) \ll |p|^{\scriptscriptstyle - \frac{1}{2}}
\end{equation*}
see Lemma \ref{Estimate Zloc}. \!These estimates provide sufficient decay in the parameters $c_{\scriptscriptstyle 1}$ 
and $c_{\scriptscriptstyle 2}.$ 

\item[\textasteriskcentered] To obtain the required decay in the remaining parameter, we use again the properties 
of the $p$-parts of the $D_{4}$-untwisted {W}eyl group multiple {D}irichlet series combined with an inductive argument, 
to improve upon the convexity bound \eqref{eq: basic-initial-estimate} of 
$
Z^{(h)}\big(\tfrac{1}{2}, \tfrac{1}{2}, \tfrac{1}{2}, w; \chi_{c_{\scriptscriptstyle 2}}, \chi_{c_{\scriptscriptstyle 1}}\big)
$ 
in the $c_{\scriptscriptstyle 3}$-aspect, see Proposition \ref{key-proposition}.
\end{itemize}

The reader will no doubt have noticed the special role played by the $p$-parts in the argument. In \cite{DV} 
it is shown that the coefficients of these generating series can be expressed in terms of the eigenvalues of 
Frobenius acting on the $\ell$-adic \'etale cohomology of moduli of admissible double covers of genus zero 
{\it stable} curves with marked points, hence in terms of Weil algebraic integers. \!Thus, the more conceptual 
reason behind the asymptotics and estimates discussed in Section \ref{Estim} is precisely the {\it dominance} 
condition (see \cite{DV}) satisfied by the $p$-parts of the untwisted multiple {D}irichlet series associated to any 
(not just cubic) moment of quadratic {D}irichlet {$L$}-functions. \!However, in the present context we take 
advantage of the completely explicit nature of the  $D_{4}$-{W}eyl group multiple 
{D}irichlet series (see Appendix \ref{Appendix B}) to deduce the relevant facts about its $p$-parts.

\vskip5pt
{\bf Acknowledgements.} I would like to thank the organizers of the {\it fifth Bucharest number theory day} 
conference, which motivated me to write this note.

\section{Notation and preliminaries} \label{prelim}  
Let $\mathbb{F}$ be a finite field with $q\, \equiv 1 \!\!\pmod 4$ elements. \!For a non-zero $m\in \mathbb{F}[x],$ 
we define its norm by $|m| = q^{\deg \, m}.$ For polynomials $d, m \in \mathbb{F}[x],$ with $m$ monic, let 
$(d \slash m)$ denote the Kronecker symbol, defined as a completely multiplicative function of $m,$ for every 
fixed $d,$ and if $m = p$ is irreducible then $(d \slash p) = 0$ if $p\mid d$ and $(d \slash p) = \pm 1$ if $p\nmid d,$ 
the $+$ or $-$ sign being determined according to whether $d$ is congruent to a square modulo $p$ or not; we take 
$(d \slash 1) = 1.$ The symbol $(d \slash m)$ is also completely multiplicative as a function of $d,$ for every $m.$ Since 
we are assuming that $q\equiv 1 \!\!\pmod 4,$ we have the simpler quadratic reciprocity law:  
\begin{equation*} 
\left(\frac{d}{m} \right) =  \left(\frac{m}{d} \right) \qquad 
\text{(for coprime non-constant monic polynomials $d, m \in \mathbb{F}[x]$).}
\end{equation*} 
In addition, if $b \in \mathbb{F}^{\times}$ then $\left(\frac{b}{m}\right) = \mathrm{sgn}(b)^{\deg \, m}$ for all non-constant 
$m \in \mathbb{F}[x],$ where, for 
$
d(x) = b_{\scriptscriptstyle 0}\, x^{n} + b_{\scriptscriptstyle 1}\,  x^{n - 1} + \cdots +  b_{\scriptscriptstyle n}
\! \in \mathbb{F}[x]
$ 
($b_{\scriptscriptstyle 0} \ne 0$), we define $\mathrm{sgn}(d) = 1$ if 
$
b_{\scriptscriptstyle 0} \in (\mathbb{F}^{\times})^{\scriptscriptstyle 2}
$ 
and $\mathrm{sgn}(d) = - 1$ if 
$
b_{\scriptscriptstyle 0} \notin (\mathbb{F}^{\times})^{\scriptscriptstyle 2}.
$ 

For $d = d_{\scriptscriptstyle 0}\footnote{Very often in this work, a monic polynomial $d$ will be expressed as 
$d = d_{\scriptscriptstyle 0}^{} d_{\scriptscriptstyle 1}^{\, 2}$ with $d_{\scriptscriptstyle 0}^{}$ monic and square-free, 
which justifies the notation.}$ square-free, let   
$
\chi_{d_{\scriptscriptstyle 0}}\!(m) = (d_{\scriptscriptstyle 0} \slash m).
$ 
The {$L$}-series attached to the character $\chi_{d_{\scriptscriptstyle 0}}$ is defined by 
\begin{equation*}
L(s, \chi_{d_{\scriptscriptstyle 0}}) \;\; = \sum_{\substack{m \in \mathbb{F}[x] \\ m - \text{monic}}} 
\chi_{d_{\scriptscriptstyle 0}}\!(m)  |m|^{-s}
= \, \prod_{p - \text{monic \& irred.}} (1 - \chi_{d_{\scriptscriptstyle 0}}\!(p)  |p|^{-s})^{\scriptscriptstyle -1}
\qquad \text{(for complex $s$ with $\Re(s) > 1$).}
\end{equation*} 
It is well-known that $L(s, \chi_{d_{\scriptscriptstyle 0}})$ is a polynomial in $q^{-s}$ of degree $\deg \, d_{\scriptscriptstyle 0} - 1$ 
when $d_{\scriptscriptstyle 0}$ is non-constant; if $d_{\scriptscriptstyle 0} \in \mathbb{F}^{\times}$ then   
\begin{equation*}
L(s, \chi_{d_{\scriptscriptstyle 0}}) = \, \zeta(s) \, = \frac{1}{1 - q^{{\scriptscriptstyle 1} - s}} 
\;\;\;\; \text{(when $\mathrm{sgn}(d_{\scriptscriptstyle 0}) = 1$)}
\;\;\; \text{and} \;\;\;  L(s, \chi_{d_{\scriptscriptstyle 0}}) \, = \frac{1}{1 + q^{{\scriptscriptstyle 1} - s}} 
\;\;\;\; \text{(when $\mathrm{sgn}(d_{\scriptscriptstyle 0}) = - 1$).}
\end{equation*} 
Moreover, if we define $\gamma_{\scriptscriptstyle q}(s,\,  d)$ by 
\begin{equation} \label{eq: gamma-s-d}
\gamma_{\scriptscriptstyle q}(s,\,  d) : = 
q^{\frac{1}{2}{\scriptscriptstyle \left(3 \, + \, (-1)^{\deg \, d} \right)} \left(s - \frac{1}{2}\right)}\, 
\big(\! 1  -  \mathrm{sgn}(d) q^{- s} \big)^{\scriptscriptstyle \left(1 \, + \, (-1)^{\deg \, d}\right)\slash 2}\, 
\big(\! 1 - \mathrm{sgn}(d) q^{s - 1} \big)^{\scriptscriptstyle - \,  \left(1 \, + \, (-1)^{\deg \, d}\right)\slash 2}
\end{equation}
then the function $L(s, \chi_{d_{\scriptscriptstyle 0}})$ satisfies the functional equation 
\begin{equation} \label{eq: funct-eq-L}
L(s, \chi_{d_{\scriptscriptstyle 0}}) = \gamma_{\scriptscriptstyle q}(s,\,  d_{\scriptscriptstyle 0}) 
|d_{\scriptscriptstyle 0}|^{\frac{1}{2} - s} L(1 - s, \chi_{d_{\scriptscriptstyle 0}}).
\end{equation}

\subsection{The Chinta-Gunnells action} \label{Chi-Gunn} 
We shall now recall an important technique developed by Chinta and Gunnells \cite{CG1} to produce certain rational functions 
associated to classical root systems, which they subsequently used as {\it building blocks} to construct {W}eyl group 
multiple {D}irichlet series (over any global field) twisted by quadratic characters. \!Strictly speaking, we shall apply this 
construction only when the underlying root system is $D_{4},$ and therefore, the material included in Appendix 
\ref{Appendix B} suffices for our purposes. \!However, we feel that the approach taken here is applicable to similar problems 
in other contexts, and for this reason, we opted to present this background material in some generality.

Let $\Phi$ be a rank $r$ irreducible simply-laced root system, and let $W = W(\Phi)$ denote the Weyl group of 
$\Phi.$ Fix an ordering of the roots and decompose $\Phi = \Phi^{+} \cup\,  \Phi^{-}$ into positive and negative roots. 
Let $\alpha_{\scriptscriptstyle 1}, \alpha_{\scriptscriptstyle 2}, \ldots, \alpha_{r}$ be the simple roots and let 
$\sigma_{\scriptscriptstyle i} \in W$ be the simple reflection through the hyperplane perpendicular to 
$\alpha_{\scriptscriptstyle i}.$ The simple reflections generate the Weyl group and satisfy the relations 
$
(\sigma_{\scriptscriptstyle i}  \, \sigma_{\! \scriptscriptstyle j})^{r_{\scriptscriptstyle i j}} = 1 
$ 
with $r_{\scriptscriptstyle i i} = 1$ for all $i,$ and $r_{\scriptscriptstyle i j} \in \{2, 3 \}$ if $i \ne \!  j.$ The indices $i$ and 
$j$ are said to be {\it adjacent} if $i \ne \! j$ and $r_{\scriptscriptstyle i j} = 3.$ The action of the simple reflections on the roots 
is given by 
\begin{equation*}
\sigma_{\scriptscriptstyle i} \, \alpha_{\scriptscriptstyle j} =\,  
\begin{cases} 
\alpha_{\scriptscriptstyle i} + \alpha_{\scriptscriptstyle j} &\mbox{if $i$ and \!$j$ are adjacent} \\
- \alpha_{\scriptscriptstyle j} & \mbox{if $i = \! j$} \\
\alpha_{\scriptscriptstyle j} & \mbox{otherwise.}
\end{cases}  
\end{equation*} 
Let $\Lambda(\Phi)$ denote the root lattice of $\Phi.$ Every element $\lambda$ of the root lattice has a unique representation 
as an integral linear combination of the simple roots 
\begin{equation*}
\lambda \, = \sum_{i = 1}^{r} \, k_{\scriptscriptstyle i} \, \alpha_{\scriptscriptstyle i}. 
\end{equation*} 
Let 
$
d(\lambda) : = \sum_{i = 1}^{r} k_{\scriptscriptstyle i}  
$ 
be the height function on $\Lambda(\Phi).$ 

In this setting, Chinta and Gunnells \cite{CG1} introduced a Weyl group action on the field of rational functions 
$\mathbb{C}(z_{\scriptscriptstyle 1}, \ldots, z_{r})$ in $r$ variables and used it to construct multiple {D}irichlet series over 
global fields having analytic continuation to $\mathbb{C}^{r}$ and satisfying a group of functional equations 
isomorphic to $W.$

To define this group action, denote ${\bf{\mathrm{z}}} = (z_{\scriptscriptstyle 1}, \ldots, z_{r}),$ and 
for $\lambda \in \Lambda(\Phi),$ set 
$
{\bf{\mathrm{z}}}^{\lambda}: = z_{\scriptscriptstyle 1}^{k_{\scriptscriptstyle 1}} \cdots z_{r}^{k_{r}} 
$ 
with $k_{\scriptscriptstyle i}$ determined by $\lambda$ as above. Following \cite{CG1}, define 
$
^{{\scriptscriptstyle \epsilon_{\scriptscriptstyle i}}}{\bf{\mathrm{z}}} =  {\bf{\mathrm{z}}}',
$ 
where  
\begin{equation*}
z_{\! \scriptscriptstyle j}' = \, 
\begin{cases} 
- \, z_{\! \scriptscriptstyle j} &\mbox{if $i$ and \!$j$ are adjacent} \\
\; z_{\! \scriptscriptstyle j} & \mbox{otherwise}
\end{cases}  
\end{equation*} 
and 
$
^{{\scriptscriptstyle \sigma_{\scriptscriptstyle i}}}{\bf{\mathrm{z}}} =  {\bf{\mathrm{z}}}',
$ 
where 
\begin{equation*}
z_{\! \scriptscriptstyle j}' = \, 
\begin{cases} 
\sqrt{q} \, z_{\scriptscriptstyle i} \, z_{\! \scriptscriptstyle j} 
&\mbox{if $i$ and \!$j$ are adjacent} \\
1\slash (q \, z_{\! \scriptscriptstyle j}) & \mbox{if $i = \! j$} \\
\; z_{\! \scriptscriptstyle j} & \mbox{otherwise.}
\end{cases}  
\end{equation*} 
Here $q \ge 1$ is a fixed parameter. One checks easily that 
$
^{{\scriptscriptstyle \epsilon_{\scriptscriptstyle i}^{\scriptscriptstyle 2}}}\!{\bf{\mathrm{z}}} = {\bf{\mathrm{z}}},
$ 
$
^{{\scriptscriptstyle \epsilon_{\scriptscriptstyle i} \, \epsilon_{\! \scriptscriptstyle j}}}{\bf{\mathrm{z}}} \, = \, 
^{{\scriptscriptstyle \epsilon_{\! \scriptscriptstyle j} \, \epsilon_{\scriptscriptstyle i}}}{\bf{\mathrm{z}}}, 
$ 
and that  
\begin{equation*}
^{{\scriptscriptstyle \sigma_{\scriptscriptstyle i} \, \epsilon_{\! \scriptscriptstyle j}}}{\bf{\mathrm{z}}} = \, 
\begin{cases} 
^{{\scriptscriptstyle \epsilon_{\scriptscriptstyle i} \, \epsilon_{\! \scriptscriptstyle j}\, \sigma_{\scriptscriptstyle i}}}{\bf{\mathrm{z}}}
&\mbox{if $i$ and \!$j$ are adjacent} \\
^{{\scriptscriptstyle \epsilon_{\! \scriptscriptstyle j}\, \sigma_{\scriptscriptstyle i}}}{\bf{\mathrm{z}}} & \mbox{otherwise.}
\end{cases}  
\end{equation*} 
Letting 
$
f_{\scriptscriptstyle i}^{\scriptscriptstyle \pm}({\bf{\mathrm{z}}}) = \left(f({\bf{\mathrm{z}}}) 
\pm  f(^{{\scriptscriptstyle \epsilon_{\scriptscriptstyle i}}}{\bf{\mathrm{z}}})\right)\slash 2,
$ 
for $f \in \mathbb{C}({\bf{\mathrm{z}}}),$ one defines the action of a simple reflection 
$\sigma_{\scriptscriptstyle i}$ on $\mathbb{C}({\bf{\mathrm{z}}})$ by 
\begin{equation*}
(f\, \vert \, \sigma_{\scriptscriptstyle i})({\bf{\mathrm{z}}}) =
- \, \frac{1 - q z_{\scriptscriptstyle i}}{q z_{\scriptscriptstyle i}(1 - z_{\scriptscriptstyle i})}
f_{\scriptscriptstyle i}^{\scriptscriptstyle +}(^{{\scriptscriptstyle \sigma_{\scriptscriptstyle i}}}{\bf{\mathrm{z}}}) \, + 
\frac{1}{\sqrt{q} \, z_{\scriptscriptstyle i}} 
f_{\scriptscriptstyle i}^{\scriptscriptstyle -}(^{{\scriptscriptstyle \sigma_{\scriptscriptstyle i}}}{\bf{\mathrm{z}}}). 
\end{equation*} 
In \cite[Lemma~3.2]{CG1} it has been verified that this action extends to a $W$-action on 
$\mathbb{C}({\bf{\mathrm{z}}}).$   

Using this Weyl group action, one can construct a $W$-invariant rational function $f \in \mathbb{C}({\bf{\mathrm{z}}})$ such that 
$
f(0, \ldots, 0; q) = 1, 
$ 
and satisfying the following limiting condition: 
\begin{equation} \label{eq: limiting condition}
\text{for each $i = 1, \ldots r,$ the function 
$
\left. (1 - z_{\scriptscriptstyle i}) \cdot f({\bf{\mathrm{z}}}; q)
\right\vert_{\text{$z_{\! \scriptscriptstyle j}  = 0$ for all $j$ adjacent to $i$}}\,
$ 
is independent of $z_{\scriptscriptstyle i}.$}
\end{equation} 
The rational function satisfying these conditions is unique. When the root system is $D_{4},$ the uniqueness of this 
function follows easily from \cite[Theorem~3.7]{BD} by a simple specialization, and in the general case, it follows 
similarly from the results in \cite{White1} and \cite{White2}. To construct this function, let $\Delta({\bf{\mathrm{z}}})$ 
be defined by   
\begin{equation*}
\Delta({\bf{\mathrm{z}}}) \;\, = \prod_{\alpha \, \in \, \Phi^{+}} \big(1 - \, q^{d(\alpha)} {\bf{\mathrm{z}}}^{2 \alpha} \big)
\end{equation*} 
and, for $\sigma \in  W,$ put
\begin{equation*}
j(\sigma, \, {\bf{\mathrm{z}}}) = \frac{\Delta({\bf{\mathrm{z}}})}{\Delta(^{{\scriptscriptstyle \sigma}}{\bf{\mathrm{z}}})}.
\end{equation*} 
Note that  
$
j(\sigma_{\scriptscriptstyle i}, \, {\bf{\mathrm{z}}}) = - \, q z_{\scriptscriptstyle i}^{2}
$ 
for each simple reflection $\sigma_{\scriptscriptstyle i},$ and that this function satisfies the one-cocycle relation 
\begin{equation*}
j(\sigma' \sigma, \, {\bf{\mathrm{z}}}) = j(\sigma'\!,\, ^{{\scriptscriptstyle \sigma}}{\bf{\mathrm{z}}})j(\sigma, \, {\bf{\mathrm{z}}})
\qquad \text{(for all $\sigma, \sigma' \in W$).}
\end{equation*} 
Finally, we define the rational function $f({\bf{\mathrm{z}}}; q)$ by 
\begin{equation} \label{eq: CG-average}
f({\bf{\mathrm{z}}}; q) = \Delta({\bf{\mathrm{z}}})^{\scriptscriptstyle -1} 
\cdot \sum_{\sigma\, \in \, W} j(\sigma, \, {\bf{\mathrm{z}}})(1\vert \, \sigma)({\bf{\mathrm{z}}}).
\end{equation} 
The fact that this function satisfies the required conditions is established in \cite[Theorem~3.4]{CG1}. 

The rational function \eqref{eq: CG-average} corresponding to the root system $D_{4}$ is further discussed in 
Appendix \ref{Appendix B}, \!and will be used in the next section to construct a family of multiple 
{D}irichlet series over rational function fields satisfying the usual analytic properties.

\section{Multiple {D}irichlet series}  \label{Three expressions of Multiple {D}irichlet series} 
Consider the rational function $f$ defined in Appendix \ref{Appendix B}, Eqs. \!\eqref{rat-funct-Z-D4}, 
\eqref{rational-function-D4}, and expand it in a power series     
\begin{equation*} 
f(z_{\scriptscriptstyle 1}^{}\!, \ldots, z_{r}^{}, z_{r + \scriptscriptstyle 1}^{}; q) \;\; =  
\sum_{k_{\scriptscriptstyle 1}\!, \ldots, \, k_{r}\!, \, l \, \ge \, 0}  
\, a(k_{\scriptscriptstyle 1}\!, \ldots,  k_{r}, l; q) \, z_{\scriptscriptstyle 1}^{\scriptscriptstyle k_{\scriptscriptstyle 1}} 
\! \cdots \,  z_{r}^{\scriptscriptstyle k_{r}} z_{r + \scriptscriptstyle 1}^{\scriptscriptstyle l} \qquad \text{(with $r = 3$)}
\footnote{We shall assume throughout that $r = 3.$ However, since most of the functions (and other quantities) involved 
can be defined for other values of $r$ as well, we prefer (in such instances) to denote this value by $r$ -- rather than 
taking it to be $3.$}.
\end{equation*}

We now use the coefficients of $f$ to construct the relevant family of multiple {D}irichlet series.

Let $c\in \mathbb{F}[x]$ be monic and square-free, and fix a factorization 
$c = c_{\scriptscriptstyle 1}c_{\scriptscriptstyle 2}c_{\scriptscriptstyle 3}$ 
(with $c_{\scriptscriptstyle i} \in \mathbb{F}[x]$ monic). \!Choose a  
$
\theta_{\scriptscriptstyle 0} \in \mathbb{F}^{\times} \setminus (\mathbb{F}^{\times})^{\scriptscriptstyle 2},
$ 
and let $a_{\scriptscriptstyle 1}, \, a_{\scriptscriptstyle 2} \in \{1, \theta_{\scriptscriptstyle 0} \}.$ For 
$
\mathbf{s} = (s_{\scriptscriptstyle 1}, \ldots, s_{r + \scriptscriptstyle 1})
$ 
with $\Re(s_{\scriptscriptstyle i})$ sufficiently large, we define the multiple {D}irichlet series 
$
Z^{(c)}(\mathbf{s};  \chi_{a_{\scriptscriptstyle 2} c_{\scriptscriptstyle 2}}, \chi_{a_{\scriptscriptstyle 1} c_{\scriptscriptstyle 1}}) 
$ 
by the absolutely convergent series   
\begin{equation} \label{eq: MDS-vers0}
Z^{(c)}(\mathbf{s};  \chi_{a_{\scriptscriptstyle 2} c_{\scriptscriptstyle 2}}, \chi_{a_{\scriptscriptstyle 1} 
c_{\scriptscriptstyle 1}}) 
\;\; : = \sum_{\substack{m_{\scriptscriptstyle 1}\!, \ldots, \, m_{r}, \, d - \mathrm{monic} \\  
d = d_{\scriptscriptstyle 0}^{} d_{\scriptscriptstyle 1}^{2}, \; d_{\scriptscriptstyle 0}^{} \; \mathrm{sq. \; free} \\ 
(m_{\scriptscriptstyle 1} \cdots \, m_{r} \, d, \, c) = 1}}  \, \frac{\chi_{a_{\scriptscriptstyle 1} c_{\scriptscriptstyle 1} d_{\scriptscriptstyle 0}}\!(\widehat{m}_{\scriptscriptstyle 1}) \, \cdots \, 
\chi_{a_{\scriptscriptstyle 1} c_{\scriptscriptstyle 1}d_{\scriptscriptstyle 0}}\!(\widehat{m}_{r})\,
\chi_{a_{\scriptscriptstyle 2} c_{\scriptscriptstyle 2}}\!(d_{\scriptscriptstyle 0})}
{|m_{\scriptscriptstyle 1}|^{s_{\scriptscriptstyle 1}} \cdots \, |m_{r}|^{s_{r}} |d|^{s_{r + \scriptscriptstyle 1}}} \cdot 
A(m_{\scriptscriptstyle 1}, \ldots, m_{r}, d)
\end{equation} 
where $\widehat{m}_{\scriptscriptstyle i}$ ($i = 1, \ldots, r$) denotes the part of 
$m_{\scriptscriptstyle i}$ coprime to $d_{\scriptscriptstyle 0},$ and the coefficients
$
A(m_{\scriptscriptstyle 1}, m_{\scriptscriptstyle 2}, m_{\scriptscriptstyle 3}, \ldots, m_{r}, d)
$ 
are defined as follows:
\begin{enumerate}[label=(\roman*)]
\item If $p \in \mathbb{F}[x]$ is a monic irreducible, put 
\begin{equation*}  
A\big(p^{\, k_{\scriptscriptstyle 1}}\!, \ldots, p^{\, k_{r}}\!, p^{\, l} \big)
= a\big(k_{\scriptscriptstyle 1}, \ldots, k_{r}, l; q^{\deg \, p}\big)
\end{equation*}   

\item For monic $m_{\scriptscriptstyle 1}, \ldots, \, m_{r}, \, d$ with 
$(m_{\scriptscriptstyle 1} \cdots \, m_{r} \, d, c) = 1,$ 
we have  
\begin{equation*}   
A(m_{\scriptscriptstyle 1}, \ldots, m_{r}, d)
\;\, = \prod_{\substack{p^{\, k_{\scriptscriptstyle i}} \parallel \, m_{\scriptscriptstyle i}\\ p^{\, l} \parallel \, d}} 
A\big(p^{\, k_{\scriptscriptstyle 1}}\!, \ldots, p^{\, k_{r}}\!, p^{\, l} \big)
\end{equation*} 
the product being taken over monic irreducibles $p  \in  \mathbb{F}[x].$
\end{enumerate}

The series \eqref{eq: MDS-vers0} has two alternative expressions allowing us to show that 
$
Z^{(c)}(\mathbf{s};  \chi_{a_{\scriptscriptstyle 2} c_{\scriptscriptstyle 2}}, \chi_{a_{\scriptscriptstyle 1} c_{\scriptscriptstyle 1}}) 
$ 
admits meromorphic continuation and satisfies a group of functional equations. Indeed, for every 
monic polynomial $d = d_{\scriptscriptstyle 0}^{} d_{\scriptscriptstyle 1}^{2}$ coprime to $c,$ we can express 
\begin{equation*} 
\begin{split}
&\sum_{\substack{m_{\scriptscriptstyle 1}\!, \ldots, \, m_{r} - \mathrm{monic}\\ 
(m_{\scriptscriptstyle 1} \cdots \, m_{r}, \, c) = 1}}  
\,\frac{\chi_{a_{\scriptscriptstyle 1} c_{\scriptscriptstyle 1} d_{\scriptscriptstyle 0}}\!(\widehat{m}_{\scriptscriptstyle 1}) \, \cdots \, 
\chi_{a_{\scriptscriptstyle 1} c_{\scriptscriptstyle 1}d_{\scriptscriptstyle 0}}\!(\widehat{m}_{r})}
{|m_{\scriptscriptstyle 1}|^{s_{\scriptscriptstyle 1}} \cdots \, |m_{r}|^{s_{r}}} \cdot 
A(m_{\scriptscriptstyle 1}, \ldots, m_{r}, d) \\ 
& = \prod_{\substack{p^{\, l} \parallel \, d \\ l-\text{odd}}} \, 
\left(\sum_{k_{\scriptscriptstyle 1}\!, \ldots, \, k_{r}\ge 0}
\frac{A\big(p^{\, k_{\scriptscriptstyle 1}}\!, \ldots, p^{\, k_{r}}\!, p^{\, l} \big)}
{|p|^{\, k_{\scriptscriptstyle 1}s_{\scriptscriptstyle 1} + \cdots + k_{r}s_{r}}}\right)
\prod_{\substack{p \nmid c \\ p^{\, l} \parallel \, d \\ l-\text{even}}} \,
\left(\sum_{k_{\scriptscriptstyle 1}\!, \ldots, \, k_{r}\ge 0}
\frac{\chi_{a_{\scriptscriptstyle 1} c_{\scriptscriptstyle 1} d_{\scriptscriptstyle 0}}\!(p)^{\, k_{\scriptscriptstyle 1} + \cdots + k_{r}}
A\big(p^{\, k_{\scriptscriptstyle 1}}\!, \ldots, p^{\, k_{r}}\!, p^{\, l} \big)}
{|p|^{\, k_{\scriptscriptstyle 1}s_{\scriptscriptstyle 1} + \cdots + k_{r}s_{r}}} \right).
\end{split}
\end{equation*} 
If $l = 0,$ we have  
\begin{equation*}  
A\big(p^{\, k_{\scriptscriptstyle 1}}\!, \ldots, p^{\, k_{r}}\!, 1 \big)
= a\big(k_{\scriptscriptstyle 1}, \ldots, k_{r}, 0; q^{\deg \, p}\big) = 1
\end{equation*} 
(see Appendix \ref{Appendix B}, \eqref{initial-conditions-f}) and thus 
\begin{equation*} 
\prod_{p \nmid cd} \,
\left(\sum_{k_{\scriptscriptstyle 1}\!, \ldots, \, k_{r}\ge 0}
\frac{\chi_{a_{\scriptscriptstyle 1} c_{\scriptscriptstyle 1} d_{\scriptscriptstyle 0}}\!(p)^{\, k_{\scriptscriptstyle 1} + \cdots + k_{r}}
A\big(p^{\, k_{\scriptscriptstyle 1}}\!, \ldots, p^{\, k_{r}}\!, 1 \big)}
{|p|^{\, k_{\scriptscriptstyle 1}s_{\scriptscriptstyle 1} + \cdots + k_{r}s_{r}}} \right) \,
= \, \prod_{i = 1}^{r}\, L^{\scriptscriptstyle (c_{\scriptscriptstyle 2}c_{\scriptscriptstyle 3} d_{\scriptscriptstyle 1})}
(s_{\scriptscriptstyle i}, \chi_{a_{\scriptscriptstyle 1} c_{\scriptscriptstyle 1}  d_{\scriptscriptstyle 0}}).
\end{equation*} 
The remaining part ($l \ne 0$) of the two products can be expressed as 
\begin{equation*} 
\begin{split}
&\prod_{\substack{p^{\, l} \parallel \, d \\ l \, \equiv \, 1 \!\!\!\!\!\!\!\pmod 2}} \,
P_{\scriptscriptstyle l}(|p|^{\, - s_{\scriptscriptstyle 1}}\!, \ldots, \, |p|^{\, - s_{r}}; \; q^{\deg \, p}) \\
& \cdot \prod_{\substack{p \, \mid \, d_{\scriptscriptstyle 1} \\ p^{\, l} \parallel \, d \\ l \, \equiv \, 0 \!\!\!\!\!\!\!\pmod 2}} \,
\left(\prod_{i = 1}^{r}\, 
\big(1 \, - \, \chi_{a_{\scriptscriptstyle 1} c_{\scriptscriptstyle 1} d_{\scriptscriptstyle 0}}\!(p)
|p|^{\, - s_{\scriptscriptstyle i}} \big)^{-1}
\cdot \; P_{\scriptscriptstyle l}(\chi_{a_{\scriptscriptstyle 1} c_{\scriptscriptstyle 1} d_{\scriptscriptstyle 0}}\!(p) |p|^{\, - s_{\scriptscriptstyle 1}}\!, \ldots, \, 
\chi_{a_{\scriptscriptstyle 1} c_{\scriptscriptstyle 1} d_{\scriptscriptstyle 0}}\!(p)|p|^{\, - s_{r}}; \; q^{\deg \, p}) \right).
\end{split}
\end{equation*} 
Thus, if we define the {D}irichlet polynomial 
\begin{equation} \label{eq: polyPd}
\begin{split}
P_{\scriptscriptstyle d}(s_{\scriptscriptstyle 1}\!, \ldots, \, s_{r}; 
\chi_{a_{\scriptscriptstyle 1} c_{\scriptscriptstyle 1} d_{\scriptscriptstyle 0}}) \;\; = \!
& \prod_{\substack{p^{\, l} \parallel \, d \\ l \, \equiv \, 1 \!\!\!\!\!\!\!\pmod 2}} \,
P_{\scriptscriptstyle l}(|p|^{\, - s_{\scriptscriptstyle 1}}\!, \ldots, \, |p|^{\, - s_{r}}; \; q^{\deg \, p}) \\
& \cdot \prod_{\substack{p \, \mid \, d_{\scriptscriptstyle 1} \\ p^{\, l} \parallel \, d \\ l \, \equiv \, 0 \!\!\!\!\!\!\!\pmod 2}} \,
P_{\scriptscriptstyle l}(\chi_{a_{\scriptscriptstyle 1} c_{\scriptscriptstyle 1} d_{\scriptscriptstyle 0}}\!(p) 
|p|^{\, - s_{\scriptscriptstyle 1}}\!, \ldots, \, 
\chi_{a_{\scriptscriptstyle 1} c_{\scriptscriptstyle 1} d_{\scriptscriptstyle 0}}\!(p)|p|^{\, - s_{r}}; \; q^{\deg \, p})
\end{split}
\end{equation} 
then we can write   
\begin{equation}  \label{eq: MDS-vers1}
Z^{(c)}(\mathbf{s};  \chi_{a_{\scriptscriptstyle 2} c_{\scriptscriptstyle 2}}, \chi_{a_{\scriptscriptstyle 1} c_{\scriptscriptstyle 1}}) 
\; = \sum_{\substack{(d, \, c) = 1 \\  d = d_{\scriptscriptstyle 0}^{} d_{\scriptscriptstyle 1}^{2}}}  \, 
\frac{\prod_{i = 1}^{r}\, L^{\scriptscriptstyle (c_{\scriptscriptstyle 2}c_{\scriptscriptstyle 3})}
(s_{\scriptscriptstyle i}, \chi_{a_{\scriptscriptstyle 1} c_{\scriptscriptstyle 1}  d_{\scriptscriptstyle 0}}) \, \cdot 
\,\chi_{a_{\scriptscriptstyle 2} c_{\scriptscriptstyle 2}}\!(d_{\scriptscriptstyle 0})\, 
P_{\scriptscriptstyle d}(s_{\scriptscriptstyle 1}\!, \ldots, \, s_{r}; \chi_{a_{\scriptscriptstyle 1} c_{\scriptscriptstyle 1} d_{\scriptscriptstyle 0}})}{|d|^{s_{r + \scriptscriptstyle 1}}}.
\end{equation} 
Now fix monics $m_{\scriptscriptstyle 1}, \ldots, \, m_{r}$ coprime to $c,$ and write 
$
m_{\scriptscriptstyle 1} \cdots \, m_{r} = n_{\scriptscriptstyle 0}^{}  n_{\scriptscriptstyle 1}^{2}  
$ 
with $n_{\scriptscriptstyle 0}^{}$ square-free. \!As
\begin{equation*}  
A\big(p^{\, k_{\scriptscriptstyle 1}}\!, \ldots, p^{\, k_{r}}\!, p^{\, l} \big)
= a\big(k_{\scriptscriptstyle 1}, \ldots, k_{r}, l; q^{\deg \, p}\big) = 0 
\end{equation*} 
if $\sum k_{\scriptscriptstyle i} \equiv l \equiv 1 \!\!\pmod 2$ (see Appendix \ref{Appendix B}, 
\eqref{odd-l-ranked-coeff}), we have 
\begin{equation*} 
\begin{split}
&\sum_{\substack{(d, \, c) = 1\\  
d = d_{\scriptscriptstyle 0}^{} d_{\scriptscriptstyle 1}^{2}, \; d_{\scriptscriptstyle 0}^{} \; \mathrm{sq. \; free}}}  \, \frac{\chi_{a_{\scriptscriptstyle 1} c_{\scriptscriptstyle 1} d_{\scriptscriptstyle 0}}\!(\widehat{m}_{\scriptscriptstyle 1}) \, \cdots \, 
\chi_{a_{\scriptscriptstyle 1} c_{\scriptscriptstyle 1}d_{\scriptscriptstyle 0}}\!(\widehat{m}_{r})\,
\chi_{a_{\scriptscriptstyle 2} c_{\scriptscriptstyle 2}}\!(d_{\scriptscriptstyle 0})}
{|d|^{s_{r + \scriptscriptstyle 1}}} \cdot A(m_{\scriptscriptstyle 1}, \ldots, m_{r}, d)\\
& = \chi_{a_{\scriptscriptstyle 1} c_{\scriptscriptstyle 1}}\!(n_{\scriptscriptstyle 0}^{}) \;\; \cdot
\prod_{\substack{p^{\, k_{\scriptscriptstyle i}} \parallel \, m_{\scriptscriptstyle i}
\\ \sum k_{\scriptscriptstyle i} \, \equiv \, 1 \!\!\!\!\!\!\!\pmod 2}} \, 
\left(\, \sum_{\substack{l = 0 \\  l \, \equiv \, 0 \!\!\!\!\!\!\!\pmod 2}}^{\infty}
\frac{A\big(p^{\, k_{\scriptscriptstyle 1}}\!, \ldots, p^{\, k_{r}}\!, p^{\, l} \big)}
{|p|^{\, l s_{r + \scriptscriptstyle 1}}}\right)\\
& \cdot \prod_{\substack{p \nmid c \\ p^{\, k_{\scriptscriptstyle i}} \parallel \, m_{\scriptscriptstyle i}
\\ \sum k_{\scriptscriptstyle i} \, \equiv \, 0 \!\!\!\!\!\!\pmod 2}} \,
\left(\, \sum_{\substack{l = 1 \\  l \, \equiv \, 1 \!\!\!\!\!\!\!\pmod 2}}^{\infty}
\frac{\chi_{a_{\scriptscriptstyle 2} c_{\scriptscriptstyle 2} n_{\scriptscriptstyle 0}}\!(p)^{l}
A\big(p^{\, k_{\scriptscriptstyle 1}}\!, \ldots, p^{\, k_{r}}\!, p^{\, l} \big)}{|p|^{\, l s_{r + \scriptscriptstyle 1}}} \, + \,   
\frac{A\big(p^{\, k_{\scriptscriptstyle 1}}\!, \ldots, p^{\, k_{r}}\!, p^{\, l - 1} \big)}{|p|^{\, (l - 1) s_{r + \scriptscriptstyle 1}}} \right).
\end{split}
\end{equation*} 
Since 
\begin{equation*} 
A(1\!, \ldots, 1\!, p^{\, l}) = a(0, \ldots, 0, l; q^{\deg \, p}) = 1
\end{equation*} 
(see Appendix \ref{Appendix B}, \eqref{initial-conditions-f}) we can again compute the part corresponding to $k_{\scriptscriptstyle 1} = \cdots = k_{r} = 0$ 
as 
\begin{equation*} 
\prod_{p \nmid c n_{\scriptscriptstyle 0}n_{\scriptscriptstyle 1}} \,
\left(\, \sum_{l = 0}^{\infty} \, 
\frac{\chi_{a_{\scriptscriptstyle 2} c_{\scriptscriptstyle 2} n_{\scriptscriptstyle 0}}\!(p)^{l}}{|p|^{\, l s_{r + \scriptscriptstyle 1}}}\right) \; = \prod_{p \nmid  c_{\scriptscriptstyle 1}c_{\scriptscriptstyle 3}n_{\scriptscriptstyle 1}} \,
\big(1 \, - \, \chi_{a_{\scriptscriptstyle 2} c_{\scriptscriptstyle 2} n_{\scriptscriptstyle 0}}\!(p)
|p|^{\, - s_{r + \scriptscriptstyle 1}} \big)^{-1}
= \, L^{\scriptscriptstyle (c_{\scriptscriptstyle 1} c_{\scriptscriptstyle 3} n_{\scriptscriptstyle 1})}
(s_{r + \scriptscriptstyle 1}, \chi_{a_{\scriptscriptstyle 2} c_{\scriptscriptstyle 2} n_{\scriptscriptstyle 0}}).
\end{equation*} 
The remaining part of the expression is 
\begin{equation*} 
\chi_{a_{\scriptscriptstyle 1}  c_{\scriptscriptstyle 1}}\!(n_{\scriptscriptstyle 0}^{}) \;\; \cdot
\prod_{\substack{p^{\, k_{\scriptscriptstyle i}} \parallel \, m_{\scriptscriptstyle i}
\\ |\underline{k}| \, \equiv \, 1 \!\!\!\!\!\!\!\pmod 2}} \, 
Q_{\underline{k}}(|p|^{\, - s_{r + \scriptscriptstyle 1}}; \, q^{\deg \, p}) \; \cdot 
\prod_{\substack{p \, \mid \, n_{\scriptscriptstyle 1} \\ p^{\, k_{\scriptscriptstyle i}} \parallel \, m_{\scriptscriptstyle i} 
\\ |\underline{k}| \, \equiv \, 0 \!\!\!\!\!\!\pmod 2}} \, 
\big(1 \, - \, \chi_{a_{\scriptscriptstyle 2} c_{\scriptscriptstyle 2} n_{\scriptscriptstyle 0}}\!(p)
|p|^{\, - s_{r + \scriptscriptstyle 1}} \big)^{-1}
Q_{\underline{k}}(\chi_{a_{\scriptscriptstyle 2} c_{\scriptscriptstyle 2} n_{\scriptscriptstyle 0}}\!(p)
|p|^{\, - s_{r + \scriptscriptstyle 1}}; \, q^{\deg \, p}). 
\end{equation*} 

As before, for 
$
\underline{m} = (m_{\scriptscriptstyle 1}, \ldots, \, m_{r}),
$ 
we define the {D}irichlet polynomial 
\begin{equation} \label{eq: polyQm}
Q_{\underline{m}}(s_{r + \scriptscriptstyle 1}; 
\chi_{a_{\scriptscriptstyle 2}  c_{\scriptscriptstyle 2} n_{\scriptscriptstyle 0}}) \;\; = \!
\prod_{\substack{p^{\, k_{\scriptscriptstyle i}} \parallel \, m_{\scriptscriptstyle i}
\\ |\underline{k}| \, \equiv \, 1 \!\!\!\!\!\!\!\pmod 2}} \, 
Q_{\underline{k}}(|p|^{\, - s_{r + \scriptscriptstyle 1}}; \, q^{\deg \, p}) \; \cdot
\prod_{\substack{p \, \mid \, n_{\scriptscriptstyle 1} \\ p^{\, k_{\scriptscriptstyle i}} \parallel \, m_{\scriptscriptstyle i} 
\\ |\underline{k}| \, \equiv \, 0 \!\!\!\!\!\!\pmod 2}} \, 
Q_{\underline{k}}(\chi_{a_{\scriptscriptstyle 2} c_{\scriptscriptstyle 2} n_{\scriptscriptstyle 0}}\!(p)
|p|^{\, - s_{r + \scriptscriptstyle 1}}; \, q^{\deg \, p}).
\end{equation}
Accordingly, we can also write
\begin{equation} \label{eq: MDS-vers2}
Z^{(c)}(\mathbf{s};  \chi_{a_{\scriptscriptstyle 2} c_{\scriptscriptstyle 2}}, \chi_{a_{\scriptscriptstyle 1} c_{\scriptscriptstyle 1}}) 
\;\; = \sum_{\substack{(m_{\scriptscriptstyle 1} \cdots \, m_{r}, \, c) = 1 \\ 
m_{\scriptscriptstyle 1} \cdots \, m_{r} \, = \, n_{\scriptscriptstyle 0}^{}  n_{\scriptscriptstyle 1}^{2}}}  \, 
\frac{L^{\scriptscriptstyle (c_{\scriptscriptstyle 1} c_{\scriptscriptstyle 3})}
(s_{r + \scriptscriptstyle 1}, \chi_{a_{\scriptscriptstyle 2} c_{\scriptscriptstyle 2} n_{\scriptscriptstyle 0}})\, 
\chi_{a_{\scriptscriptstyle 1}  c_{\scriptscriptstyle 1}}\!(n_{\scriptscriptstyle 0}^{})\,
Q_{\underline{m}}(s_{r + \scriptscriptstyle 1}; \chi_{a_{\scriptscriptstyle 2} c_{\scriptscriptstyle 2} n_{\scriptscriptstyle 0}})}
{|m_{\scriptscriptstyle 1}|^{s_{\scriptscriptstyle 1}} \cdots \, |m_{r}|^{s_{r}}}.
\end{equation} 
The expressions \eqref{eq: MDS-vers1} and \eqref{eq: MDS-vers2} will be used in the next two sections to establish the analytic properties of 
$Z^{(c)}(\mathbf{s};  \chi_{a_{\scriptscriptstyle 2} c_{\scriptscriptstyle 2}}, \chi_{a_{\scriptscriptstyle 1} c_{\scriptscriptstyle 1}})$ and compute the 
residues at some of its poles.

\subsection{Functional equations and analytic continuation of multiple {D}irichlet series}
The polynomials 
$
P_{\scriptscriptstyle d}(s_{\scriptscriptstyle 1}\!, \ldots, \, s_{r}; 
\chi_{a_{\scriptscriptstyle 1} c_{\scriptscriptstyle 1} d_{\scriptscriptstyle 0}})
$ 
are symmetric in $s_{\scriptscriptstyle 1}, \ldots, \, s_{r},$ and by \eqref{eq: polynomials-P-Q-func-eq} we have
\begin{equation} \label{eq: polynomials-P-func-eq-in-s}
P_{\scriptscriptstyle d}(s_{\scriptscriptstyle 1}, \ldots, \, s_{r}; 
\chi_{a_{\scriptscriptstyle 1} c_{\scriptscriptstyle 1} d_{\scriptscriptstyle 0}}) 
= |d_{\scriptscriptstyle 1}|^{1 - 2 s_{\scriptscriptstyle 1}}
P_{\scriptscriptstyle d}(1 - s_{\scriptscriptstyle 1}, \ldots, \, s_{r}; 
\chi_{a_{\scriptscriptstyle 1} c_{\scriptscriptstyle 1} d_{\scriptscriptstyle 0}}). 
\end{equation} 
The polynomials 
$
Q_{\underline{m}}(s_{r + \scriptscriptstyle 1}; 
\chi_{a_{\scriptscriptstyle 2}  c_{\scriptscriptstyle 2} n_{\scriptscriptstyle 0}})
$ 
satisfy the functional equation    
\begin{equation} \label{eq: polynomials-Q-func-eq-in-s}
Q_{\underline{m}}(s_{r + \scriptscriptstyle 1}; 
\chi_{a_{\scriptscriptstyle 2}  c_{\scriptscriptstyle 2} n_{\scriptscriptstyle 0}}) = 
|n_{\scriptscriptstyle 1}|^{1 - 2 s_{r + \scriptscriptstyle 1}}
Q_{\underline{m}}(1 - s_{r + \scriptscriptstyle 1}; 
\chi_{a_{\scriptscriptstyle 2}  c_{\scriptscriptstyle 2} n_{\scriptscriptstyle 0}}) 
\end{equation} 
where, for 
$
\underline{m} = (m_{\scriptscriptstyle 1}, \ldots, \, m_{r}),
$ 
we write 
$
m_{\scriptscriptstyle 1} \cdots \, m_{r} = n_{\scriptscriptstyle 0}^{}  n_{\scriptscriptstyle 1}^{2}  
$ 
with $n_{\scriptscriptstyle 0}^{}$ square-free. 

We now apply \eqref{eq: funct-eq-L}, 
\eqref{eq: polynomials-P-func-eq-in-s} and \eqref{eq: polynomials-Q-func-eq-in-s} to describe the 
functional equations of  
$
Z^{(c)}(\mathbf{s};  \chi_{a_{\scriptscriptstyle 2} c_{\scriptscriptstyle 2}}, \chi_{a_{\scriptscriptstyle 1} 
c_{\scriptscriptstyle 1}}).$ 
We shall follow here \cite{DGH}; see also \cite{CG1}.

First assume that $\deg c_{\scriptscriptstyle 2}$ is even. \!We split the sum in \eqref{eq: MDS-vers2} 
into two parts according as $\deg n_{\scriptscriptstyle 0}^{}$ is even or odd. \!By applying 
\eqref{eq: polynomials-P-func-eq-in-s}, \eqref{eq: polynomials-Q-func-eq-in-s}, and \eqref{eq: funct-eq-L} 
in the form 
\begin{equation*}
L^{\scriptscriptstyle (c_{\scriptscriptstyle 1} c_{\scriptscriptstyle 3})}
(s_{r + \scriptscriptstyle 1}, \chi_{a_{\scriptscriptstyle 2} c_{\scriptscriptstyle 2} n_{\scriptscriptstyle 0}}) = 
\gamma_{\scriptscriptstyle q}(s_{r + \scriptscriptstyle 1},\,  
a_{\scriptscriptstyle 2} c_{\scriptscriptstyle 2} n_{\scriptscriptstyle 0}) \, 
|c_{\scriptscriptstyle 2} n_{\scriptscriptstyle 0}|^{\frac{1}{2} - s_{r + \scriptscriptstyle 1}} 
L^{\scriptscriptstyle (c_{\scriptscriptstyle 1} c_{\scriptscriptstyle 3})}
(1 - s_{r + \scriptscriptstyle 1}, \chi_{a_{\scriptscriptstyle 2} c_{\scriptscriptstyle 2} n_{\scriptscriptstyle 0}})\, 
\frac{L_{{\scriptscriptstyle c_{\scriptscriptstyle 1} c_{\scriptscriptstyle 3}}}\!(1 - s_{r + \scriptscriptstyle 1}, \chi_{a_{\scriptscriptstyle 2} c_{\scriptscriptstyle 2} n_{\scriptscriptstyle 0}})}
{L_{{\scriptscriptstyle c_{\scriptscriptstyle 1} c_{\scriptscriptstyle 3}}}\!(s_{r + \scriptscriptstyle 1}, 
\chi_{a_{\scriptscriptstyle 2} c_{\scriptscriptstyle 2} n_{\scriptscriptstyle 0}})}
\end{equation*} 
we find that 
\begin{equation*}
\begin{split}
&Z^{(c)}(\mathbf{s};  \chi_{a_{\scriptscriptstyle 2} c_{\scriptscriptstyle 2}}, 
\chi_{a_{\scriptscriptstyle 1} c_{\scriptscriptstyle 1}})\\ 
& = \, \gamma_{\scriptscriptstyle q}^{+}(s_{r + \scriptscriptstyle 1}; \, a_{\scriptscriptstyle 2})\, 
|c_{\scriptscriptstyle 2}|^{\frac{1}{2} - s_{r + \scriptscriptstyle 1}} 
\!\sum_{\substack{(m_{\scriptscriptstyle 1} \cdots \, m_{r}, \, c) = 1 \\ 
m_{\scriptscriptstyle 1} \cdots \, m_{r} \, = \, n_{\scriptscriptstyle 0}^{}  n_{\scriptscriptstyle 1}^{2} \\ 
\deg \, n_{\scriptscriptstyle 0}^{} - \mathrm{even}}}  \, 
\frac{L^{\scriptscriptstyle (c_{\scriptscriptstyle 1} c_{\scriptscriptstyle 3})}
(1 - s_{r + \scriptscriptstyle 1}, \chi_{a_{\scriptscriptstyle 2} c_{\scriptscriptstyle 2} n_{\scriptscriptstyle 0}})
\frac{L_{{\scriptscriptstyle c_{\scriptscriptstyle 1} c_{\scriptscriptstyle 3}}}\!(1 - s_{r + \scriptscriptstyle 1}, \, \chi_{a_{\scriptscriptstyle 2} c_{\scriptscriptstyle 2} n_{\scriptscriptstyle 0}})}
{L_{{\scriptscriptstyle c_{\scriptscriptstyle 1} c_{\scriptscriptstyle 3}}}\!(s_{r + \scriptscriptstyle 1}, \,  
\chi_{a_{\scriptscriptstyle 2} c_{\scriptscriptstyle 2} n_{\scriptscriptstyle 0}})}\, 
\chi_{c_{\scriptscriptstyle 1}}\!(n_{\scriptscriptstyle 0}^{})\,
Q_{\underline{m}}(1 - s_{r + \scriptscriptstyle 1}; \chi_{a_{\scriptscriptstyle 2} c_{\scriptscriptstyle 2} n_{\scriptscriptstyle 0}})}
{|m_{\scriptscriptstyle 1}|^{s_{\scriptscriptstyle 1} + s_{r + \scriptscriptstyle 1} - \frac{1}{2}} \cdots \, 
|m_{r}|^{s_{r} + s_{r + \scriptscriptstyle 1} - \frac{1}{2}}} \\ 
& + \, \gamma_{\scriptscriptstyle q}^{-}(s_{r + \scriptscriptstyle 1})\, 
|c_{\scriptscriptstyle 2}|^{\frac{1}{2} - s_{r + \scriptscriptstyle 1}} 
\!\sum_{\substack{(m_{\scriptscriptstyle 1} \cdots \, m_{r}, \, c) = 1 \\ 
m_{\scriptscriptstyle 1} \cdots \, m_{r} \, = \, n_{\scriptscriptstyle 0}^{}  n_{\scriptscriptstyle 1}^{2} \\ 
\deg \, n_{\scriptscriptstyle 0}^{} - \mathrm{odd}}}  \, 
\frac{L^{\scriptscriptstyle (c_{\scriptscriptstyle 1} c_{\scriptscriptstyle 3})}
(1 - s_{r + \scriptscriptstyle 1}, \chi_{a_{\scriptscriptstyle 2} c_{\scriptscriptstyle 2} n_{\scriptscriptstyle 0}})
\frac{L_{{\scriptscriptstyle c_{\scriptscriptstyle 1} c_{\scriptscriptstyle 3}}}\!(1 - s_{r + \scriptscriptstyle 1}, \,  \chi_{a_{\scriptscriptstyle 2} c_{\scriptscriptstyle 2} n_{\scriptscriptstyle 0}})}
{L_{{\scriptscriptstyle c_{\scriptscriptstyle 1} c_{\scriptscriptstyle 3}}}\!(s_{r + \scriptscriptstyle 1}, \, 
\chi_{a_{\scriptscriptstyle 2} c_{\scriptscriptstyle 2} n_{\scriptscriptstyle 0}})}\, 
\chi_{a_{\scriptscriptstyle 1}  c_{\scriptscriptstyle 1}}\!(n_{\scriptscriptstyle 0}^{})\,
Q_{\underline{m}}(1 - s_{r + \scriptscriptstyle 1}; \chi_{a_{\scriptscriptstyle 2} c_{\scriptscriptstyle 2} n_{\scriptscriptstyle 0}})}
{|m_{\scriptscriptstyle 1}|^{s_{\scriptscriptstyle 1} + s_{r + \scriptscriptstyle 1} - \frac{1}{2}} \cdots \, 
|m_{r}|^{s_{r} + s_{r + \scriptscriptstyle 1} - \frac{1}{2}}}
\end{split}
\end{equation*} 
where we put
\begin{equation*} 
\gamma_{\scriptscriptstyle q}^{+}(s_{r + \scriptscriptstyle 1}; \, a_{\scriptscriptstyle 2}) : = 
\frac{q^{2 s_{r + \scriptscriptstyle 1} - 1}(1 \, - \, \mathrm{sgn}(a_{\scriptscriptstyle 2}) q^{ - s_{r + \scriptscriptstyle 1}})}
{1 \, - \, \mathrm{sgn}(a_{\scriptscriptstyle 2}) q^{s_{r + \scriptscriptstyle 1} - 1}}
\;\;\; \mathrm{and} \;\;\; 
\gamma_{\scriptscriptstyle q}^{-}(s_{r + \scriptscriptstyle 1}) : = 
q^{s_{r + \scriptscriptstyle 1} - \frac{1}{2}}.
\end{equation*} 
Notice that in the first sum we have also used the fact that 
$
\chi_{a_{\scriptscriptstyle 1}  c_{\scriptscriptstyle 1}}\!(n_{\scriptscriptstyle 0}^{})
= \chi_{c_{\scriptscriptstyle 1}}\!(n_{\scriptscriptstyle 0}^{})
$ 
when $n_{\scriptscriptstyle 0}^{}$ has even degree.

To simplify this expression, multiply 
$ 
Z^{(c)}(\mathbf{s};  \chi_{a_{\scriptscriptstyle 2} c_{\scriptscriptstyle 2}}, \chi_{a_{\scriptscriptstyle 1} c_{\scriptscriptstyle 1}}) 
$
by the product
$  
\prod_{p \, \mid \, c_{\scriptscriptstyle 1}  c_{\scriptscriptstyle 3}} \big(\! 1 \, - \, |p|^{2 s_{r + \scriptscriptstyle 1} - 2}\big).
$ 
If we define $U_{ \scriptscriptstyle m}(s_{r + \scriptscriptstyle 1}) = 1$ for $m \in \mathbb{F}^{\times}\!,$ and 
\begin{equation*}
U_{ \scriptscriptstyle m}(s_{r + \scriptscriptstyle 1}) = \prod_{p \, \mid \, m} 
\frac{|p|^{s_{r + \scriptscriptstyle 1} - 1}\big(1 - |p|^{1 - 2 s_{r + \scriptscriptstyle 1}} \big)}
{1 - |p|^{-  \scriptscriptstyle 1}}
\end{equation*} 
for $m$ square-free of positive degree (the product being over the monic irreducible divisors of $m$), 
we can express: 
\begin{equation*}
\begin{split}
& \frac{L_{c_{\scriptscriptstyle 1} c_{\scriptscriptstyle 3}}
(1 - s_{r + \scriptscriptstyle 1}, \chi_{a_{\scriptscriptstyle 2} c_{\scriptscriptstyle 2} n_{\scriptscriptstyle 0}})}
{L_{c_{\scriptscriptstyle 1} c_{\scriptscriptstyle 3}}
(s_{r + \scriptscriptstyle 1}, \chi_{a_{\scriptscriptstyle 2} c_{\scriptscriptstyle 2} n_{\scriptscriptstyle 0}})} \, \cdot
\prod_{p \, \mid \, c_{\scriptscriptstyle 1} c_{\scriptscriptstyle 3}} \big(\! 1 \, - \, |p|^{2 s_{r + \scriptscriptstyle 1} - 2}\big) 
\, = \! \prod_{p \, \mid \, c_{\scriptscriptstyle 1} c_{\scriptscriptstyle 3}} \!
\big(\! 1 \, + \, 
\chi_{a_{\scriptscriptstyle 2} c_{\scriptscriptstyle 2} n_{\scriptscriptstyle 0}}\!(p)|p|^{s_{r + \scriptscriptstyle 1} - 1} \big)  
\big(\! 1 \, - \, \chi_{a_{\scriptscriptstyle 2} c_{\scriptscriptstyle 2} n_{\scriptscriptstyle 0}}\!(p)
|p|^{- s_{r + \scriptscriptstyle 1}} \big)\\
& = \prod_{p \, \mid \, c_{\scriptscriptstyle 1} c_{\scriptscriptstyle 3}} \! \big(1 \, - \, |p|^{-  \scriptscriptstyle 1} \big) 
\left(1 \, + \, \chi_{a_{\scriptscriptstyle 2} c_{\scriptscriptstyle 2} n_{\scriptscriptstyle 0}}\!(p)\, 
U_{ \scriptscriptstyle p}(s_{r + \scriptscriptstyle 1}) \right) 
= \frac{\varphi(c_{\scriptscriptstyle 1} c_{\scriptscriptstyle 3})}{|c_{\scriptscriptstyle 1} c_{\scriptscriptstyle 3}|}
\sum_{m \, \mid \, c_{\scriptscriptstyle 1} c_{\scriptscriptstyle 3}} \!
\chi_{a_{\scriptscriptstyle 2} c_{\scriptscriptstyle 2} n_{\scriptscriptstyle 0}}\!(m)\, 
U_{ \scriptscriptstyle m}(s_{r + \scriptscriptstyle 1}).
\end{split}
\end{equation*} 
Here $\varphi$ is Euler's totient function over $\mathbb{F}[x].$ Letting 
\begin{equation*} 
^{{\scriptscriptstyle \sigma_{\scriptscriptstyle i}}}\mathbf{s} :
= \big(s_{\scriptscriptstyle 1}, \ldots, 1 - s_{\scriptscriptstyle i}, \ldots, s_{r}, s_{i} + s_{r + \scriptscriptstyle 1} - \tfrac{1}{2}\big)\;\; 
\; \text{for $i\le r,$ and} \;\;\,  
^{{\scriptscriptstyle \sigma_{r + \scriptscriptstyle 1}}}\mathbf{s} 
= \big(s_{\scriptscriptstyle 1} + s_{r + \scriptscriptstyle 1} - \tfrac{1}{2}, \ldots, s_{r} + s_{r + \scriptscriptstyle 1} - \tfrac{1}{2}, 
1 - s_{r + \scriptscriptstyle 1}\big) 
\end{equation*} 
it follows that 
\begin{equation*}
\begin{split}
&\prod_{p \, \mid \, c_{\scriptscriptstyle 1}  c_{\scriptscriptstyle 3}} \big(\! 1 \, - \, |p|^{2 s_{r + \scriptscriptstyle 1} - 2}\big)
\, \cdot \, Z^{(c)}(\mathbf{s};  \chi_{a_{\scriptscriptstyle 2} c_{\scriptscriptstyle 2}}, 
\chi_{a_{\scriptscriptstyle 1} c_{\scriptscriptstyle 1}})\\ 
& = \, \tfrac{1}{2}\, \gamma_{\scriptscriptstyle q}^{+}(s_{r + \scriptscriptstyle 1}; \, a_{\scriptscriptstyle 2})\, 
|c_{\scriptscriptstyle 2}|^{\frac{1}{2} - s_{r + \scriptscriptstyle 1}} 
\frac{\varphi(c_{\scriptscriptstyle 1} c_{\scriptscriptstyle 3})}{|c_{\scriptscriptstyle 1} c_{\scriptscriptstyle 3}|}
\!\!\!\sum_{\substack{m \, \mid \, c_{\scriptscriptstyle 1} c_{\scriptscriptstyle 3} \\ (c_{\scriptscriptstyle 1}\!, \, m) = e}}
\chi_{a_{\scriptscriptstyle 2} c_{\scriptscriptstyle 2}}\!(m) \, 
U_{ \scriptscriptstyle m}(s_{r + \scriptscriptstyle 1}) 
\left\{Z^{(c)}(^{{\scriptscriptstyle \sigma_{r + \scriptscriptstyle 1}}}\mathbf{s}; 
\chi_{a_{\scriptscriptstyle 2} c_{\scriptscriptstyle 2}}, 
\chi_{c_{\scriptscriptstyle 1} m \slash e^{\scriptscriptstyle 2}})  \, + \, 
Z^{(c)}(^{{\scriptscriptstyle \sigma_{r + \scriptscriptstyle 1}}}\mathbf{s}; 
\chi_{a_{\scriptscriptstyle 2} c_{\scriptscriptstyle 2}}, 
\chi_{{\scriptscriptstyle \theta_{\scriptscriptstyle 0}}c_{\scriptscriptstyle 1} m \slash e^{\scriptscriptstyle 2}}) \right\}\\ 
& + \, \tfrac{1}{2}\, \gamma_{\scriptscriptstyle q}^{-}(s_{r + \scriptscriptstyle 1})\, 
|c_{\scriptscriptstyle 2}|^{\frac{1}{2} - s_{r + \scriptscriptstyle 1}} 
\frac{\varphi(c_{\scriptscriptstyle 1} c_{\scriptscriptstyle 3})}{|c_{\scriptscriptstyle 1} c_{\scriptscriptstyle 3}|}
\!\!\!\sum_{\substack{m \, \mid \, c_{\scriptscriptstyle 1} c_{\scriptscriptstyle 3} \\ (c_{\scriptscriptstyle 1}\!, \, m) = e}} 
\chi_{a_{\scriptscriptstyle 2} c_{\scriptscriptstyle 2}}\!(m) \, 
U_{ \scriptscriptstyle m}(s_{r + \scriptscriptstyle 1}) 
\left\{Z^{(c)}(^{{\scriptscriptstyle \sigma_{r + \scriptscriptstyle 1}}}\mathbf{s}; 
\chi_{a_{\scriptscriptstyle 2} c_{\scriptscriptstyle 2}}, 
\chi_{a_{\scriptscriptstyle 1} c_{\scriptscriptstyle 1} m\slash e^{\scriptscriptstyle 2}})  \, - \, 
Z^{(c)}(^{{\scriptscriptstyle \sigma_{r + \scriptscriptstyle 1}}}\mathbf{s}; 
\chi_{a_{\scriptscriptstyle 2} c_{\scriptscriptstyle 2}}, 
\chi_{{\scriptscriptstyle \theta_{\scriptscriptstyle 0}} 
a_{\scriptscriptstyle 1} c_{\scriptscriptstyle 1} m\slash e^{\scriptscriptstyle 2}}) \right\}\!. 
\end{split}
\end{equation*} 
In this formula, the two sums are over all divisors $m$ of $c_{\scriptscriptstyle 1} c_{\scriptscriptstyle 3},$ and $e$ 
denotes the gcd of $m$ and $c_{\scriptscriptstyle 1}.$

When $\deg c_{\scriptscriptstyle 2}$ is odd, the functional equation is obtained by just switching the factors     
$\gamma_{\scriptscriptstyle q}^{+}(s_{r + \scriptscriptstyle 1}; \, a_{\scriptscriptstyle 2})$ and 
$\gamma_{\scriptscriptstyle q}^{-}(s_{r + \scriptscriptstyle 1}).$ 

We can combine the two cases, and write this functional equation as  
\begin{equation} \label{eq: functional-eq-Z-sigma_r+1}
\begin{split}
& Z^{(c)}(\mathbf{s};  \chi_{a_{\scriptscriptstyle 2} c_{\scriptscriptstyle 2}}, 
\chi_{a_{\scriptscriptstyle 1} c_{\scriptscriptstyle 1}})\,  
= \,  \tfrac{1}{2}\,  |c_{\scriptscriptstyle 2}|^{\frac{1}{2} - s_{r + \scriptscriptstyle 1}} 
\frac{\varphi(c_{\scriptscriptstyle 1} c_{\scriptscriptstyle 3})}{|c_{\scriptscriptstyle 1} c_{\scriptscriptstyle 3}|}
\prod_{p \, \mid \, c_{\scriptscriptstyle 1}  c_{\scriptscriptstyle 3}} \big(\! 1 \, - \, 
|p|^{2 s_{r + \scriptscriptstyle 1} - 2}\big)^{\scriptscriptstyle -1}\\
&\cdot \sum_{\vartheta \in \{1, \, \theta_{\scriptscriptstyle 0} \}}                
\chi_{a_{\scriptscriptstyle 1} \!{\scriptscriptstyle \vartheta}}(c_{\scriptscriptstyle 2})
\!\left\{\gamma_{\scriptscriptstyle q}^{+}(s_{r + \scriptscriptstyle 1}; \, a_{\scriptscriptstyle 2}) \, + \, 
\mathrm{sgn}(a_{\scriptscriptstyle 1} \!\vartheta)\, \gamma_{\scriptscriptstyle q}^{-}(s_{r + \scriptscriptstyle 1})\right\} 
\!\!\!\!\sum_{\substack{m \, \mid \, c_{\scriptscriptstyle 1} c_{\scriptscriptstyle 3} \\ (c_{\scriptscriptstyle 1}\!, \, m) = e}}
\chi_{a_{\scriptscriptstyle 2} c_{\scriptscriptstyle 2}}\!(m) \, 
U_{ \scriptscriptstyle m}(s_{r + \scriptscriptstyle 1})\,  
Z^{(c)}(^{{\scriptscriptstyle \sigma_{r + \scriptscriptstyle 1}}}\mathbf{s}; 
\chi_{a_{\scriptscriptstyle 2} c_{\scriptscriptstyle 2}}, 
\chi_{{\scriptscriptstyle \vartheta} c_{\scriptscriptstyle 1} m \slash e^{\scriptscriptstyle 2}}).
\end{split}
\end{equation}

Similarly, we can use the expression \eqref{eq: MDS-vers1} of   
$
Z^{(c)}(\mathbf{s};  \chi_{a_{\scriptscriptstyle 2} c_{\scriptscriptstyle 2}}, \chi_{a_{\scriptscriptstyle 1} c_{\scriptscriptstyle 1}}), 
$ 
and the functional equations \eqref{eq: polynomials-P-func-eq-in-s} and   
\begin{equation*}
L^{\scriptscriptstyle (c_{\scriptscriptstyle 2} c_{\scriptscriptstyle 3})}
(s_{\scriptscriptstyle 1}, \chi_{a_{\scriptscriptstyle 1} c_{\scriptscriptstyle 1} d_{\scriptscriptstyle 0}}) = 
\gamma_{\scriptscriptstyle q}(s_{\scriptscriptstyle 1},\,  
a_{\scriptscriptstyle 1} c_{\scriptscriptstyle 1} d_{\scriptscriptstyle 0}) \, 
|c_{\scriptscriptstyle 1} d_{\scriptscriptstyle 0}|^{\frac{1}{2} - s_{\scriptscriptstyle 1}} 
L^{\scriptscriptstyle (c_{\scriptscriptstyle 2} c_{\scriptscriptstyle 3})}
(1 - s_{\scriptscriptstyle 1}, \chi_{a_{\scriptscriptstyle 1} c_{\scriptscriptstyle 1} d_{\scriptscriptstyle 0}})\, 
\frac{L_{{\scriptscriptstyle c_{\scriptscriptstyle 2} c_{\scriptscriptstyle 3}}}\!(1 - s_{\scriptscriptstyle 1}, 
\chi_{a_{\scriptscriptstyle 1} c_{\scriptscriptstyle 1} d_{\scriptscriptstyle 0}})}
{L_{{\scriptscriptstyle c_{\scriptscriptstyle 2} c_{\scriptscriptstyle 3}}}\!(s_{\scriptscriptstyle 1}, 
\chi_{a_{\scriptscriptstyle 1} c_{\scriptscriptstyle 1} d_{\scriptscriptstyle 0}})}
\end{equation*} 
to get: 
\begin{equation} \label{eq: functional-eq-Z-sigma1}
\begin{split}
& Z^{(c)}(\mathbf{s};  \chi_{a_{\scriptscriptstyle 2} c_{\scriptscriptstyle 2}}, 
\chi_{a_{\scriptscriptstyle 1} c_{\scriptscriptstyle 1}})\, = \, \tfrac{1}{2}\,  
|c_{\scriptscriptstyle 1}|^{\frac{1}{2} - s_{\scriptscriptstyle 1}}
\frac{\varphi(c_{\scriptscriptstyle 2} c_{\scriptscriptstyle 3})}{|c_{\scriptscriptstyle 2} c_{\scriptscriptstyle 3}|}
\prod_{p \, \mid \, c_{\scriptscriptstyle 2}  c_{\scriptscriptstyle 3}} 
\big(\! 1 \, - \, |p|^{2 s_{\scriptscriptstyle 1} - 2}\big)^{\scriptscriptstyle -1}\\ 
& \cdot \sum_{\vartheta' \in \{1, \, \theta_{\scriptscriptstyle 0} \}} 
\chi_{a_{\scriptscriptstyle 2} {\scriptscriptstyle \vartheta'}}(c_{\scriptscriptstyle 1})
\!\left\{\gamma_{\scriptscriptstyle q}^{+}(s_{\scriptscriptstyle 1}; \, a_{\scriptscriptstyle 1})\, + \, 
\mathrm{sgn}(a_{\scriptscriptstyle 2} \vartheta')\, 
\gamma_{\scriptscriptstyle q}^{-}(s_{\scriptscriptstyle 1})\right\} 
\!\!\!\!\sum_{\substack{\ell \, \mid \, c_{\scriptscriptstyle 2} c_{\scriptscriptstyle 3} \\ (c_{\scriptscriptstyle 2}, \, \ell) = b}}
\chi_{a_{\scriptscriptstyle 1} c_{\scriptscriptstyle 1}}\!(\ell) \, 
U_{ \scriptscriptstyle \ell}(s_{\scriptscriptstyle 1}) 
Z^{(c)}(^{{\scriptscriptstyle \sigma_{\scriptscriptstyle 1}}}\mathbf{s}; 
\chi_{{\scriptscriptstyle \vartheta'} \!c_{\scriptscriptstyle 2} \ell \slash b^{\scriptscriptstyle 2}}, 
\chi_{a_{\scriptscriptstyle 1} c_{\scriptscriptstyle 1}}). 
\end{split}
\end{equation} 
Of course, by symmetry, \!we have similar functional equations in the variables $s_{\scriptscriptstyle 2}, \ldots, s_{r}.$ \!Writing 
explicitly now $r = 3$ and taking $s_{\scriptscriptstyle 1} \! = s_{\scriptscriptstyle 2} \! = s_{\scriptscriptstyle 3} = s,$ we can express the functional equation $\sigma_{\scriptscriptstyle 1} \sigma_{\scriptscriptstyle 2} \sigma_{\scriptscriptstyle 3}$ as 
\begin{equation} \label{eq: functional-eq-Z-sigma1sigma2sigma3}
\begin{split}
& Z^{(c)}(\mathbf{s};  \chi_{a_{\scriptscriptstyle 2} c_{\scriptscriptstyle 2}}, 
\chi_{a_{\scriptscriptstyle 1} c_{\scriptscriptstyle 1}})\\ 
& \! = \, \tfrac{1}{2}\, |c_{\scriptscriptstyle 1}|^{\frac{3}{2} - 3 s}\, 
\frac{\varphi(c_{\scriptscriptstyle 2} c_{\scriptscriptstyle 3})^{\scriptscriptstyle 3} 
V_{\! c_{\scriptscriptstyle 2} c_{\scriptscriptstyle 3}}\!(s)}
{|c_{\scriptscriptstyle 2} c_{\scriptscriptstyle 3}|^{\scriptscriptstyle 3}} 
\prod_{p \, \mid \, c_{\scriptscriptstyle 2}  c_{\scriptscriptstyle 3}} \big(\! 1 \, - \, |p|^{2 s - 2}\big)^{\! \scriptscriptstyle - 3}\\
& \cdot \sum_{\vartheta' \in \{1, \, \theta_{\scriptscriptstyle 0} \}} 
\chi_{a_{\scriptscriptstyle 2} {\scriptscriptstyle \vartheta'}}(c_{\scriptscriptstyle 1})
\!\left\{\gamma_{\scriptscriptstyle q}^{+}(s; a_{\scriptscriptstyle 1})^{\scriptscriptstyle 3}\, + \, 
\mathrm{sgn}(a_{\scriptscriptstyle 2} \vartheta')\,
\gamma_{\scriptscriptstyle q}^{-}(s)^{\scriptscriptstyle 3} \right\} 
\!\!\!\!\sum_{\substack{\ell \, \mid \, c_{\scriptscriptstyle 2} c_{\scriptscriptstyle 3} \\ (c_{\scriptscriptstyle 2}, \, \ell) = b}}
\chi_{a_{\scriptscriptstyle 1} c_{\scriptscriptstyle 1}}\!(\ell) \, W_{\scriptscriptstyle \ell}(s)\,
Z^{(c)}(^{{\scriptscriptstyle \sigma_{\scriptscriptstyle 1}\sigma_{\scriptscriptstyle 2}
\sigma_{\scriptscriptstyle 3}}}\mathbf{s}; \chi_{{\scriptscriptstyle \vartheta'} 
\!c_{\scriptscriptstyle 2} \ell \slash b^{\scriptscriptstyle 2}}, 
\chi_{a_{\scriptscriptstyle 1} c_{\scriptscriptstyle 1}})
\end{split}
\end{equation} 
where $V_{\scriptscriptstyle \ell}(s) = W_{\scriptscriptstyle \ell}(s) = 1$ for $\ell \in \mathbb{F}^{\times}\!,$ and 
\begin{equation*} 
V_{\scriptscriptstyle \ell}(s) = \prod_{p \, \mid \, \ell} \left(1 + 3 U_{\! p}^{2}(s) \right) 
\;\;\; \mathrm{and} \;\;\; 
W_{\scriptscriptstyle \ell}(s) = \prod_{p \, \mid \, \ell} 
\frac{U_{\! p}^{}(s)\left(3 +  U_{\! p}^{2}(s) \right)}{1 + 3 U_{\! p}^{2}(s)}
\end{equation*}  
for $\ell$ square-free of positive degree. This functional equation will be used in the next section.  

As in \cite{CG1} and \cite{DGH}, by applying the above functional equations and Bochner's theorem \cite{Boh}, it 
follows that 
$
Z^{(c)}(\mathbf{s};  \chi_{a_{\scriptscriptstyle 2} c_{\scriptscriptstyle 2}}, 
\chi_{a_{\scriptscriptstyle 1} c_{\scriptscriptstyle 1}})
$ 
can be meromorphically continued to $\mathbb{C}^{r + \scriptscriptstyle 1}$ ($r = 3$). Moreover, as in 
\cite[Proposition~4.11]{DGH}, the function 
\begin{equation} \label{eq: central-value-Z-var-w}
\big(1\, - \, q^{3 - 4w} \big)
\big(1 \, - \, q^{2 - 2 w} \big)^{\! \scriptscriptstyle 7} 
Z^{(c)}\big(\tfrac{1}{2}, \tfrac{1}{2}, \tfrac{1}{2}, w; \, \chi_{a_{\scriptscriptstyle 2} c_{\scriptscriptstyle 2}}, 
\chi_{a_{\scriptscriptstyle 1} c_{\scriptscriptstyle 1}}\big)
\end{equation} 
is entire and has order one. \!This function is, in fact, a polynomial in $q^{\scriptscriptstyle - w},$ but we shall not need 
this piece of information.

\subsection{Convexity bound} 
We shall now obtain a convexity bound for the function \eqref{eq: central-value-Z-var-w} analogous to that 
proved in \cite[Proposition~4.12]{DGH} over the rationals.

To obtain this estimate, we first note that by Proposition \ref{poly-P-estimate} and \eqref{eq: polyPd} we have 
\begin{equation*}
\big|P_{\scriptscriptstyle d}\big(\tfrac{1}{2}, \tfrac{1}{2}, \tfrac{1}{2}; 
\chi_{a_{\scriptscriptstyle 1} c_{\scriptscriptstyle 1} d_{\scriptscriptstyle 0}}\big)\big| \, \le  \, 
\big(\! \tfrac{843}{1 - 5^{\scriptscriptstyle - 2\,  \eta}}\!\big)^{\scriptscriptstyle \omega(d_{\scriptscriptstyle 1})} 
\, |d_{\scriptscriptstyle 1}|^{\scriptscriptstyle \frac{1}{2} + \eta} 
\end{equation*} 
for every small positive $\eta.$ Here $\omega(d_{\scriptscriptstyle 1})$ denotes the number of distinct monic irreducible factors of 
$d_{\scriptscriptstyle 1}.$ Choosing, for example, $\eta = 1\slash 5,$ we find easily that 
\begin{equation*}
\begin{split}
\big|Z^{(c)}\big(\tfrac{1}{2}, \tfrac{1}{2}, \tfrac{1}{2}, w; \chi_{a_{\scriptscriptstyle 2} c_{\scriptscriptstyle 2}}, 
\chi_{a_{\scriptscriptstyle 1} c_{\scriptscriptstyle 1}}\big)\big| 
& \, \le \, \big(\tfrac{3}{2}\big)^{\!{\scriptscriptstyle 3} \omega(c_{\scriptscriptstyle 2}c_{\scriptscriptstyle 3})} \; \cdot
\!\sum_{\substack{(d_{\scriptscriptstyle 0}^{}, \, c) = 1 \\  d_{\scriptscriptstyle 0}^{}-\mathrm{monic\, \& \, sq. \; free}}}\, 
\frac{L({\scriptstyle \frac{1}{2}}, \chi_{a_{\scriptscriptstyle 1} c_{\scriptscriptstyle 1}  d_{\scriptscriptstyle 0}})^{\scriptscriptstyle 3}}
{|d_{\scriptscriptstyle 0}|^{\scriptscriptstyle \Re(w)}}
\sum_{d_{\scriptscriptstyle 1}-\mathrm{monic}} \, 
\frac{(1776)^{\scriptscriptstyle \omega(d_{\scriptscriptstyle 1})}}
{|d_{\scriptscriptstyle 1}|^{\scriptscriptstyle 13\slash 10}} \\
& < \, \big(\tfrac{3}{2}\big)^{\!{\scriptscriptstyle 3} \omega(c_{\scriptscriptstyle 2}c_{\scriptscriptstyle 3})} 
\, \zeta\big(\tfrac{13}{10}\big)^{\! \scriptscriptstyle 1776}\; \cdot
\!\sum_{\substack{(d_{\scriptscriptstyle 0}^{}, \, c) = 1 \\  d_{\scriptscriptstyle 0}^{}-\mathrm{monic\, \& \, sq. \; free}}}\, 
\frac{L({\scriptstyle \frac{1}{2}}, \chi_{a_{\scriptscriptstyle 1} c_{\scriptscriptstyle 1}  d_{\scriptscriptstyle 0}})^{\scriptscriptstyle 3}}
{|d_{\scriptscriptstyle 0}|^{\scriptscriptstyle \Re(w)}}
\end{split}
\end{equation*} 
for all $w \in \mathbb{C}$ with $\Re(w) > 1.$ By Theorem \ref{estimate-L-realpart-w=1}, the last series is convergent. Moreover, 
for $w > 1$ and small $0 < \delta < w - 1,$ we have  
\begin{equation*} 
\sum_{\substack{(d_{\scriptscriptstyle 0}^{}, \, c) = 1 \\  
d_{\scriptscriptstyle 0}^{}-\mathrm{monic\, \& \, sq. \; free}}}  \, 
\frac{L({\scriptstyle \frac{1}{2}}, \chi_{a_{\scriptscriptstyle 1} c_{\scriptscriptstyle 1}  d_{\scriptscriptstyle 0}})^{\scriptscriptstyle 3}}
{|d_{\scriptscriptstyle 0}|^{w}} \, 
\ll_{\scriptscriptstyle \delta, \, q} \, |c_{\scriptscriptstyle 1}|^{\delta}\!\! \sum_{d-\mathrm{monic}}\,  \frac{1}{|d|^{w - \delta}}
\, \ll_{\scriptscriptstyle \delta, \, q} \, \frac{|c_{\scriptscriptstyle 1}|^{\delta}}{1 - q^{1 + \delta - w}}.
\end{equation*} 
The implied constant can be taken to be 
$
64 \, q^{{\scriptscriptstyle \delta} \, e^{\scriptscriptstyle 30 \slash \delta}}.
$ 
It follows that 
\begin{equation} \label{eq: EstimateZ1}
\big|Z^{(c)}\big(\tfrac{1}{2}, \tfrac{1}{2}, \tfrac{1}{2}, w; \chi_{a_{\scriptscriptstyle 2} c_{\scriptscriptstyle 2}}, 
\chi_{a_{\scriptscriptstyle 1} c_{\scriptscriptstyle 1}}\big)\big| 
\, < \, 64 \, \zeta\big(\tfrac{13}{10}\big)^{\! \scriptscriptstyle 1776} 
\big(\tfrac{3}{2}\big)^{\!{\scriptscriptstyle 3} \omega(c_{\scriptscriptstyle 2}c_{\scriptscriptstyle 3})} 
\cdot
\frac{q^{{\scriptscriptstyle \delta} \, e^{\scriptscriptstyle 30 \slash \delta}}\, |c_{\scriptscriptstyle 1}|^{\delta}}
{1 - q^{1 + \delta - {\scriptstyle \Re(w)}}}.
\end{equation} 
We shall now establish a similar estimate when $w \in \mathbb{C}$ with $\Re(w) = - \delta,$ for small positive $\delta.$ To do so, 
we shall apply the functional equation corresponding to the Weyl group element 
$ 
\tau : = \sigma_{\scriptscriptstyle 4}
\sigma_{\scriptscriptstyle 1}\sigma_{\scriptscriptstyle 2}\sigma_{\scriptscriptstyle 3}\sigma_{\scriptscriptstyle 4}
\sigma_{\scriptscriptstyle 1}\sigma_{\scriptscriptstyle 2}\sigma_{\scriptscriptstyle 3}\sigma_{\scriptscriptstyle 4}
$ 
relating the values of  
$
Z^{(c)}(\mathbf{s}; \chi_{a_{\scriptscriptstyle 2} c_{\scriptscriptstyle 2}}, 
\chi_{a_{\scriptscriptstyle 1} c_{\scriptscriptstyle 1}})
$ 
to the values of a linear combination of similar multiple {D}irichlet series at 
$
^{{\scriptscriptstyle \tau}}\mathbf{s},  
$ 
and then make use of \eqref{eq: EstimateZ1}. Note that 
\begin{equation*}
^{\tau}\big(\tfrac{1}{2}, \tfrac{1}{2}, \tfrac{1}{2}, w \big)
= \big(\tfrac{1}{2}, \tfrac{1}{2}, \tfrac{1}{2}, 1 - w \big).
\end{equation*} 
Following \cite{CD} and \cite{DGH}, we write the functional equations \eqref{eq: functional-eq-Z-sigma_r+1}, 
\eqref{eq: functional-eq-Z-sigma1} in the matrix notation. If we denote by $\vec{Z}^{(c)}(\mathbf{s})$ the column vector whose 
entries are the multiple {D}irichlet series 
$
Z^{(c)}(\mathbf{s};  \chi_{a_{\scriptscriptstyle 2} c_{\scriptscriptstyle 2}}, 
\chi_{a_{\scriptscriptstyle 1} c_{\scriptscriptstyle 1}}),
$ 
then there are matrices $\mathrm{X}_{c}(s_{\scriptscriptstyle 4})$ and $\mathrm{Y}_{\! c}(s)$ such that  
\begin{equation*}
\vec{Z}^{(c)}(\mathbf{s}) = \mathrm{X}_{c}(s_{\scriptscriptstyle 4}) \cdot 
\vec{Z}^{(c)}(^{{\scriptscriptstyle \sigma_{\scriptscriptstyle 4}}}\mathbf{s}) \;\;\;\;  \text{and} \;\;\;\; 
\vec{Z}^{(c)}(\mathbf{s}) = \mathrm{Y}_{\! c}(s_{\scriptscriptstyle i}) \cdot
\vec{Z}^{(c)}(^{{\scriptscriptstyle \sigma_{\scriptscriptstyle i}}}\mathbf{s}) \qquad \text{(for $i = 1, 2, 3$).}
\end{equation*} 
Taking $\mathbf{s} = \big(\tfrac{1}{2}, \tfrac{1}{2}, \tfrac{1}{2}, w \big)$ and applying successively the functional equations corresponding to 
$ 
\sigma_{\scriptscriptstyle 4}, \sigma_{\scriptscriptstyle 1}, \sigma_{\scriptscriptstyle 2}, \ldots, \sigma_{\scriptscriptstyle 4},
$ 
\!we obtain:
\begin{equation*}
\vec{Z}^{(c)}\big(\tfrac{1}{2}, \tfrac{1}{2}, \tfrac{1}{2}, w \big) = \mathscr{M}(w) \cdot 
\vec{Z}^{(c)}\big(\tfrac{1}{2}, \tfrac{1}{2}, \tfrac{1}{2}, 1 - w \big)
\end{equation*} 
where the matrix $\mathscr{M}(w)$ is given by
\begin{equation*}
\mathscr{M}(w) = \mathrm{X}_{c}(w) \mathrm{Y}_{\! c}(w)^{\scriptscriptstyle 3} \, \mathrm{X}_{c}\big(2 w - \tfrac{1}{2}\big)\, 
\mathrm{Y}_{\! c}(w)^{\scriptscriptstyle 3} \, \mathrm{X}_{c}(w).
\end{equation*} 
We shall now estimate the entries of the matrix $\mathscr{M}(w).$ Let 
$s_{\scriptscriptstyle 4} \in \mathbb{C}$ with $\Re(s_{\scriptscriptstyle 4}) < 0.$ Since
\begin{equation*}
|\gamma_{\scriptscriptstyle q}^{+}(s_{\scriptscriptstyle 4}; \, a_{\scriptscriptstyle 2})| \, + \, 
|\gamma_{\scriptscriptstyle q}^{-}(s_{\scriptscriptstyle 4})| < 4
\;\;\;\;\;\; |U_{\scriptscriptstyle m}(s_{\scriptscriptstyle 4})| \le 3^{\omega(m)}|m|^{- \Re(s_{\scriptscriptstyle 4})}
\end{equation*} 
and 
\begin{equation*}
\prod_{p \, \mid \, c_{\scriptscriptstyle 1}  c_{\scriptscriptstyle 3}} \big(\! 1 \, - \, 
|p|^{2 \Re(s_{\scriptscriptstyle 4}) - 2}\big)^{\scriptscriptstyle -1}  \le \, 
\big(\!\tfrac{25}{24}\!\big)^{\!\omega(c_{\scriptscriptstyle 1}  c_{\scriptscriptstyle 3})} 
\end{equation*} 
we have by \eqref{eq: functional-eq-Z-sigma_r+1} that 
\begin{equation} \label{eq: intermediate-estimate1}
|Z^{(c)}(\mathbf{s};  \chi_{a_{\scriptscriptstyle 2} c_{\scriptscriptstyle 2}}, 
\chi_{a_{\scriptscriptstyle 1} c_{\scriptscriptstyle 1}})|\,  
\le \, 4^{\omega(c_{\scriptscriptstyle 1} c_{\scriptscriptstyle 3}) + 1}\, 
|c_{\scriptscriptstyle 2}|^{\frac{1}{2} - \Re(s_{\scriptscriptstyle 4})} 
\!\!\sum_{\substack{m \, \mid \, c_{\scriptscriptstyle 1} c_{\scriptscriptstyle 3} \\ (c_{\scriptscriptstyle 1}\!, \, m) = e}} 
|m|^{- \Re(s_{\scriptscriptstyle 4})} \cdot \tfrac{1}{2}\sum_{\vartheta \in \{1, \, \theta_{\scriptscriptstyle 0} \}} 
|Z^{(c)}(^{{\scriptscriptstyle \sigma_{\scriptscriptstyle 4}}}\mathbf{s}; 
\chi_{a_{\scriptscriptstyle 2} c_{\scriptscriptstyle 2}}, 
\chi_{{\scriptscriptstyle \vartheta} c_{\scriptscriptstyle 1} m \slash e^{\scriptscriptstyle 2}})|.
\end{equation} 
Note that for every divisor $m$ of $c_{\scriptscriptstyle 1} c_{\scriptscriptstyle 3},$ the monic polynomial 
$c_{\scriptscriptstyle 1} m \slash e^{\scriptscriptstyle 2}$ is also a divisor of $c_{\scriptscriptstyle 1} c_{\scriptscriptstyle 3}.$ Conversely, to every 
pair $(l_{\scriptscriptstyle 1}, e')$ with $l_{\scriptscriptstyle 1}  \! \mid  c_{\scriptscriptstyle 1}$ and $e'  \mid c_{\scriptscriptstyle 3}$ there corresponds $m : = (c_{\scriptscriptstyle 1} \slash l_{\scriptscriptstyle 1}) e'.$ 

Similarly, for $s_{\scriptscriptstyle i} \in \mathbb{C}$ with $\Re(s_{\scriptscriptstyle i}) < 0$ ($i = 1, 2, 3$), we have 
\begin{equation} \label{eq: intermediate-estimate2}
|Z^{(c)}(\mathbf{s};  \chi_{a_{\scriptscriptstyle 2} c_{\scriptscriptstyle 2}}, 
\chi_{a_{\scriptscriptstyle 1} c_{\scriptscriptstyle 1}})| \, 
\le \, 4^{\omega(c_{\scriptscriptstyle 2} c_{\scriptscriptstyle 3}) + 1} 
|c_{\scriptscriptstyle 1}|^{\frac{1}{2} - \Re(s_{\scriptscriptstyle i})}
\!\!\sum_{\substack{\ell \, \mid \, c_{\scriptscriptstyle 2} c_{\scriptscriptstyle 3} \\ (c_{\scriptscriptstyle 2}, \, \ell) = b}} 
|\ell|^{- \Re(s_{\scriptscriptstyle i})} \cdot \tfrac{1}{2} \sum_{\vartheta' \in \{1, \, \theta_{\scriptscriptstyle 0} \}}  
|Z^{(c)}(^{{\scriptscriptstyle \sigma_{\scriptscriptstyle i}}}\mathbf{s}; 
\chi_{{\scriptscriptstyle \vartheta'} \!c_{\scriptscriptstyle 2} \ell \slash b^{\scriptscriptstyle 2}}, 
\chi_{a_{\scriptscriptstyle 1} c_{\scriptscriptstyle 1}})|. 
\end{equation} 
To estimate the entries of the matrix $\mathscr{M}(w)$ when $\Re(w) = - \delta,$ we need to estimate an expression $E$ of the form 
\begin{equation*}
E = |c|^{\scriptscriptstyle 10 \delta}\,
|c_{\scriptscriptstyle 2}|^{\scriptscriptstyle \frac{1}{2}} |b_{\scriptscriptstyle 2}|^{\scriptscriptstyle \frac{3}{2}} 
\Big(|b_{\scriptscriptstyle 3}|\,  |m|^{\scriptscriptstyle \frac{1}{2}} \!\Big)  
\, |b_{\scriptscriptstyle 4}|^{\scriptscriptstyle \frac{3}{2}} |b_{\scriptscriptstyle 5}|^{\scriptscriptstyle \frac{1}{2}}
\end{equation*}
where $m, b_{\scriptscriptstyle 2}, \ldots, b_{\scriptscriptstyle 5} \in \mathbb{F}[x]$ are (monic) divisors of $c$ such that 
$
(c_{\scriptscriptstyle 2}, \, b_{\scriptscriptstyle 2}) = (b_{\scriptscriptstyle 3}, \, m) = (b_{\scriptscriptstyle i}, \, b_{\scriptscriptstyle i + 1}) = 1,
$ 
\!and $b_{\scriptscriptstyle 4}$ is $b_{\scriptscriptstyle 2} m$ modulo squares. \!Let $p$ be a monic irreducible divisor of $c.$ 
If $p \mid c_{\scriptscriptstyle 2}$ then $p \nmid b_{\scriptscriptstyle 2},$ \!and the power of $p$ dividing 
$ 
b_{\scriptscriptstyle 3}^{\scriptscriptstyle 2}  m  \, b_{\scriptscriptstyle 4}^{\scriptscriptstyle 3} b_{\scriptscriptstyle 5}
$
\!cannot exceed 4; it is 4 if and only if $p \mid m,$ and hence $p \mid b_{\scriptscriptstyle 4}.$ If $p \nmid c_{\scriptscriptstyle 2}$ then 
the power of $p$ dividing  
$
b_{\scriptscriptstyle 2}^{\scriptscriptstyle 3} b_{\scriptscriptstyle 3}^{\scriptscriptstyle 2}   m \,   
b_{\scriptscriptstyle 4}^{\scriptscriptstyle 3} b_{\scriptscriptstyle 5}
$
cannot exceed 6; it is 6 if and only if $p \mid b_{\scriptscriptstyle 2}$ and $p \mid b_{\scriptscriptstyle 4}$ (hence $p \nmid m$). Thus 
\begin{equation*}
E \le |c_{\scriptscriptstyle 2}|^{\scriptscriptstyle \frac{5}{2} + 10 \delta} \, |c\slash c_{\scriptscriptstyle 2}|^{\, \scriptscriptstyle 3 + 10 \delta}.
\end{equation*} 
Since the dimension of the matrices $\mathrm{X}_{c}$ and $\mathrm{Y}_{\! c}$ is $4 \cdot 3^{\omega(c)}\!,$ it follows from \eqref{eq: EstimateZ1} that, 
\!for $w \in \mathbb{C}$ with $\Re(w) = - \delta,$ \!we have  
\begin{equation*}
\big|Z^{(c)}\big(\tfrac{1}{2}, \tfrac{1}{2}, \tfrac{1}{2}, w;  \chi_{a_{\scriptscriptstyle 2} c_{\scriptscriptstyle 2}}, 
\chi_{a_{\scriptscriptstyle 1} c_{\scriptscriptstyle 1}}\big)\big|\,  
\ll_{\scriptscriptstyle \delta, \, q}  18^{{\scriptscriptstyle 9} \, \omega(c)} |c_{\scriptscriptstyle 2}|^{\scriptscriptstyle \frac{5}{2} + 11 \delta} 
\, |c\slash c_{\scriptscriptstyle 2}|^{\, \scriptscriptstyle 3 + 11 \delta}.
\end{equation*} 
Thus by applying the Phragmen-Lindel\"of principle, for every $\delta > 0,$ we have the estimate 
\begin{equation} \label{eq: basic-initial-estimate}
\big|\big(1\, - \, q^{3 - 4w} \big)
\big(1 \, - \, q^{2 - 2 w} \big)^{\! \scriptscriptstyle 7} 
Z^{(c)}\big(\tfrac{1}{2}, \tfrac{1}{2}, \tfrac{1}{2}, w; \, \chi_{a_{\scriptscriptstyle 2} c_{\scriptscriptstyle 2}}, 
\chi_{a_{\scriptscriptstyle 1} c_{\scriptscriptstyle 1}}\big)\big|\,  \ll_{\scriptscriptstyle \delta, \, q} 
20^{\, {\scriptscriptstyle 9} \, \omega(c)}
|c_{\scriptscriptstyle 2}|^{\scriptscriptstyle \frac{5}{2}(1 - \Re(w)) + \delta} 
\, |c\slash c_{\scriptscriptstyle 2}|^{\, \scriptscriptstyle 3 (1 - \Re(w)) + \delta}
\end{equation} 
for all $w$ with $0 \le \Re(w) \le 1.$

As noted in the introduction, one of the main ingredients in the proof of Theorem \ref{Main Theorem A} is an 
improvement of \eqref{eq: basic-initial-estimate} in the $c_{\scriptscriptstyle 3}$-aspect. This will be established 
in Proposition \ref{key-proposition}.

\section{Poles of multiple {D}irichlet series and their residues} 
Throughout this section we are assuming that $r = 3.$ \!By \eqref{eq: MDS-vers2}, 
\!the multiple {D}irichlet series 
$
Z^{(c)}(\mathbf{s};  \chi_{a_{\scriptscriptstyle 2} c_{\scriptscriptstyle 2}}, \chi_{a_{\scriptscriptstyle 1} c_{\scriptscriptstyle 1}}) 
$ 
has a (simple) pole at $s_{r + \scriptscriptstyle 1} \! = s_{\scriptscriptstyle 4} = 1$ only if 
$a_{\scriptscriptstyle 2} = c_{\scriptscriptstyle 2} = 1,$ and the part of 
$
Z^{(c_{\scriptscriptstyle 1} c_{\scriptscriptstyle 3})}(\mathbf{s};  1, \chi_{a_{\scriptscriptstyle 1} c_{\scriptscriptstyle 1}}) 
$ 
that contributes to this pole is 
\begin{equation*}
\zeta^{\scriptscriptstyle (c_{\scriptscriptstyle 1} c_{\scriptscriptstyle 3})}(s_{r + \scriptscriptstyle 1}) 
\; \cdot \sum_{\substack{(m_{\scriptscriptstyle 1} \cdots \, m_{r}, \, c_{\scriptscriptstyle 1} c_{\scriptscriptstyle 3}) = 1 
\\ m_{\scriptscriptstyle 1} \cdots \, m_{r} \, = \, n_{\scriptscriptstyle 1}^{2}}}  \, 
\frac{Q_{\underline{m}}(s_{r + \scriptscriptstyle 1}; 1)}
{|m_{\scriptscriptstyle 1}|^{s_{\scriptscriptstyle 1}} \cdots \, |m_{r}|^{s_{r}}} 
\, = \, \zeta^{\scriptscriptstyle (c_{\scriptscriptstyle 1} c_{\scriptscriptstyle 3})}(s_{r + \scriptscriptstyle 1}) \, \cdot
\prod_{p \nmid c_{\scriptscriptstyle 1} c_{\scriptscriptstyle 3}}
\bigg(\, \sum_{|\underline{k}| \, \equiv \, 0 \!\!\!\!\!\!\pmod 2} 
\frac{Q_{\underline{k}}(|p|^{\, - s_{r + \scriptscriptstyle 1}}; \,  |p|)}
{|p|^{k_{\scriptscriptstyle 1} s_{\scriptscriptstyle 1} + \cdots + k_{r}s_{r}}} \bigg).
\end{equation*} 
From the definition of the polynomials $Q_{\underline{k}}(t_{r + \scriptscriptstyle 1}; q)$ 
(see Appendix \ref{Appendix B}) it is straightforward to check that 
\begin{equation} \label{eq: part1-residue-w=1}
\left.\frac{Z^{(c_{\scriptscriptstyle 1} c_{\scriptscriptstyle 3})}
(\mathbf{s};  1, \chi_{a_{\scriptscriptstyle 1} c_{\scriptscriptstyle 1}})}
{\zeta(s_{\scriptscriptstyle 4})} 
\right\vert_{s_{\scriptscriptstyle 4} = 1} 
= \; \frac{\zeta^{\scriptscriptstyle (c_{\scriptscriptstyle 1} c_{\scriptscriptstyle 3})}(2 s_{\scriptscriptstyle 1} 
+ 2 s_{\scriptscriptstyle 2} + 2 s_{\scriptscriptstyle 3} - 1)}{\zeta_{\scriptscriptstyle c_{\scriptscriptstyle 1} 
c_{\scriptscriptstyle 3}}\!(1)} \, \prod_{i = 1}^{3}
\zeta^{\scriptscriptstyle (c_{\scriptscriptstyle 1} c_{\scriptscriptstyle 3})}(2 s_{\scriptscriptstyle i}) \; \cdot
\prod_{1\le i < j \le 3} \zeta^{\scriptscriptstyle (c_{\scriptscriptstyle 1} c_{\scriptscriptstyle 3})}
(s_{\scriptscriptstyle i} + s_{\! \scriptscriptstyle j}).
\end{equation} 
This ``modified'' residue of 
$
Z^{(c_{\scriptscriptstyle 1} c_{\scriptscriptstyle 3})}(\mathbf{s};  1, \chi_{a_{\scriptscriptstyle 1} c_{\scriptscriptstyle 1}}) 
$ 
(to which we shall refer as {\sl residue}) is more convenient to work with in our context. We also have that  
\begin{equation} \label{eq: part2-residue-w=1} 
\left.\frac{Z^{(c_{\scriptscriptstyle 1} c_{\scriptscriptstyle 3})}
(\mathbf{s};  \chi_{\scriptscriptstyle \theta_{\scriptscriptstyle 0}}, \chi_{a_{\scriptscriptstyle 1} c_{\scriptscriptstyle 1}})}
{L(s_{\scriptscriptstyle 4}, \chi_{\scriptscriptstyle \theta_{\scriptscriptstyle 0}})}
\right\vert_{q^{ - s_{\scriptscriptstyle 4}} = - q^{-1}} = \; 
\left.\frac{Z^{(c_{\scriptscriptstyle 1} c_{\scriptscriptstyle 3})}
(\mathbf{s};  1, \chi_{a_{\scriptscriptstyle 1} c_{\scriptscriptstyle 1}})}
{\zeta(s_{\scriptscriptstyle 4})} 
\right\vert_{s_{\scriptscriptstyle 4} = 1}.
\end{equation}

For our purposes, it will suffice to compute the residues at the remaining poles of 
$
Z^{(c)}(\mathbf{s};  \chi_{a_{\scriptscriptstyle 2} c_{\scriptscriptstyle 2}}, \chi_{a_{\scriptscriptstyle 1} c_{\scriptscriptstyle 1}})
$ 
only when $s_{\scriptscriptstyle 1} \! = s_{\scriptscriptstyle 2}$ $= s_{\scriptscriptstyle 3}  = \tfrac{1}{2},$ and 
$a_{\scriptscriptstyle 1} \! = 1.$\!\footnote{As in \cite{DGH}, the only possible poles of the function 
$
Z^{(c)}(\mathbf{s};  \chi_{a_{\scriptscriptstyle 2} c_{\scriptscriptstyle 2}}, \chi_{c_{\scriptscriptstyle 1}})
$ 
(with   
$
s_{\scriptscriptstyle 1} \! = s_{\scriptscriptstyle 2}$ $= s_{\scriptscriptstyle 3}  = {\small 1\slash 2}
$) 
may occur when $q^{- s_{\scriptscriptstyle 4}} = \pm \, 1\slash q$ of order at most seven, and 
$q^{- s_{\scriptscriptstyle 4}} = \pm \, q^{\scriptscriptstyle - 3\slash 4}, \, \pm\,  i \, q^{\scriptscriptstyle - 3\slash 4}$ of order 
at most one.} Fix $\vartheta' \in \{1, \, \theta_{\scriptscriptstyle 0} \},$ and let $\rho(\vartheta')$ be such that 
$\rho(\vartheta') \in \{\pm 1\}$ if $\vartheta' \! = 1$ or $\rho(\vartheta') \in \{\pm i\}$ if 
$\vartheta' \!= \theta_{\scriptscriptstyle 0}.$ We define 
\begin{equation} \label{eq: comb-Gamma} 
\Gamma(a_{\scriptscriptstyle 2}, \vartheta'; \rho(\vartheta')) = 
\left.\sum_{\vartheta \in \{1, \, \theta_{\scriptscriptstyle 0} \}}  
\left\{\gamma_{\scriptscriptstyle q}^{+}(s_{\scriptscriptstyle 4}; \, a_{\scriptscriptstyle 2}) \, + \, 
\mathrm{sgn}(\vartheta)\, \gamma_{\scriptscriptstyle q}^{-}(s_{\scriptscriptstyle 4})\right\} 
\left\{\gamma_{\scriptscriptstyle q}^{+}(s_{\scriptscriptstyle 4}; \vartheta)^{3} 
+ \,  \mathrm{sgn}(a_{\scriptscriptstyle 2} \vartheta')\,
\gamma_{\scriptscriptstyle q}^{-}(s_{\scriptscriptstyle 4})^{3} \right\}
\right \vert_{q^{ - s_{\scriptscriptstyle 4}} = \, \overline{\rho(\vartheta')}\, q^{\scriptscriptstyle - \frac{3}{4}}}.
\end{equation}  
For the reader's convenience, we give the explicit values of $\Gamma(a_{\scriptscriptstyle 2}, \vartheta'; \rho(\vartheta'))$ 
in the following table. 
\begin{table}[h] 
\centering
\begin{tabular}{|ccc|}
\hline 
$a_{\scriptscriptstyle 2}$ & \small{$\rho(\vartheta')$} & \small{$\Gamma(a_{\scriptscriptstyle 2}, \vartheta' \!; \rho(\vartheta'))$} \\
\hline
$1$  & $1$   & \textcolor{MyDarkBlue}{\small{$2 \, (1 + q^{\scriptscriptstyle 1\slash 4} + 10 \, q^{\scriptscriptstyle 1\slash 2} + 7  q^{\scriptscriptstyle 3 \slash 4} 
\, + 20 \, q + 7  q^{\scriptscriptstyle 5\slash 4} + 10 \, q^{\scriptscriptstyle 3\slash 2} + q^{\scriptscriptstyle 7\slash 4} + q^{\scriptscriptstyle 2})$}} \\  
$\theta_{\scriptscriptstyle 0}$ & $- 1$ & \small{$2 \, (1 + q^{\scriptscriptstyle 1\slash 4} + 10 \, q^{\scriptscriptstyle 1\slash 2} + 7  q^{\scriptscriptstyle 3 \slash 4}\, + 20 \, q + 7  q^{\scriptscriptstyle 5\slash 4} + 10 \, q^{\scriptscriptstyle 3\slash 2} + q^{\scriptscriptstyle 7\slash 4} + q^{\scriptscriptstyle 2})$} \\ 
$1$ & $-1$ & \textcolor{MyDarkGreen}{\small{$2 \, (1 - q^{\scriptscriptstyle 1\slash 4} + 10 \, q^{\scriptscriptstyle 1\slash 2} - 7  q^{\scriptscriptstyle 3 \slash 4} 
\, + 20 \, q - 7  q^{\scriptscriptstyle 5\slash 4} + 10 \, q^{\scriptscriptstyle 3\slash 2} - q^{\scriptscriptstyle 7\slash 4} + q^{\scriptscriptstyle 2})$}} \\
$\theta_{\scriptscriptstyle 0}$ & $1$ & \small{$2 \, (1 - q^{\scriptscriptstyle 1\slash 4} + 10 \, q^{\scriptscriptstyle 1\slash 2} - 7  q^{\scriptscriptstyle 3 \slash 4} 
\, + 20 \, q - 7  q^{\scriptscriptstyle 5\slash 4} + 10 \, q^{\scriptscriptstyle 3\slash 2} - q^{\scriptscriptstyle 7\slash 4} + q^{\scriptscriptstyle 2})$} \\ 
$1$ & $i$ & \textcolor{MyDarkRed}{\small{$2 \, (1 - i \, q^{\scriptscriptstyle 1\slash 4}\! - 4 \, q^{\scriptscriptstyle 1\slash 2} 
+ 7 i  \, q^{\scriptscriptstyle 3 \slash 4} + 6 \, q - 7 i \, q^{\scriptscriptstyle 5\slash 4}\! - 4 \, q^{\scriptscriptstyle 3\slash 2} + i \, q^{\scriptscriptstyle 7\slash 4} + q^{\scriptscriptstyle 2})$}} \\
$\theta_{\scriptscriptstyle 0}$ & $-i$ & \small{$2 \, (1 - i \, q^{\scriptscriptstyle 1\slash 4}\! - 4 \, q^{\scriptscriptstyle 1\slash 2} 
+ 7 i  \, q^{\scriptscriptstyle 3 \slash 4} + 6 \, q - 7 i \, q^{\scriptscriptstyle 5\slash 4}\! - 4 \, q^{\scriptscriptstyle 3\slash 2} + i \, q^{\scriptscriptstyle 7\slash 4} + q^{\scriptscriptstyle 2})$} \\ 
$1$ & $- i$ & \textcolor{violet}{\small{$2 \, (1 + i \, q^{\scriptscriptstyle 1\slash 4}\! - 4 \, q^{\scriptscriptstyle 1\slash 2} 
- 7 i  \, q^{\scriptscriptstyle 3 \slash 4} + 6 \, q + 7 i \, q^{\scriptscriptstyle 5\slash 4}\! - 4 \, q^{\scriptscriptstyle 3\slash 2} - i \, q^{\scriptscriptstyle 7\slash 4} + q^{\scriptscriptstyle 2})$}} \\
$\theta_{\scriptscriptstyle 0}$ & $i$ & \small{$2 \, (1 + i \, q^{\scriptscriptstyle 1\slash 4}\! - 4 \, q^{\scriptscriptstyle 1\slash 2} 
- 7 i  \, q^{\scriptscriptstyle 3 \slash 4} + 6 \, q + 7 i \, q^{\scriptscriptstyle 5\slash 4}\! - 4 \, q^{\scriptscriptstyle 3\slash 2} - i \, q^{\scriptscriptstyle 7\slash 4} + q^{\scriptscriptstyle 2})$}\\
\hline
\end{tabular} 
\end{table} 
{\vskip-7pt} 
Note that there are four distinct values of $\Gamma(a_{\scriptscriptstyle 2}, \vartheta'; \rho(\vartheta'))$ in total, indicated with four different colors.  
\vskip10pt
\begin{prop}\label{MDS-residue-3/4} --- Let $c\in \mathbb{F}[x]$ be monic and square-free, and let $a_{\scriptscriptstyle 2}, 
\vartheta' \in \{1, \theta_{\scriptscriptstyle 0} \}.$ Suppose that $c = c_{\scriptscriptstyle 1}c_{\scriptscriptstyle 2}c_{\scriptscriptstyle 3}$ 
with $c_{\scriptscriptstyle i} \in \mathbb{F}[x]$ monic for all $i.$ Then, for $\rho(\vartheta')$ as above, \!we have
\begin{equation*}
\begin{split}
& \left.\Big(\! 1 - \rho(\vartheta') q^{{\scriptscriptstyle \frac{3}{4} - s_{\scriptscriptstyle 4}}} \!\Big)\, 
Z^{(c)}\big(\tfrac{1}{2}, \tfrac{1}{2}, \tfrac{1}{2}, \, s_{\scriptscriptstyle 4}; 
\chi_{a_{\scriptscriptstyle 2} c_{\scriptscriptstyle 2}}, \chi_{c_{\scriptscriptstyle 1}}\! \big)
\right\vert_{q^{ - s_{\scriptscriptstyle 4}} \, = \; \overline{\rho(\vartheta')}\, q^{\scriptscriptstyle - \frac{3}{4}}}\\
& = \frac{\chi_{a_{\scriptscriptstyle 2} c_{\scriptscriptstyle 2}}\!(c_{\scriptscriptstyle 1})}{8}\,
\Gamma(a_{\scriptscriptstyle 2}, \vartheta'; \rho(\vartheta'))\, 
L\big(\tfrac{1}{2}, \chi_{{\scriptscriptstyle \vartheta'}}\! \big)^{\scriptscriptstyle 7}\\
& \cdot \;  \rho(\vartheta')^{\deg \, c_{\scriptscriptstyle 1}}\,  |c_{\scriptscriptstyle 1}|^{\scriptscriptstyle -  1 \slash 4} 
\prod_{p \, \mid \, c_{\scriptscriptstyle 1}} 
\!\big(1 - \chi_{{\scriptscriptstyle \vartheta'}}(p) |p|^{\scriptscriptstyle - 1 \slash 2} \big)^{\! \scriptscriptstyle 8}
\big(1 + \chi_{{\scriptscriptstyle \vartheta'}}(p) |p|^{\scriptscriptstyle - 1 \slash 2} \big)^{\! \scriptscriptstyle 2} 
\big(1 + 6 \, \chi_{{\scriptscriptstyle \vartheta'}}(p) |p|^{\scriptscriptstyle - 1 \slash 2}  +   |p|^{\scriptscriptstyle -1} \big)\\
& \cdot \; |c_{\scriptscriptstyle 2}|^{\scriptscriptstyle -  1 \slash 2} \prod_{p \, \mid \, c_{\scriptscriptstyle 2}}
\!\big(1 - \chi_{{\scriptscriptstyle \vartheta'}}(p) |p|^{\scriptscriptstyle - 1 \slash 2} \big)^{\! \scriptscriptstyle 8}
\big(1 + \chi_{{\scriptscriptstyle \vartheta'}}(p) |p|^{\scriptscriptstyle - 1 \slash 2} \big)\, 
\big(3 + 7 \chi_{{\scriptscriptstyle \vartheta'}}(p) |p|^{\scriptscriptstyle - 1 \slash 2} + 3\,  |p|^{\scriptscriptstyle -1} \big)\\
& \cdot \; \prod_{p \, \mid \, c_{\scriptscriptstyle 3}} 
\!\big(1 - \chi_{{\scriptscriptstyle \vartheta'}}(p) |p|^{\scriptscriptstyle - 1 \slash 2} \big)^{\! \scriptscriptstyle 8}
\big(1 + \chi_{{\scriptscriptstyle \vartheta'}}(p) |p|^{\scriptscriptstyle - 1 \slash 2} \big)\, 
\big(1 + 7 \chi_{{\scriptscriptstyle \vartheta'}}(p) |p|^{\scriptscriptstyle - 1 \slash 2} 
+ 13 \, |p|^{\scriptscriptstyle -1} + 7 \chi_{{\scriptscriptstyle \vartheta'}}(p) |p|^{\scriptscriptstyle - 3 \slash 2}  + |p|^{\scriptscriptstyle - 2}\big).
\end{split}
\end{equation*} 
\end{prop}

\begin{proof} \!We first apply the functional equation \eqref{eq: functional-eq-Z-sigma_r+1} (recall that $r = 3$), and write  
\begin{equation*}
\begin{split}
& Z^{(c)}(\mathbf{s};  \chi_{a_{\scriptscriptstyle 2} c_{\scriptscriptstyle 2}}, 
\chi_{c_{\scriptscriptstyle 1}})\,  
= \,  \tfrac{1}{2}\,  |c_{\scriptscriptstyle 2}|^{\frac{1}{2} - s_{\scriptscriptstyle 4}} 
\frac{\varphi(c_{\scriptscriptstyle 1} c_{\scriptscriptstyle 3})}{|c_{\scriptscriptstyle 1} c_{\scriptscriptstyle 3}|}
\prod_{p \, \mid \, c_{\scriptscriptstyle 1}  c_{\scriptscriptstyle 3}} \big(\! 1 \, - \, 
|p|^{2 s_{\scriptscriptstyle 4} - 2}\big)^{\scriptscriptstyle -1}\\
&\cdot \sum_{\vartheta \in \{1, \, \theta_{\scriptscriptstyle 0} \}}                
\chi_{{\scriptscriptstyle \vartheta}}(c_{\scriptscriptstyle 2})
\!\left\{\gamma_{\scriptscriptstyle q}^{+}(s_{\scriptscriptstyle 4}; \, a_{\scriptscriptstyle 2}) \, + \, 
\mathrm{sgn}(\vartheta)\, \gamma_{\scriptscriptstyle q}^{-}(s_{\scriptscriptstyle 4})\right\} 
\!\!\!\!\sum_{\substack{m \, \mid \, c_{\scriptscriptstyle 1} c_{\scriptscriptstyle 3} \\ (c_{\scriptscriptstyle 1}\!, \, m) = e}}
\chi_{a_{\scriptscriptstyle 2} c_{\scriptscriptstyle 2}}\!(m) \, 
U_{\scriptscriptstyle m}(s_{\scriptscriptstyle 4})\,  
Z^{(c)}(^{{\scriptscriptstyle \sigma_{\scriptscriptstyle 4}}}\mathbf{s}; 
\chi_{a_{\scriptscriptstyle 2} c_{\scriptscriptstyle 2}}, 
\chi_{{\scriptscriptstyle \vartheta} c_{\scriptscriptstyle 1} m \slash e^{\scriptscriptstyle 2}}).
\end{split}
\end{equation*} 
Letting $s_{\scriptscriptstyle 1} \! = s_{\scriptscriptstyle 2} \! = s_{\scriptscriptstyle 3} = s,$ we have by \eqref{eq: functional-eq-Z-sigma1sigma2sigma3} 
that
 \begin{equation*}
\begin{split}
& Z^{(c)}(^{{\scriptscriptstyle \sigma_{\scriptscriptstyle 4}}}\mathbf{s};  \chi_{a_{\scriptscriptstyle 2} c_{\scriptscriptstyle 2}}, 
\chi_{{\scriptscriptstyle \vartheta} c_{\scriptscriptstyle 1} m \slash e^{\scriptscriptstyle 2}})\\ 
& \! = \, \frac{1}{2}\, \left|\frac{c_{\scriptscriptstyle 1} m}{e^{\scriptscriptstyle 2}}\right|^{\, 3(1 - s - s_{\scriptscriptstyle 4})}\, 
\frac{\varphi\left(\! \frac{c e^{\scriptscriptstyle 2}}{c_{\scriptscriptstyle 1} m} \!\right)^{\!\! \scriptscriptstyle 3} 
V_{\!\!\!\! \frac{c e^{\scriptscriptstyle 2}}{c_{\scriptscriptstyle 1} m}}\!(s + s_{\scriptscriptstyle 4} - {\scriptscriptstyle \frac{1}{2}})}
{\left| \frac{c e^{\scriptscriptstyle 2}}{c_{\scriptscriptstyle 1} m} \right|^{\scriptscriptstyle 3}} 
\prod_{p \, \mid \frac{c e^{\scriptscriptstyle 2}}{c_{\scriptscriptstyle 1} m}} 
\big(\! 1 \, - \, |p|^{2 s + 2 s_{\scriptscriptstyle 4} - 3}\big)^{\! \scriptscriptstyle - 3}\\
& \cdot \sum_{\vartheta' \in \{1, \, \theta_{\scriptscriptstyle 0} \}} 
\chi_{a_{\scriptscriptstyle 2} {\scriptscriptstyle \vartheta'}}\left(\!\frac{c_{\scriptscriptstyle 1} m}{e^{\scriptscriptstyle 2}}\!\right)
\left\{\gamma_{\scriptscriptstyle q}^{+}(s + s_{\scriptscriptstyle 4} - {\scriptscriptstyle \frac{1}{2}}; \vartheta)^{\scriptscriptstyle 3} 
+ \,  \mathrm{sgn}(a_{\scriptscriptstyle 2} \vartheta')\,
\gamma_{\scriptscriptstyle q}^{-}(s + s_{\scriptscriptstyle 4} - {\scriptscriptstyle \frac{1}{2}})^{\scriptscriptstyle 3} \right\} \\
&\cdot \sum_{\substack{\ell \, \mid \frac{c_{\scriptscriptstyle 2} c_{\scriptscriptstyle 3}e^{\scriptscriptstyle 2}}{m}
\\ (c_{\scriptscriptstyle 2}, \, \ell) = b}}
\chi_{{\scriptscriptstyle \vartheta} c_{\scriptscriptstyle 1} m \slash e^{\scriptscriptstyle 2}}\!(\ell) \, 
W_{\scriptscriptstyle \ell}(s + s_{\scriptscriptstyle 4} - {\scriptscriptstyle \frac{1}{2}})\,
Z^{(c)}(^{{\scriptscriptstyle \sigma_{\scriptscriptstyle 1}\sigma_{\scriptscriptstyle 2}
\sigma_{\scriptscriptstyle 3}\sigma_{\scriptscriptstyle 4}}} \mathbf{s}; \chi_{{\scriptscriptstyle \vartheta'} 
\!c_{\scriptscriptstyle 2} \ell \slash b^{\scriptscriptstyle 2}}, 
\chi_{{\scriptscriptstyle \vartheta} c_{\scriptscriptstyle 1} m \slash e^{\scriptscriptstyle 2}})
\end{split}
\end{equation*} 
and thus we can write 
\begin{equation*}
\begin{split}
& Z^{(c)}(\mathbf{s};  \chi_{a_{\scriptscriptstyle 2} c_{\scriptscriptstyle 2}}, 
\chi_{c_{\scriptscriptstyle 1}})\,  
= \,  \frac{\chi_{a_{\scriptscriptstyle 2}}\!(c_{\scriptscriptstyle 1})}{4}\,  |c_{\scriptscriptstyle 2}|^{\frac{1}{2} - s_{\scriptscriptstyle 4}} 
\frac{\varphi(c_{\scriptscriptstyle 1} c_{\scriptscriptstyle 3})}{|c_{\scriptscriptstyle 1} c_{\scriptscriptstyle 3}|}
\prod_{p \, \mid \, c_{\scriptscriptstyle 1}  c_{\scriptscriptstyle 3}} \big(\! 1 \, - \, 
|p|^{2 s_{\scriptscriptstyle 4} - 2}\big)^{\scriptscriptstyle -1}\\
&\cdot \sum_{\vartheta, \, \vartheta' \in \{1, \, \theta_{\scriptscriptstyle 0} \}}                
\chi_{{\scriptscriptstyle \vartheta}}(c_{\scriptscriptstyle 2}) \, 
\chi_{{\scriptscriptstyle \vartheta'}}(c_{\scriptscriptstyle 1})
\!\left\{\gamma_{\scriptscriptstyle q}^{+}(s_{\scriptscriptstyle 4}; \, a_{\scriptscriptstyle 2}) \, + \, 
\mathrm{sgn}(\vartheta)\, \gamma_{\scriptscriptstyle q}^{-}(s_{\scriptscriptstyle 4})\right\} 
\left\{\gamma_{\scriptscriptstyle q}^{+}(s + s_{\scriptscriptstyle 4} - {\scriptscriptstyle \frac{1}{2}}; \vartheta)^{\scriptscriptstyle 3} 
+ \,  \mathrm{sgn}(a_{\scriptscriptstyle 2} \vartheta')\,
\gamma_{\scriptscriptstyle q}^{-}(s + s_{\scriptscriptstyle 4} - {\scriptscriptstyle \frac{1}{2}})^{\scriptscriptstyle 3} \right\}\\
&\cdot \sum_{\substack{m \, \mid \, c_{\scriptscriptstyle 1} c_{\scriptscriptstyle 3} \\ (c_{\scriptscriptstyle 1}\!, \, m) = e}} 
\chi_{{\scriptscriptstyle \vartheta'} \!c_{\scriptscriptstyle 2}}\!(m)\, 
U_{\scriptscriptstyle m}(s_{\scriptscriptstyle 4})  
\left|\frac{c_{\scriptscriptstyle 1} m}{e^{\scriptscriptstyle 2}}\right|^{\, 3(1 - s - s_{\scriptscriptstyle 4})}\, 
\frac{\varphi\left(\! \frac{c e^{\scriptscriptstyle 2}}{c_{\scriptscriptstyle 1} m} \!\right)^{\!\! \scriptscriptstyle 3} 
V_{\!\!\!\! \frac{c e^{\scriptscriptstyle 2}}{c_{\scriptscriptstyle 1} m}}\!(s + s_{\scriptscriptstyle 4} - {\scriptscriptstyle \frac{1}{2}})}
{\left| \frac{c e^{\scriptscriptstyle 2}}{c_{\scriptscriptstyle 1} m} \right|^{\scriptscriptstyle 3}} 
\prod_{p \, \mid \frac{c e^{\scriptscriptstyle 2}}{c_{\scriptscriptstyle 1} m}} 
\big(\! 1 \, - \, |p|^{2 s + 2 s_{\scriptscriptstyle 4} - 3}\big)^{\! \scriptscriptstyle - 3} \\
&\cdot \sum_{\substack{\ell \, \mid \frac{c_{\scriptscriptstyle 2} c_{\scriptscriptstyle 3}e^{\scriptscriptstyle 2}}{m}
\\ (c_{\scriptscriptstyle 2}, \, \ell) = b}}
\chi_{{\scriptscriptstyle \vartheta} c_{\scriptscriptstyle 1} m \slash e^{\scriptscriptstyle 2}}\!(\ell) \, 
W_{\scriptscriptstyle \ell}(s + s_{\scriptscriptstyle 4} - {\scriptscriptstyle \frac{1}{2}})\,
Z^{(c)}(^{{\scriptscriptstyle \sigma_{\scriptscriptstyle 1}\sigma_{\scriptscriptstyle 2}
\sigma_{\scriptscriptstyle 3}\sigma_{\scriptscriptstyle 4}}} \mathbf{s}; \chi_{{\scriptscriptstyle \vartheta'} 
\!c_{\scriptscriptstyle 2} \ell \slash b^{\scriptscriptstyle 2}}, 
\chi_{{\scriptscriptstyle \vartheta} c_{\scriptscriptstyle 1} m \slash e^{\scriptscriptstyle 2}}).
\end{split}
\end{equation*} 
The multiple {D}irichlet series 
$
Z^{(c)}(^{{\scriptscriptstyle \sigma_{\scriptscriptstyle 1}\sigma_{\scriptscriptstyle 2}
\sigma_{\scriptscriptstyle 3}\sigma_{\scriptscriptstyle 4}}} \mathbf{s}; \chi_{{\scriptscriptstyle \vartheta'} 
\!c_{\scriptscriptstyle 2} \ell \slash b^{\scriptscriptstyle 2}}, 
\chi_{{\scriptscriptstyle \vartheta} c_{\scriptscriptstyle 1} m \slash e^{\scriptscriptstyle 2}})
$
that contribute to the {\sl residue} we are interested in occur whenever 
$
c_{\scriptscriptstyle 2} \ell = b^{\scriptscriptstyle 2},
$ 
i.e., $\ell = b = c_{\scriptscriptstyle 2}.$ Thus, for $s \in \mathbb{C}$ and $\vartheta' \in \{1, \, \theta_{\scriptscriptstyle 0}\},$ 
we have
\begin{equation*}
\begin{split} 
& \left.\big(1 - \rho(\vartheta')\, q^{(3 - 3 s - 2 s_{\scriptscriptstyle 4})\slash 2} \big)\, 
Z^{(c)}(\mathbf{s};  \chi_{a_{\scriptscriptstyle 2} c_{\scriptscriptstyle 2}}, \chi_{c_{\scriptscriptstyle 1}}\!)
\right\vert_{q^{ - s_{\scriptscriptstyle 4}}  \, =\;  \overline{\rho(\vartheta')}\, q^{3(s - 1)\slash 2}}\\
& = \,  \frac{\chi_{{\scriptscriptstyle \vartheta'} \! a_{\scriptscriptstyle 2} c_{\scriptscriptstyle 2}}\!(c_{\scriptscriptstyle 1})}{8}\,  
|c|^{\, 3(1 - s - s_{\scriptscriptstyle 4})}\, 
|c_{\scriptscriptstyle 2}|^{\frac{1}{2} - s_{\scriptscriptstyle 4}}\,
W_{\! c_{\scriptscriptstyle 2}}\!(s + s_{\scriptscriptstyle 4} - {\scriptscriptstyle \frac{1}{2}})
\frac{\varphi(c_{\scriptscriptstyle 1} c_{\scriptscriptstyle 3})}{|c_{\scriptscriptstyle 1} c_{\scriptscriptstyle 3}|}
\prod_{p \, \mid \, c_{\scriptscriptstyle 1}  c_{\scriptscriptstyle 3}} \big(\! 1 \, - \, 
|p|^{2 s_{\scriptscriptstyle 4} - 2}\big)^{\scriptscriptstyle -1}
\cdot \frac{\zeta^{\scriptscriptstyle (c)}(8 - 6 s - 6 s_{\scriptscriptstyle 4})\, 
\zeta^{\scriptscriptstyle (c)}(3 - 2 s - 2 s_{\scriptscriptstyle 4})^{\scriptscriptstyle 6}}{\zeta_{c}(1)}\\
&\cdot \sum_{\vartheta \in \{1, \, \theta_{\scriptscriptstyle 0} \}}  
\left\{\gamma_{\scriptscriptstyle q}^{+}(s_{\scriptscriptstyle 4}; \, a_{\scriptscriptstyle 2}) \, + \, 
\mathrm{sgn}(\vartheta)\, \gamma_{\scriptscriptstyle q}^{-}(s_{\scriptscriptstyle 4})\right\} 
\left\{\gamma_{\scriptscriptstyle q}^{+}(s + s_{\scriptscriptstyle 4} - {\scriptscriptstyle \frac{1}{2}}; \vartheta)^{\scriptscriptstyle 3} 
+ \,  \mathrm{sgn}(a_{\scriptscriptstyle 2} \vartheta')\,
\gamma_{\scriptscriptstyle q}^{-}(s + s_{\scriptscriptstyle 4} - {\scriptscriptstyle \frac{1}{2}})^{\scriptscriptstyle 3} \right\}\\
&\cdot \sum_{\substack{m \, \mid \, c_{\scriptscriptstyle 1} c_{\scriptscriptstyle 3} \\ (c_{\scriptscriptstyle 1}\!, \, m) = e}} 
\chi_{{\scriptscriptstyle \vartheta'}}\!(m)\, 
U_{\scriptscriptstyle m}(s_{\scriptscriptstyle 4}) 
\left|\frac{c e^{\scriptscriptstyle 2}}{c_{\scriptscriptstyle 1} m}\right|^{\, 3(s + s_{\scriptscriptstyle 4} - 2)}\! 
\varphi\left(\! \frac{c e^{\scriptscriptstyle 2}}{c_{\scriptscriptstyle 1} m} \!\right)^{\!\! \scriptscriptstyle 3} 
V_{\!\!\!\! \frac{c e^{\scriptscriptstyle 2}}{c_{\scriptscriptstyle 1} m}}\!(s + s_{\scriptscriptstyle 4} - {\scriptscriptstyle \frac{1}{2}}) 
\prod_{p \, \mid \frac{c e^{\scriptscriptstyle 2}}{c_{\scriptscriptstyle 1} m}} 
\big(\! 1 \, - \, |p|^{2 s + 2 s_{\scriptscriptstyle 4} - 3}\big)^{\! \scriptscriptstyle - 3}
\end{split}
\end{equation*} 
where $\rho(\vartheta') \in \{\pm 1\}$ if $\vartheta' = 1$ or $\rho(\vartheta') \in \{\pm i\}$ if $\vartheta' = \theta_{\scriptscriptstyle 0}.$ 
Here $s_{\scriptscriptstyle 4}$ is such that $q^{ - s_{\scriptscriptstyle 4}} = \overline{\rho(\vartheta')}\, q^{3(s - 1)\slash 2}.$ 
Letting $S$ temporarily denote the inner sum over $m,$ we can write  
\begin{equation*}
\begin{split}
S \, & = \sum_{e\, \mid \, c_{\scriptscriptstyle 1}} \, \sum_{e' \, \mid \, c_{\scriptscriptstyle 3}}  
\chi_{{\scriptscriptstyle \vartheta'}}\!(e e')\, 
U_{\! e e'}(s_{\scriptscriptstyle 4}) 
\left|\frac{c_{\scriptscriptstyle 3}}{e'} e c_{\scriptscriptstyle 2} \right|^{\, 3(s + s_{\scriptscriptstyle 4} - 2)}\! 
\varphi\left(\! \frac{c_{\scriptscriptstyle 3}}{e'} e c_{\scriptscriptstyle 2}\!\right)^{\!\! \scriptscriptstyle 3} 
V_{\!{\scriptscriptstyle \frac{c_{\scriptscriptstyle 3}}{e'}} e c_{\scriptscriptstyle 2}}\!(s + s_{\scriptscriptstyle 4} - {\scriptscriptstyle \frac{1}{2}}) 
\!\!\prod_{p \, \mid \frac{c_{\scriptscriptstyle 3}}{e'} e c_{\scriptscriptstyle 2}} 
\big(\! 1 \, - \, |p|^{2 s + 2 s_{\scriptscriptstyle 4} - 3}\big)^{\! \scriptscriptstyle - 3}\\
& =\,  |c_{\scriptscriptstyle 2}|^{\, 3(s + s_{\scriptscriptstyle 4} - 2)}
\varphi(c_{\scriptscriptstyle 2})^{\scriptscriptstyle 3} \, 
V_{\! c_{\scriptscriptstyle 2}}\!(s + s_{\scriptscriptstyle 4} - {\scriptscriptstyle \frac{1}{2}}) 
\!\prod_{p \, \mid \, c_{\scriptscriptstyle 2}} 
\big(\! 1 \, - \, |p|^{2 s + 2 s_{\scriptscriptstyle 4} - 3}\big)^{\! \scriptscriptstyle - 3}\\
& \cdot \prod_{p\, \mid \, c_{\scriptscriptstyle 1}} \!\left\{1 +    
\chi_{{\scriptscriptstyle \vartheta'}}\!(p)\, |p|^{\, 3(s + s_{\scriptscriptstyle 4} - 2)}\, \varphi(p)^{\scriptscriptstyle 3}\, 
U_{\! p}(s_{\scriptscriptstyle 4})\,V_{\! p}(s + s_{\scriptscriptstyle 4} - {\scriptscriptstyle \frac{1}{2}}) \, 
\big(\! 1 \, - \, |p|^{2 s + 2 s_{\scriptscriptstyle 4} - 3}\big)^{\! \scriptscriptstyle - 3}\right\}\\
& \cdot \prod_{p \, \mid \, c_{\scriptscriptstyle 3}}\!\!\left\{\chi_{{\scriptscriptstyle \vartheta'}}\!(p)\, 
U_{\! p}(s_{\scriptscriptstyle 4}) \, + \,  |p|^{\, 3(s + s_{\scriptscriptstyle 4} - 2)} \, \varphi(p)^{\scriptscriptstyle 3} \,
V_{\! p}(s + s_{\scriptscriptstyle 4} - {\scriptscriptstyle \frac{1}{2}})\,  
\big(\! 1 \, - \, |p|^{2 s + 2 s_{\scriptscriptstyle 4} - 3}\big)^{\! \scriptscriptstyle - 3} \right\}
\end{split}
\end{equation*}
the products being over monic irreducibles. \!Taking $s = \frac{1}{2}$ and 
$
q^{ - s_{\scriptscriptstyle 4}} \!= \overline{\rho(\vartheta')}\, q^{\scriptscriptstyle - \frac{3}{4}}\!,
$ 
it follows easily from the definitions of $U_{ \scriptscriptstyle \ell}(s), V_{\scriptscriptstyle \ell}(s)$ 
and $W_{\scriptscriptstyle \ell}(s)$ that 
\begin{equation*}
\begin{split}
& \left.\Big(\! 1 - \rho(\vartheta') q^{{\scriptscriptstyle \frac{3}{4} - s_{\scriptscriptstyle 4}}} \!\Big)\, 
Z^{(c)}\big(\tfrac{1}{2}, \tfrac{1}{2}, \tfrac{1}{2}, \, s_{\scriptscriptstyle 4}; 
\chi_{a_{\scriptscriptstyle 2} c_{\scriptscriptstyle 2}}, \chi_{c_{\scriptscriptstyle 1}}\! \big)
\right\vert_{q^{ - s_{\scriptscriptstyle 4}} \, = \; \overline{\rho(\vartheta')}\, q^{\scriptscriptstyle - \frac{3}{4}}}\\
& = \frac{\chi_{a_{\scriptscriptstyle 2} c_{\scriptscriptstyle 2}}\!(c_{\scriptscriptstyle 1})}{8}\, 
\Gamma(a_{\scriptscriptstyle 2}, \vartheta'; \rho(\vartheta'))\, 
L\big(\tfrac{1}{2}, \chi_{{\scriptscriptstyle \vartheta'}}\! \big)^{\scriptscriptstyle 7}\\
& \cdot \;  \rho(\vartheta')^{\deg \, c_{\scriptscriptstyle 1}}\,  |c_{\scriptscriptstyle 1}|^{\scriptscriptstyle -  1 \slash 4} 
\prod_{p \, \mid \, c_{\scriptscriptstyle 1}} 
\!\big(1 - \chi_{{\scriptscriptstyle \vartheta'}}(p) |p|^{\scriptscriptstyle - 1 \slash 2} \big)^{\! \scriptscriptstyle 8}
\big(1 + \chi_{{\scriptscriptstyle \vartheta'}}(p) |p|^{\scriptscriptstyle - 1 \slash 2} \big)^{\! \scriptscriptstyle 2} 
\big(1 + 6 \, \chi_{{\scriptscriptstyle \vartheta'}}(p) |p|^{\scriptscriptstyle - 1 \slash 2}  +   |p|^{\scriptscriptstyle -1} \big)\\
& \cdot \; |c_{\scriptscriptstyle 2}|^{\scriptscriptstyle -  1 \slash 2} \prod_{p \, \mid \, c_{\scriptscriptstyle 2}}
\!\big(1 - \chi_{{\scriptscriptstyle \vartheta'}}(p) |p|^{\scriptscriptstyle - 1 \slash 2} \big)^{\! \scriptscriptstyle 8}
\big(1 + \chi_{{\scriptscriptstyle \vartheta'}}(p) |p|^{\scriptscriptstyle - 1 \slash 2} \big)\, 
\big(3 + 7 \chi_{{\scriptscriptstyle \vartheta'}}(p) |p|^{\scriptscriptstyle - 1 \slash 2} + 3 \, |p|^{\scriptscriptstyle -1} \big)\\
& \cdot \; \prod_{p \, \mid \, c_{\scriptscriptstyle 3}} 
\!\big(1 - \chi_{{\scriptscriptstyle \vartheta'}}(p) |p|^{\scriptscriptstyle - 1 \slash 2} \big)^{\! \scriptscriptstyle 8}
\big(1 + \chi_{{\scriptscriptstyle \vartheta'}}(p) |p|^{\scriptscriptstyle - 1 \slash 2} \big)\, 
\big(1 + 7 \chi_{{\scriptscriptstyle \vartheta'}}(p) |p|^{\scriptscriptstyle - 1 \slash 2} 
+ 13 \, |p|^{\scriptscriptstyle -1} + 7 \chi_{{\scriptscriptstyle \vartheta'}}(p) |p|^{\scriptscriptstyle - 3 \slash 2}  + |p|^{\scriptscriptstyle - 2}\big)
\end{split}
\end{equation*} 
as asserted. 
\end{proof}

In particular, if $c = a_{\scriptscriptstyle 2} = 1,$ we have 
\begin{equation*}
\left.\Big(\! 1 - \rho(\vartheta') q^{{\scriptscriptstyle \frac{3}{4} - s_{\scriptscriptstyle 4}}} \!\Big)\, 
Z\big(\tfrac{1}{2}, \tfrac{1}{2}, \tfrac{1}{2}, \, s_{\scriptscriptstyle 4}; 1, 1 \!\big)
\right\vert_{q^{ - s_{\scriptscriptstyle 4}} \, = \; \overline{\rho(\vartheta')}\, q^{\scriptscriptstyle - \frac{3}{4}}}
= \, \frac{1}{8} \, \Gamma(1, \vartheta'; \rho(\vartheta'))\, 
L\big(\tfrac{1}{2}, \chi_{{\scriptscriptstyle \vartheta'}}\! \big)^{\scriptscriptstyle 7}
\end{equation*} 
equality which can also be verified directly from the explicit expression of $Z(\mathbf{s}; 1, 1)$ given in the second appendix.

\section{Sieving} \label{different-form}
For $h\in \mathbb{F}[x]$ square-free monic and 
$
a_{\scriptscriptstyle 2} \in \{1, \theta_{\scriptscriptstyle 0} \},
$ 
put 
\begin{equation*}
Z(\mathbf{s}, \chi_{a_{\scriptscriptstyle 2}}\!; h)\;\;  = 
\sum_{\substack{m_{\scriptscriptstyle 1}\!, \ldots, \, m_{r}, \, d - \mathrm{monic} \\  
d = d_{\scriptscriptstyle 0}^{} d_{\scriptscriptstyle 1}^{2}, \; d_{\scriptscriptstyle 0}^{} \; \mathrm{sq. \; free} \\ d_{\scriptscriptstyle 1}^{} \equiv \, 0 \!\!\!\pmod h}}  \, 
\frac{\chi_{d_{\scriptscriptstyle 0}}\!(\widehat{m}_{\scriptscriptstyle 1})\, \cdots \, 
\chi_{d_{\scriptscriptstyle 0}}\! (\widehat{m}_{r})\,\chi_{a_{\scriptscriptstyle 2}}\!(d_{\scriptscriptstyle 0})}
{|m_{\scriptscriptstyle 1}|^{s_{\scriptscriptstyle 1}}
\cdots \, |m_{r}|^{s_{r}} |d|^{s_{r + \scriptscriptstyle 1}}} \cdot A(m_{\scriptscriptstyle 1}, \ldots, m_{r}, d).
\end{equation*} 
The series in the right-hand side is absolutely convergent if $s_{\scriptscriptstyle 1}, \ldots, s_{r + \scriptscriptstyle 1}$ 
are complex numbers with sufficiently large real parts. \!Let $\mu(h)$ denote the M\"obius function defined for non-zero polynomials over $\mathbb{F}$ by $\mu(h) = (- 1)^{\scriptscriptstyle \omega(h)}$ if $h$ is square-free, and $h$ is a constant times a product of $\omega(h)$ distinct monic irreducibles, and $\mu(h) = 0$ if $h$ is {\it not} square-free; it is understood that $\mu(h) = 1$ if $h \in \mathbb{F}^{\times}.$ We have the usual property of M\"obius functions:
\begin{equation*}
\sum_{\substack{h \, \mid \, d \\ h - \mathrm{monic}}} \mu(h) = 
\begin{cases}
1 & \mbox{if}  \;\,  \deg \, d = 0 \\ 
0 & \mbox{if}  \;\,  \deg \, d \ge 1.
\end{cases}
\end{equation*} 

We have the following simple lemma:
\vskip10pt
\begin{lem}\label{sieving} --- For 
$
a_{\scriptscriptstyle 2} \in \{1, \theta_{\scriptscriptstyle 0} \}
$ 
and 
$
\mathbf{s} = (s_{\scriptscriptstyle 1}, \ldots, s_{r + \scriptscriptstyle 1}) \in \mathbb{C}^{r + \scriptscriptstyle 1}
$ 
with $\Re(s_{\scriptscriptstyle i})$ sufficiently large, define 
\begin{equation*}
Z_{\scriptscriptstyle 0}(\mathbf{s}, \chi_{a_{\scriptscriptstyle 2}}) \;\; = 
\sum_{d_{\scriptscriptstyle 0} - \mathrm{monic \; \& \; sq. \; free}} \, 
L(s_{\scriptscriptstyle 1}, \chi_{\scriptscriptstyle d_{\scriptscriptstyle 0}}) \, \cdots \,  
L(s_{r}, \chi_{\scriptscriptstyle d_{\scriptscriptstyle 0}})\, 
\chi_{a_{\scriptscriptstyle 2}}\!(d_{\scriptscriptstyle 0})\,  
|d_{\scriptscriptstyle 0}|^{ - s_{r + \scriptscriptstyle 1}}.
\end{equation*} 
Then we have the equality 
\begin{equation} \label{eq: sum-sq-free-vs-MDS}
Z_{\scriptscriptstyle 0}(\mathbf{s}, \chi_{a_{\scriptscriptstyle 2}})\; = 
\sum_{h - \mathrm{monic}}\,  \mu(h) Z(\mathbf{s}, \chi_{a_{\scriptscriptstyle 2}}\!; h).
\end{equation} 
\end{lem} 

\begin{proof} \!The right-hand side of the equality is 
\begin{equation*} 
\sum_{\substack{m_{\scriptscriptstyle 1}\!, \ldots, \, m_{r}, \, d_{\scriptscriptstyle 0}^{} - \mathrm{monic} \\  
d_{\scriptscriptstyle 0}^{} \; \mathrm{sq. \; free}}}  \, 
\frac{\chi_{d_{\scriptscriptstyle 0}}\!(\widehat{m}_{\scriptscriptstyle 1})\, \cdots \, 
\chi_{d_{\scriptscriptstyle 0}}\! (\widehat{m}_{r})\,\chi_{a_{\scriptscriptstyle 2}}\!(d_{\scriptscriptstyle 0})}
{|m_{\scriptscriptstyle 1}|^{s_{\scriptscriptstyle 1}}
\cdots \, |m_{r}|^{s_{r}} |d_{\scriptscriptstyle 0}|^{s_{r + \scriptscriptstyle 1}}} \cdot 
A(m_{\scriptscriptstyle 1}, \ldots, m_{r}, d_{\scriptscriptstyle 0})
\end{equation*} 
where, as before, $\widehat{m}_{\scriptscriptstyle i}$ is the part of $m_{\scriptscriptstyle i}$ coprime to 
$d_{\scriptscriptstyle 0}.$ Recall that the coefficients 
$
A(m_{\scriptscriptstyle 1}, \ldots, m_{r}, d_{\scriptscriptstyle 0})
$ 
are multiplicative and that, for every monic irreducible $p,$ 
$A\big(p^{\, k_{\scriptscriptstyle 1}}\!, \ldots, p^{\, k_{r}}\!\!, p \big) = 0,$
unless $k_{\scriptscriptstyle 1} = \cdots = k_{r} = 0$ in which case 
$A(1, \ldots, 1, p) = 1.$ It follows that the above sum equals   
\begin{equation*} 
\sum_{\substack{m_{\scriptscriptstyle 1}\!, \ldots, \, m_{r}, \, d_{\scriptscriptstyle 0}^{} - \mathrm{monic} \\  
d_{\scriptscriptstyle 0}^{} \; \mathrm{sq. \; free}\\ (m_{\scriptscriptstyle 1} \cdots \, m_{r},  \, d_{\scriptscriptstyle 0}) = 1}}  \, 
\frac{\chi_{d_{\scriptscriptstyle 0}}\!(m_{\scriptscriptstyle 1})\, \cdots \, 
\chi_{d_{\scriptscriptstyle 0}}\! (m_{r})\,\chi_{a_{\scriptscriptstyle 2}}\!(d_{\scriptscriptstyle 0})}
{|m_{\scriptscriptstyle 1}|^{s_{\scriptscriptstyle 1}}
\cdots \, |m_{r}|^{s_{r}} |d_{\scriptscriptstyle 0}|^{s_{r + \scriptscriptstyle 1}}} \cdot 
A(m_{\scriptscriptstyle 1}, \ldots, m_{r}, 1)
\end{equation*} 
and our assertion follows from the fact that
$  
A(m_{\scriptscriptstyle 1}, \ldots, m_{r}, 1) = 1.
$
\end{proof}

We can express the function $Z(\mathbf{s}, \chi_{a_{\scriptscriptstyle 2}}\!; h)$ in terms of the multiple {D}irichlet series 
$
Z^{(c)}(\mathbf{s};  \chi_{a_{\scriptscriptstyle 2} c_{\scriptscriptstyle 2}}, \chi_{a_{\scriptscriptstyle 1} c_{\scriptscriptstyle 1}}), 
$ 
discussed in the previous sections, as follows. \!Let $c$ be a monic divisor of $h,$ and write 
$h = c c'.$ Decompose  
\begin{equation*} 
c = p_{\scriptscriptstyle 1}^{} \cdots \, p_{\scriptscriptstyle l}^{} 
\;\;\, \text{and}\;\;\, 
c' \! = p_{\scriptscriptstyle l + 1}' \cdots \, p_{n}'
\end{equation*} 
into monic irreducibles. \!Consider 
\begin{equation} \label{eq: term}
\frac{\chi_{c d_{\scriptscriptstyle 0}}\!(\widehat{m}_{\scriptscriptstyle 1})\, \cdots \, 
\chi_{c d_{\scriptscriptstyle 0}}\!(\widehat{m}_{r})
\,\chi_{a_{\scriptscriptstyle 2}}\!(c d_{\scriptscriptstyle 0})}{|m_{\scriptscriptstyle 1}|^{s_{\scriptscriptstyle 1}}
\cdots \, |m_{r}|^{s_{r}} |c d_{\scriptscriptstyle 0}^{} d_{\scriptscriptstyle 1}^{\scriptscriptstyle 2}|^{s_{r + \scriptscriptstyle 1}}} \cdot A(m_{\scriptscriptstyle 1}, \ldots, m_{r}, c d_{\scriptscriptstyle 0}^{} d_{\scriptscriptstyle 1}^{\scriptscriptstyle 2})
\qquad \text{(with $d_{\scriptscriptstyle 0}^{}$ square-free coprime to $h$)}
\end{equation} 
representing a term of $Z(\mathbf{s}, \chi_{a_{\scriptscriptstyle 2}}; h).$ From this expression, 
we can factor out a piece corresponding to $h$ (i.e., to $p_{\scriptscriptstyle 1}^{}, \ldots,\\ p_{n}').$ Let 
$
p_{\scriptscriptstyle i}^{\alpha_{\scriptscriptstyle i k}^{}}\!, (p_{\scriptscriptstyle j}')^{\alpha_{\! \scriptscriptstyle j k}'}
\! \parallel m_{\scriptscriptstyle k}^{}
$ 
(for $1 \le i \le l,$ $l + 1 \le j \le n$ and $1 \le k \le r$), and 
$
p_{\scriptscriptstyle i}^{\beta_{\scriptscriptstyle i}^{}}\!, (p_{\scriptscriptstyle j}')^{\beta_{\! \scriptscriptstyle j}'}
\! \parallel c d_{\scriptscriptstyle 1}^{2} 
$ 
with $\beta_{\scriptscriptstyle i} \ge 3$ odd and $\beta_{\scriptscriptstyle j}' \ge 2$ even. Since 
$p_{\scriptscriptstyle i}^{} \mid c d_{\scriptscriptstyle 0}$ (hence 
$p_{\scriptscriptstyle i}^{} \nmid \widehat{m}_{\scriptscriptstyle k}$), we can factor out from \eqref{eq: term} the product
\begin{equation*}
\prod_{i = 1}^{l}
\frac{A(p_{\scriptscriptstyle i}^{\alpha_{\scriptscriptstyle i 1}^{}}\!, \ldots, 
p_{\scriptscriptstyle i}^{\alpha_{\scriptscriptstyle i r}^{}}\!, p_{\scriptscriptstyle i}^{\beta_{\scriptscriptstyle i}^{}})}
{|p_{\scriptscriptstyle i}^{}|^{\alpha_{\scriptscriptstyle i 1} s_{\scriptscriptstyle 1} + \cdots + 
\alpha_{\scriptscriptstyle i r} s_{\scriptscriptstyle r} + \beta_{\scriptscriptstyle i} s_{\scriptscriptstyle r + 1}}} 
\qquad \text{(with $\beta_{\scriptscriptstyle i} \ge 3$ odd)}.
\end{equation*} 
To isolate the remaining irreducibles, we note that 
$
(p_{\scriptscriptstyle j}')^{\alpha_{\! \scriptscriptstyle j k}'}
\! \parallel \widehat{m}_{\scriptscriptstyle k}^{}
$ 
for all $j$ and $k.$ Thus we can also factor out from \eqref{eq: term} the product 
\begin{equation*}
\prod_{j = l + 1}^{n}
\frac{\chi_{c}(p_{\scriptscriptstyle j}')^{\alpha_{\scriptscriptstyle j 1}'^{} + \cdots + \alpha_{\scriptscriptstyle j r}'^{}} 
A((p_{\scriptscriptstyle j}')^{\alpha_{\scriptscriptstyle j 1}'^{}}\!, \ldots, 
(p_{\scriptscriptstyle j}')^{\alpha_{\scriptscriptstyle j r}'^{}}\!, (p_{\scriptscriptstyle j}')^{\beta_{\scriptscriptstyle j}'^{}})}
{|p_{\scriptscriptstyle j}'^{}|^{\alpha_{\scriptscriptstyle j 1}' s_{\scriptscriptstyle 1} + \cdots + 
\alpha_{\scriptscriptstyle j r}' s_{\scriptscriptstyle r} + \beta_{\scriptscriptstyle j}' s_{\scriptscriptstyle r + 1}}} 
\qquad \text{(with $\beta_{\scriptscriptstyle j}' \ge 2$ even)}.
\end{equation*} 
Consequently, we can write the expression \eqref{eq: term} as 
\begin{equation} \label{eq: factored-term}
\begin{split}
&\frac{\chi_{c d_{\scriptscriptstyle 0}}\!(\hat{n}_{\scriptscriptstyle 1})\, \cdots \, 
\chi_{c d_{\scriptscriptstyle 0}}\! (\hat{n}_{r})\,\chi_{a_{\scriptscriptstyle 2}}\!(d_{\scriptscriptstyle 0})
\, A(n_{\scriptscriptstyle 1}, \ldots, n_{r}, d_{\scriptscriptstyle 0}^{} d_{\scriptscriptstyle 1}'^{\, \scriptscriptstyle 2})}
{|n_{\scriptscriptstyle 1}|^{s_{\scriptscriptstyle 1}}\cdots \, |n_{r}|^{s_{r}} 
|d_{\scriptscriptstyle 0}^{} d_{\scriptscriptstyle 1}'^{\, \scriptscriptstyle 2}|^{s_{r + \scriptscriptstyle 1}}}
\prod_{j = l + 1}^{n}
\chi_{d_{\scriptscriptstyle 0}}\!(p_{\scriptscriptstyle j}')^{\alpha_{\scriptscriptstyle j 1}'^{} 
+ \cdots + \alpha_{\scriptscriptstyle j   r}'^{}}\\
& \cdot \chi_{a_{\scriptscriptstyle 2}}\!(c) \prod_{i = 1}^{l}
\frac{A(p_{\scriptscriptstyle i}^{\alpha_{\scriptscriptstyle i 1}^{}}\!, \ldots, 
p_{\scriptscriptstyle i}^{\alpha_{\scriptscriptstyle i r}^{}}\!, p_{\scriptscriptstyle i}^{\beta_{\scriptscriptstyle i}^{}})}
{|p_{\scriptscriptstyle i}^{}|^{\alpha_{\scriptscriptstyle i 1} s_{\scriptscriptstyle 1} + \cdots + 
\alpha_{\scriptscriptstyle i r} s_{\scriptscriptstyle r} + \beta_{\scriptscriptstyle i} s_{\scriptscriptstyle r + 1}}} 
\prod_{j = l + 1}^{n}
\frac{\chi_{c}(p_{\scriptscriptstyle j}')^{\alpha_{\scriptscriptstyle j 1}'^{} + \cdots + \alpha_{\scriptscriptstyle j r}'^{}} 
A((p_{\scriptscriptstyle j}')^{\alpha_{\scriptscriptstyle j 1}'^{}}\!, \ldots, 
(p_{\scriptscriptstyle j}')^{\alpha_{\scriptscriptstyle j r}'^{}}\!, (p_{\scriptscriptstyle j}')^{\beta_{\scriptscriptstyle j}'^{}})}
{|p_{\scriptscriptstyle j}'^{}|^{\alpha_{\scriptscriptstyle j 1}' s_{\scriptscriptstyle 1} + \cdots + 
\alpha_{\scriptscriptstyle j r}' s_{\scriptscriptstyle r} + \beta_{\scriptscriptstyle j}' s_{\scriptscriptstyle r + 1}}}.
\end{split}
\end{equation} 
Here $n_{\scriptscriptstyle 1}, \ldots, n_{r}, \, d_{\scriptscriptstyle 0}^{}, \, d_{\scriptscriptstyle 1}'$ are coprime to $h.$ 
Let $\varepsilon = (\varepsilon_{\! \scriptscriptstyle j})_{\scriptscriptstyle l + 1 \le j \le n}$ with 
$
\varepsilon_{\! \scriptscriptstyle j} \in \{0, 1\}
$ 
be defined by 
\begin{equation*}
\alpha_{\scriptscriptstyle j 1}'^{} + \cdots + \alpha_{\scriptscriptstyle j   r}'^{} \equiv \varepsilon_{\! \scriptscriptstyle j} 
\!\!\!\pmod 2.
\end{equation*} 
If we put 
$
c_{\varepsilon}' = \prod_{\scriptscriptstyle l + 1 \le j \le n}(p_{\scriptscriptstyle j}')^{\varepsilon_{\! \scriptscriptstyle j}},
$ 
we have (by the quadratic reciprocity law) that 
\begin{equation*} 
\prod_{j = l + 1}^{n}
\chi_{d_{\scriptscriptstyle 0}}\!(p_{\scriptscriptstyle j}')^{\alpha_{\scriptscriptstyle j 1}'^{} 
+ \cdots + \alpha_{\scriptscriptstyle j   r}'^{}} 
= \, \chi_{c_{\varepsilon}'}(d_{\scriptscriptstyle 0}).
\end{equation*} 
Accordingly, if we let 
\begin{equation*} 
F\!(z_{\scriptscriptstyle 1}, \ldots, z_{r + \scriptscriptstyle 1}; q) : = 
z_{r + \scriptscriptstyle 1}^{\scriptscriptstyle - 3}\,
f_{\scriptscriptstyle \mathrm{odd}}(z_{\scriptscriptstyle 1}, \ldots, z_{r + \scriptscriptstyle 1}; q) 
\, - \, z_{r + \scriptscriptstyle 1}^{\scriptscriptstyle - 2}
\end{equation*} 
and, for $a\in \{0, 1\},$ 
\begin{equation*} 
\begin{split}
G^{\scriptscriptstyle (a)}(z_{\scriptscriptstyle 1}, \ldots, z_{r + \scriptscriptstyle 1}; q) : & = 
\frac{1}{2}
\left\{f_{\scriptscriptstyle \mathrm{even}}(z_{\scriptscriptstyle 1}, \ldots, z_{\scriptscriptstyle r}, 
z_{r + \scriptscriptstyle 1}; q) 
\, - \, \prod_{k = 1}^{r}(1 - z_{\scriptscriptstyle k})^{\scriptscriptstyle -1} \right\}z_{r + \scriptscriptstyle 1}^{\scriptscriptstyle - 2} \\
& + \frac{(- 1)^{a + 2}}{2}
\left\{f_{\scriptscriptstyle \mathrm{even}}(- \, z_{\scriptscriptstyle 1}, \ldots, - \, z_{\scriptscriptstyle r}, 
z_{r + \scriptscriptstyle 1}; q) 
\, - \, \prod_{k = 1}^{r}(1 + z_{\scriptscriptstyle k})^{\scriptscriptstyle -1} \right\}z_{r + \scriptscriptstyle 1}^{\scriptscriptstyle - 2}
\end{split}
\end{equation*} 
with $f_{\scriptscriptstyle \mathrm{odd}}$ and $f_{\scriptscriptstyle \mathrm{even}}$ defined in Appendix \ref{Appendix B}, 
we obtain the key equality:  
\begin{equation} \label{eq: fundamental-eq-h}
\begin{split}
Z(\mathbf{s}, \chi_{a_{\scriptscriptstyle 2}}; h) = \,
|h|^{- 2 s_{r + \scriptscriptstyle 1}} \!\!\sum_{h = c c'} \;
\sum_{\varepsilon = (\varepsilon_{\scriptscriptstyle p'})_{p' \, \mid \, c'}} 
\, Z^{(h)}(\mathbf{s};  \chi_{a_{\scriptscriptstyle 2} c_{\varepsilon}'}, \chi_{c}) 
& \prod_{p \, \mid \, c}
F(|p|^{\, - s_{\scriptscriptstyle 1}}\!, \ldots, \, |p|^{\, - s_{r + \scriptscriptstyle 1}}; \, 
|p|)\,  |p|^{\, - s_{r + \scriptscriptstyle 1}}\\
& \cdot \chi_{a_{\scriptscriptstyle 2} c_{\varepsilon}'}(c)
\prod_{p' \, \mid \, c'} G^{(\varepsilon_{\scriptscriptstyle p'})}
(|p'|^{\, - s_{\scriptscriptstyle 1}}\!, \ldots, \, |p'|^{\, - s_{r + \scriptscriptstyle 1}}; \, |p'|).
\end{split}
\end{equation} 
Notice that the right-hand side gives the analytic continuation of 
$Z(\mathbf{s}, \chi_{a_{\scriptscriptstyle 2}}; h).$ Our main goal is to show that, for 
$\mathbf{s} = \big(\tfrac{1}{2},$ $\tfrac{1}{2}, \tfrac{1}{2}, w\big),$ 
the series obtained by substituting \eqref{eq: fundamental-eq-h} into \eqref{eq: sum-sq-free-vs-MDS} converges absolutely 
and uniformly on every compact subset of the half-plane $\Re(w) > 2 \slash 3,$ away from the points $w \in \mathbb{C}$ for 
which $q^{- w} = \pm \, q^{\scriptscriptstyle - 1}\!,$ or 
$q^{- w} = \pm \, q^{\scriptscriptstyle - 3\slash 4}\!, \, \pm\,  i \, q^{\scriptscriptstyle - 3\slash 4}.$

\section{Estimates} \label{Estim}
To prove Theorem \ref{Main Theorem A}, we will use \eqref{eq: fundamental-eq-h} in conjunction with the estimates 
of the local parts of the untwisted multiple {D}irichlet series $Z(\mathbf{s}; 1, 1)$ provided by the following elementary lemmas.
\vskip10pt
\begin{lem}\label{Estimate Zloc} --- For $|z| \le q^{\scriptscriptstyle - \frac{1}{2}}\!,$ we have the asymptotics 
\begin{equation*} 
F\big(q^{\scriptscriptstyle -\frac{1}{2}}\!, \, q^{\scriptscriptstyle -\frac{1}{2}}\!, \, 
q^{\scriptscriptstyle -\frac{1}{2}}\!, \, z  ; \, q \big) = 14 \, + \,  q \, z^{\scriptscriptstyle 2} \, + \, O(z^{\scriptscriptstyle 2})
\end{equation*} 
\phantom{}
\begin{equation*} 
G^{\scriptscriptstyle (0)}
\big(q^{\scriptscriptstyle -\frac{1}{2}}\!, \, q^{\scriptscriptstyle -\frac{1}{2}}\!, \, q^{\scriptscriptstyle -\frac{1}{2}}\!, \, z  ; \, q \big) 
= 14 \, + \,  q \, z^{\scriptscriptstyle 2} \, + \, O(q^{\scriptscriptstyle -1}),
\end{equation*} 
and the estimate
\begin{equation*} 
G^{\scriptscriptstyle (1)}
\big(q^{\scriptscriptstyle -\frac{1}{2}}\!, \, q^{\scriptscriptstyle -\frac{1}{2}}\!, \, q^{\scriptscriptstyle -\frac{1}{2}}\!, \, z  ; \, q \big) 
= O\big(q^{\scriptscriptstyle - \frac{1}{2}}\big)
\end{equation*}
the implied constants in the $O$-symbols being independent on $z$ and $q.$
\end{lem}

\begin{proof} \!Using the formulas in the Appendix \ref{Appendix B} and the definitions of 
$F$ and $G^{\scriptscriptstyle (a)}$ ($a\in \{0, 1\}$), one finds that 
\begin{equation*} 
F\big(q^{\scriptscriptstyle -\frac{1}{2}}\!, \, q^{\scriptscriptstyle -\frac{1}{2}}\!, \, 
q^{\scriptscriptstyle -\frac{1}{2}}\!, \, z  ; \, q \big) 
\, = \, \frac{1 + 7 z^{\scriptscriptstyle 2} + 7 z^{\scriptscriptstyle 4} + z^{\scriptscriptstyle 6}}
{(1 - z^{\scriptscriptstyle 2})^{\scriptscriptstyle 7}  (1 - q \, z^{\scriptscriptstyle 4}) \, z^{\scriptscriptstyle 2}} 
\, - \, \frac{1}{z^{\scriptscriptstyle 2}}
\end{equation*} 
\begin{equation*} 
\begin{split}
& G^{\scriptscriptstyle (0)}
\big(q^{\scriptscriptstyle -\frac{1}{2}}\!, \, q^{\scriptscriptstyle -\frac{1}{2}}\!, \, q^{\scriptscriptstyle -\frac{1}{2}}\!, \, z; \, q \big)\\
& = \, \frac{1}{2}\big(1 - q^{\scriptscriptstyle - \frac{1}{2}} \big)^{\scriptscriptstyle - 3} \cdot
\frac{1 + \big(7  - 14 \, q^{\scriptscriptstyle - \frac{1}{2}} + 6 \, q^{\scriptscriptstyle - 1} 
- q^{\scriptscriptstyle - \frac{3}{2}} \big)\, z^{\scriptscriptstyle 2} + 
7\big(1 - 4\, q^{\scriptscriptstyle - \frac{1}{2}} + 4\, q^{\scriptscriptstyle -1} 
- q^{\scriptscriptstyle - \frac{3}{2}} \big)\, z^{\scriptscriptstyle 4} 
+ \big(1 - 6 \, q^{\scriptscriptstyle - \frac{1}{2}} + 14\, q^{\scriptscriptstyle -1} 
- 7\, q^{\scriptscriptstyle - \frac{3}{2}} \big)\, z^{\scriptscriptstyle 6} 
- q^{\scriptscriptstyle - \frac{3}{2}} z^{\scriptscriptstyle 8}}
{(1 - z^{\scriptscriptstyle 2})^{\scriptscriptstyle 7} (1 - q\,  z^{\scriptscriptstyle 4}) \, z^{\scriptscriptstyle 2}}\\
& + \, \frac{1}{2}\big(1 + q^{\scriptscriptstyle - \frac{1}{2}} \big)^{\scriptscriptstyle - 3} \cdot
\frac{1 + \big(7  + 14 \, q^{\scriptscriptstyle - \frac{1}{2}} + 6 \, q^{\scriptscriptstyle - 1} 
+ q^{\scriptscriptstyle - \frac{3}{2}} \big)\, z^{\scriptscriptstyle 2} + 
7\big(1 + 4\, q^{\scriptscriptstyle - \frac{1}{2}} + 4\, q^{\scriptscriptstyle -1} 
+ q^{\scriptscriptstyle - \frac{3}{2}} \big)\, z^{\scriptscriptstyle 4} 
+ \big(1 + 6 \, q^{\scriptscriptstyle - \frac{1}{2}} + 14\, q^{\scriptscriptstyle -1} 
+ 7\, q^{\scriptscriptstyle - \frac{3}{2}} \big)\, z^{\scriptscriptstyle 6} 
+ q^{\scriptscriptstyle - \frac{3}{2}} z^{\scriptscriptstyle 8}}
{(1 - z^{\scriptscriptstyle 2})^{\scriptscriptstyle 7} (1 - q\,  z^{\scriptscriptstyle 4}) \, z^{\scriptscriptstyle 2}}\\
& - \, \left\{\big(1 - q^{\scriptscriptstyle - \frac{1}{2}} \big)^{\scriptscriptstyle - 3} 
+ \;  \big(1 + q^{\scriptscriptstyle - \frac{1}{2}} \big)^{\scriptscriptstyle - 3} \right\} \frac{1}{2\, z^{\scriptscriptstyle 2}}
\end{split}
\end{equation*} 
and 
\begin{equation*} 
\begin{split}
& G^{\scriptscriptstyle (1)}
\big(q^{\scriptscriptstyle -\frac{1}{2}}\!, \, q^{\scriptscriptstyle -\frac{1}{2}}\!, \, q^{\scriptscriptstyle -\frac{1}{2}}\!, \, 
z; \, q \big) \\
& = \, \frac{1}{2}\big(1 - q^{\scriptscriptstyle - \frac{1}{2}} \big)^{\scriptscriptstyle - 3} \cdot
\frac{1 + \big(7  - 14 \, q^{\scriptscriptstyle - \frac{1}{2}} + 6 \, q^{\scriptscriptstyle - 1} 
- q^{\scriptscriptstyle - \frac{3}{2}} \big)\, z^{\scriptscriptstyle 2} + 
7\big(1 - 4\, q^{\scriptscriptstyle - \frac{1}{2}} + 4\, q^{\scriptscriptstyle -1} 
- q^{\scriptscriptstyle - \frac{3}{2}} \big)\, z^{\scriptscriptstyle 4} 
+ \big(1 - 6 \, q^{\scriptscriptstyle - \frac{1}{2}} + 14\, q^{\scriptscriptstyle -1} 
- 7\, q^{\scriptscriptstyle - \frac{3}{2}} \big)\, z^{\scriptscriptstyle 6} 
- q^{\scriptscriptstyle - \frac{3}{2}} z^{\scriptscriptstyle 8}}
{(1 - z^{\scriptscriptstyle 2})^{\scriptscriptstyle 7} (1 - q\,  z^{\scriptscriptstyle 4}) \, z^{\scriptscriptstyle 2}}\\
& - \, \frac{1}{2}\big(1 + q^{\scriptscriptstyle - \frac{1}{2}} \big)^{\scriptscriptstyle - 3} \cdot
\frac{1 + \big(7  + 14 \, q^{\scriptscriptstyle - \frac{1}{2}} + 6 \, q^{\scriptscriptstyle - 1} 
+ q^{\scriptscriptstyle - \frac{3}{2}} \big)\, z^{\scriptscriptstyle 2} + 
7\big(1 + 4\, q^{\scriptscriptstyle - \frac{1}{2}} + 4\, q^{\scriptscriptstyle -1} 
+ q^{\scriptscriptstyle - \frac{3}{2}} \big)\, z^{\scriptscriptstyle 4} 
+ \big(1 + 6 \, q^{\scriptscriptstyle - \frac{1}{2}} + 14\, q^{\scriptscriptstyle -1} 
+ 7\, q^{\scriptscriptstyle - \frac{3}{2}} \big)\, z^{\scriptscriptstyle 6} 
+ q^{\scriptscriptstyle - \frac{3}{2}} z^{\scriptscriptstyle 8}}
{(1 - z^{\scriptscriptstyle 2})^{\scriptscriptstyle 7} (1 - q\,  z^{\scriptscriptstyle 4}) \, z^{\scriptscriptstyle 2}}\\
& - \, \left\{\big(1 - q^{\scriptscriptstyle - \frac{1}{2}} \big)^{\scriptscriptstyle - 3} 
- \;  \big(1 + q^{\scriptscriptstyle - \frac{1}{2}} \big)^{\scriptscriptstyle - 3} \right\} \frac{1}{2 \, z^{\scriptscriptstyle 2}}.
\end{split}
\end{equation*}

From this explicit formulas we see easily that 
\begin{equation*} 
\left|F\big(q^{\scriptscriptstyle -\frac{1}{2}}\!, \, q^{\scriptscriptstyle -\frac{1}{2}}\!, \, 
q^{\scriptscriptstyle -\frac{1}{2}}\!, \, z; \, q \big) - 14  - q \, z^{\scriptscriptstyle 2}\right| \, \le  \, 
\left(\frac{15 \, q^{\scriptscriptstyle -7} + 119 \, q^{\scriptscriptstyle -6} + 412 \, q^{\scriptscriptstyle -5} 
+ 812 \, q^{\scriptscriptstyle - 4} + 994 \, q^{\scriptscriptstyle -3}  + 770 \, q^{\scriptscriptstyle -2} 
+ 363 \, q^{\scriptscriptstyle -1} + 99}
{(1 - q^{\scriptscriptstyle -1})^{8}}\right)\, |z|^{\scriptscriptstyle 2}
\end{equation*} 
\begin{equation*} 
\begin{split}
& \hskip-51pt \left|G^{\scriptscriptstyle (0)}
\big(q^{\scriptscriptstyle -\frac{1}{2}}\!, \, q^{\scriptscriptstyle -\frac{1}{2}}\!, \, q^{\scriptscriptstyle -\frac{1}{2}}\!, 
\, z; \, q \big) - 14  - q \, z^{\scriptscriptstyle 2}\right|\\ 
& \le \, \frac{(1 + q^{\scriptscriptstyle -1})^{\scriptscriptstyle 3}(15 \, q^{\scriptscriptstyle -7} + 120 \, q^{\scriptscriptstyle - 6} 
+ 420 \, q^{\scriptscriptstyle -5} + 843 \, q^{\scriptscriptstyle -4} + 1064 \, q^{\scriptscriptstyle -3} 
+ 866 \, q^{\scriptscriptstyle -2} + 427 \, q^{\scriptscriptstyle -1} + 153)}
{q \, (1- q^{\scriptscriptstyle -1})^{\scriptscriptstyle 11}}
\end{split}
\end{equation*} 
and
\begin{equation*} 
\hskip-58pt \left|G^{\scriptscriptstyle (1)}\big(q^{\scriptscriptstyle -\frac{1}{2}}\!, \, q^{\scriptscriptstyle -\frac{1}{2}}\!, 
\, q^{\scriptscriptstyle -\frac{1}{2}}\!, \, z  ; \, q \big)\right| \, \le  \, 
\left(\frac{q^{\scriptscriptstyle -7} + 10\, q^{\scriptscriptstyle -6} + 36\, q^{\scriptscriptstyle -5} + 65\, q^{\scriptscriptstyle -4} 
+ 121\, q^{\scriptscriptstyle -3} + 134\, q^{\scriptscriptstyle -2} + 70\, q^{\scriptscriptstyle -1} + 31}
{(1- q^{\scriptscriptstyle -1})^{\scriptscriptstyle 10}}\right)\frac{1}{\sqrt{q}}
\end{equation*} 
from which the lemma follows.
\end{proof} 
The estimates in the above lemma show that there is additional decay in \eqref{eq: fundamental-eq-h} in the conductors 
of the characters $\chi_{c}$ and $\chi_{a_{\scriptscriptstyle 2} c_{\varepsilon}'}.$ However, this is not sufficient, as 
we would need enough decay in $|h|.$

To this end, define 
\begin{equation*} 
f_{\scriptscriptstyle \mathrm{even}}^{\, \pm}(z_{\scriptscriptstyle 1}, \ldots, z_{r + \scriptscriptstyle 1}; q) \, = \,
(f_{\scriptscriptstyle \mathrm{even}}(z_{\scriptscriptstyle 1}, \ldots, z_{\scriptscriptstyle r}, z_{r + \scriptscriptstyle 1}; q) 
\pm  f_{\scriptscriptstyle \mathrm{even}}(- \, z_{\scriptscriptstyle 1}, \ldots, - \, z_{\scriptscriptstyle r}, z_{r + \scriptscriptstyle 1}; q))\slash 2 \qquad \text{(with $r = 3$)}
\end{equation*} 
where $f_{\scriptscriptstyle \mathrm{even}}$ is as defined in Appendix \ref{Appendix B}. We will use the next lemma and 
an inductive argument to improve the convex bound \eqref{eq: basic-initial-estimate}, precisely in the 
$c_{\scriptscriptstyle 3}$-aspect.

\vskip10pt
\begin{lem}\label{Estimate inverse even part zeven} --- For every real $q \ge 5$ and 
$|z| \le q^{\scriptscriptstyle - \frac{1}{2}}\!,$ 
we have the estimates 
\begin{equation*} 
\big| f_{\scriptscriptstyle \mathrm{odd}}\big(q^{\scriptscriptstyle -\frac{1}{2}}\!, \, q^{\scriptscriptstyle -\frac{1}{2}}\!, \, 
q^{\scriptscriptstyle -\frac{1}{2}}\!, \, z  ; \, q \big) \big| < 17 |z|
\end{equation*} 
\phantom{}
\begin{equation*} 
\big| f_{\scriptscriptstyle \mathrm{even}}^{\, -}\big(q^{\scriptscriptstyle -\frac{1}{2}}\!, \, q^{\scriptscriptstyle -\frac{1}{2}}\!, \, 
q^{\scriptscriptstyle -\frac{1}{2}}\!, \, z  ; \, q \big) \big | <  58\, q^{\scriptscriptstyle -\frac{1}{2}}
\end{equation*} 
and if $q\equiv 1 \!\pmod 4$ is a prime power, we have the inequality
\begin{equation*} 
\frac{1}{\big| f_{\scriptscriptstyle \mathrm{even}}^{\, +}\big(q^{\scriptscriptstyle -\frac{1}{2}}\!, \, 
q^{\scriptscriptstyle -\frac{1}{2}}\!, \, q^{\scriptscriptstyle -\frac{1}{2}}\!, \, z  ; \, q \big) \big|} < 20.
\end{equation*} 
\end{lem}

\begin{proof} \!We have the explicit expressions:  
\begin{equation*}
f_{\scriptscriptstyle \mathrm{odd}}\big(q^{\scriptscriptstyle -\frac{1}{2}}\!, \, q^{\scriptscriptstyle -\frac{1}{2}}\!, \, 
q^{\scriptscriptstyle -\frac{1}{2}}\!, \, z  ; \, q \big) \, = \, 
\frac{z \, (1 + 7 z^{\scriptscriptstyle 2} + 7 z^{\scriptscriptstyle 4} + z^{\scriptscriptstyle 6})}
{(1 - z^{\scriptscriptstyle 2})^{\scriptscriptstyle 7}  (1 - q \, z^{\scriptscriptstyle 4})} 
\end{equation*} 
\phantom{} 
\begin{equation*}
f_{\scriptscriptstyle \mathrm{even}}^{\, -}\big(q^{\scriptscriptstyle -\frac{1}{2}}\!, \, q^{\scriptscriptstyle -\frac{1}{2}}\!, \, 
q^{\scriptscriptstyle -\frac{1}{2}}\!, \, z  ; \, q \big) \, = \, 
\frac{3 + q^{\scriptscriptstyle - 1} + (10 - 17\, q^{\scriptscriptstyle - 1} \! + 3 \, q^{\scriptscriptstyle - 2})\, z^{\scriptscriptstyle 2}  
+ (3  - 17\, q^{\scriptscriptstyle - 1} \! + 10 \, q^{\scriptscriptstyle - 2})\, z^{\scriptscriptstyle 4}  
+ (q^{\scriptscriptstyle  -1} \! + 3 \, q^{\scriptscriptstyle - 2})\,  z^{\scriptscriptstyle 6}}{\sqrt{q}\, 
(1 - q^{\scriptscriptstyle - 1})^{\scriptscriptstyle 3}(1 - z^{\scriptscriptstyle 2})^{\scriptscriptstyle 6} 
(1 - q\,  z^{\scriptscriptstyle 4})}
\end{equation*} 
and 
\begin{equation*}
\begin{split}
& 1\slash f_{\scriptscriptstyle \mathrm{even}}^{\, +}\big(q^{\scriptscriptstyle -\frac{1}{2}}\!, \, 
q^{\scriptscriptstyle -\frac{1}{2}}\!, \, q^{\scriptscriptstyle -\frac{1}{2}}\!, \, z  ; \, q \big)  \\ 
& = \frac{(1 - q^{\scriptscriptstyle - 1})^{\scriptscriptstyle 3}(1 - z^{\scriptscriptstyle 2})^{\scriptscriptstyle 7} 
(1 - q\,  z^{\scriptscriptstyle 4})}
{1 + 3\, q^{\scriptscriptstyle - 1} + (7 - 15\, q^{\scriptscriptstyle - 1} \! + q^{\scriptscriptstyle - 2} 
- q^{\scriptscriptstyle - 3})\, z^{\scriptscriptstyle 2} + (7  - 35\, q^{\scriptscriptstyle - 1} + 35 \, q^{\scriptscriptstyle - 2} 
- 7 \, q^{\scriptscriptstyle - 3})\, z^{\scriptscriptstyle 4} + (1 - q^{\scriptscriptstyle  -1} + 15 \, q^{\scriptscriptstyle - 2} - 7 \, q^{\scriptscriptstyle - 3})\,  z^{\scriptscriptstyle 6} - (3 \, q^{\scriptscriptstyle - 2} + q^{\scriptscriptstyle - 3})
\,  z^{\scriptscriptstyle 8}}.
\end{split}
\end{equation*} 
It follows that 
\begin{equation*} 
\big|f_{\scriptscriptstyle \mathrm{odd}}\big(q^{\scriptscriptstyle -\frac{1}{2}}\!, \, q^{\scriptscriptstyle -\frac{1}{2}}\!, \, 
q^{\scriptscriptstyle -\frac{1}{2}}\!, \, z  ; \, q \big) \big| \, \le \, 
\frac{1 + 7 |z|^{\scriptscriptstyle 2} + 7 |z|^{\scriptscriptstyle 4} + |z|^{\scriptscriptstyle 6}}
{(1 - |z|^{\scriptscriptstyle 2})^{\scriptscriptstyle 7}  (1 - q \, |z|^{\scriptscriptstyle 4})}  \cdot |z| \, \le \, 
\frac{1 + 7 \, q^{\scriptscriptstyle - 1} + 7 \, q^{\scriptscriptstyle - 2} + q^{\scriptscriptstyle - 3}}
{(1 - q^{\scriptscriptstyle - 1})^{\scriptscriptstyle 8}}  \cdot |z|.
\end{equation*} 
The expression 
\begin{equation*}  
\frac{1 + 7 \, q^{\scriptscriptstyle - 1} + 7 \, q^{\scriptscriptstyle - 2} + q^{\scriptscriptstyle - 3}}
{(1 - q^{\scriptscriptstyle - 1})^{\scriptscriptstyle 8}} \qquad \text{(for $q \ge 5$)}
\end{equation*} 
is increasing as a function of $q^{\scriptscriptstyle - 1}\!,$ and its value when $q = 5$ is $16.0217... < 17.$ 
We have similarly    
\begin{equation*} 
\big| f_{\scriptscriptstyle \mathrm{even}}^{\, -}\big(q^{\scriptscriptstyle -\frac{1}{2}}\!, \, q^{\scriptscriptstyle -\frac{1}{2}}\!, \, 
q^{\scriptscriptstyle -\frac{1}{2}}\!, \, z  ; \, q \big) \big| 
\, \le \, \frac{3 + 11\, q^{\scriptscriptstyle - 1} + 20\, q^{\scriptscriptstyle - 2} + 20\, q^{\scriptscriptstyle - 3}  
+ 11\, q^{\scriptscriptstyle - 4} + 3\, q^{\scriptscriptstyle - 5}}
{(1 - q^{\scriptscriptstyle - 1})^{\scriptscriptstyle 10}} \cdot q^{\scriptscriptstyle - \frac{1}{2}} <  
58\, q^{\scriptscriptstyle - \frac{1}{2}}
\end{equation*} 
as we had asserted.

Now the numerator of 
$
1\slash \big| f_{\scriptscriptstyle \mathrm{even}}^{\, +}\big(q^{\scriptscriptstyle -\frac{1}{2}}\!, \, 
q^{\scriptscriptstyle -\frac{1}{2}}\!, \, q^{\scriptscriptstyle -\frac{1}{2}}\!, \, z  ; \, q \big) \big|
$ 
is 
\begin{equation*}
(1 - q^{\scriptscriptstyle - 1})^{\scriptscriptstyle 3}\, |1 - z^{\scriptscriptstyle 2}|^{\scriptscriptstyle 7} \, 
|1 - q\,  z^{\scriptscriptstyle 4}| < (1 + q^{\scriptscriptstyle - 1})^{\scriptscriptstyle 8} \le 
(6\slash 5)^{\scriptscriptstyle 8}. 
\end{equation*} 
To obtain a lower bound for the denominator, we first assume that $q \ge 9.$ In this case we have 
\begin{equation*} 
\begin{split} 
& \big| 1 + 3\, q^{\scriptscriptstyle - 1} + (7 - 15\, q^{\scriptscriptstyle - 1} + q^{\scriptscriptstyle - 2} 
- q^{\scriptscriptstyle - 3})\, z^{\scriptscriptstyle 2} + (7  - 35\, q^{\scriptscriptstyle - 1} + 35 \, q^{\scriptscriptstyle - 2} 
- 7 \, q^{\scriptscriptstyle - 3})\, z^{\scriptscriptstyle 4} + (1 - q^{\scriptscriptstyle  -1} + 15 \, q^{\scriptscriptstyle - 2} - 7 \, q^{\scriptscriptstyle - 3})\,  z^{\scriptscriptstyle 6} - (3 \, q^{\scriptscriptstyle - 2} + q^{\scriptscriptstyle - 3})
\,  z^{\scriptscriptstyle 8} \big| \\ 
& \ge 1 + 3\, q^{\scriptscriptstyle - 1} - \big|(7 - 15\, q^{\scriptscriptstyle - 1} + q^{\scriptscriptstyle - 2} 
- q^{\scriptscriptstyle - 3})\, z^{\scriptscriptstyle 2} + (7  - 35\, q^{\scriptscriptstyle - 1} + 35 \, q^{\scriptscriptstyle - 2} 
- 7 \, q^{\scriptscriptstyle - 3})\, z^{\scriptscriptstyle 4} + (1 - q^{\scriptscriptstyle  -1} + 15 \, q^{\scriptscriptstyle - 2} - 7 \, q^{\scriptscriptstyle - 3})\,  z^{\scriptscriptstyle 6} - (3 \, q^{\scriptscriptstyle - 2} + q^{\scriptscriptstyle - 3})
\,  z^{\scriptscriptstyle 8} \big| \\
& \ge 1 - 4\, q^{\scriptscriptstyle - 1} - 22\, q^{\scriptscriptstyle - 2} - 37\, q^{\scriptscriptstyle - 3} - 
37\, q^{\scriptscriptstyle - 4} - 22\, q^{\scriptscriptstyle - 5} - 10\, q^{\scriptscriptstyle - 6} - q^{\scriptscriptstyle - 7} 
> 2 \slash 9.
\end{split}
\end{equation*} 

When $q = 5$ we have   
\begin{equation*} 
\left|\frac{8}{5} - \frac{8 \, z^{\scriptscriptstyle 2}}{125} 
(2 z^{\scriptscriptstyle 6} - 21 z^{\scriptscriptstyle 4} - 21 z^{\scriptscriptstyle 2} - 63) \right| \ge 
\frac{8}{5} - \frac{8 \, |z|^{\scriptscriptstyle 2}}{125} 
\big| 2 z^{\scriptscriptstyle 6} - 21 z^{\scriptscriptstyle 4} - 21 z^{\scriptscriptstyle 2} - 63 \big| 
\ge \frac{8}{5} - \frac{8}{625}\bigg(63 + \frac{21}{5} + \frac{21}{25} + \frac{2}{125}\bigg) > 
\frac{2}{9}.
\end{equation*} 
The last assertion follows from these inequalities.  
\end{proof}

For ease of notation, we let 
$
\mathscr{Z}^{(c)}(w; \chi_{a_{\scriptscriptstyle 2} c_{\scriptscriptstyle 2}}, \chi_{a_{\scriptscriptstyle 1} c_{\scriptscriptstyle 1}}) = 
Z^{(c)}\big(\tfrac{1}{2}, \tfrac{1}{2}, \tfrac{1}{2}, w; \chi_{a_{\scriptscriptstyle 2} c_{\scriptscriptstyle 2}}, \chi_{a_{\scriptscriptstyle 1}c_{\scriptscriptstyle 1}}\big)
$ 
with $a_{\scriptscriptstyle 1}, \,  a_{\scriptscriptstyle 2} \in \{1, \theta_{\scriptscriptstyle 0} \}.$ 
\vskip10pt
\begin{prop}\label{key-proposition} --- Let $c_{\scriptscriptstyle 1}, c_{\scriptscriptstyle 2}$ and 
$c_{\scriptscriptstyle 3}$ be monic polynomials such that 
$
c = c_{\scriptscriptstyle 1}c_{\scriptscriptstyle 2}c_{\scriptscriptstyle 3}
$ 
is square-free, \!and let $\omega(c)$ denote the number of irreducible factors of 
$c.$ If we define 
\begin{equation*} 
\tilde{\mathscr{Z}}^{(c)}
(w; \chi_{a_{\scriptscriptstyle 2}  c_{\scriptscriptstyle 2}}, \chi_{a_{\scriptscriptstyle 1} c_{\scriptscriptstyle 1}}) =
\big(1\, - \, q^{3 - 4w} \big)
\big(1 \, - \, q^{2 - 2w} \big)^{\! \scriptscriptstyle 7} 
\mathscr{Z}^{(c)}
(w; \chi_{a_{\scriptscriptstyle 2} c_{\scriptscriptstyle 2}}, \chi_{a_{\scriptscriptstyle 1}  c_{\scriptscriptstyle 1}})  
\qquad \text{$a_{\scriptscriptstyle 1}, \, a_{\scriptscriptstyle 2} \in \{1, \theta_{\scriptscriptstyle 0} \}$}
\end{equation*} 
then, for every $\delta > 0,$ we have the estimate 
\begin{equation}  \label{eq: fundamental-estimate-optimized} 
\tilde{\mathscr{Z}}^{(c)}
(w; \chi_{a_{\scriptscriptstyle 2} c_{\scriptscriptstyle 2}}, \chi_{a_{\scriptscriptstyle 1}  c_{\scriptscriptstyle 1}})  
\, \ll_{\scriptscriptstyle \delta, \, q} \, A_{0}^{\omega(c_{\scriptscriptstyle 1}c_{\scriptscriptstyle 2})} A_{1}^{\omega(c_{\scriptscriptstyle 3})}
|c_{\scriptscriptstyle 1}|^{\scriptscriptstyle 3 (1 - \Re(w)) + \delta} 
\, |c_{\scriptscriptstyle 2}|^{\scriptscriptstyle \frac{5}{2}(1 - \Re(w)) + \delta} 
\, |c_{\scriptscriptstyle 3}|^{\scriptscriptstyle \mathrm{max}\left\{3 \, - \, 4 \Re(w), \, 2 \, - \frac{5 \Re(w)}{2}\right\} + \delta}
\end{equation} 
with $A_{0} \!= \!20^{\, \scriptscriptstyle 9}$ \!and $A_{1}\! = 20 + 1500 \cdot 20^{\, \scriptscriptstyle 9}\!,$ for all $w$ with 
$\frac{1}{2} \le \Re(w) \le \frac{4}{5}.$
\end{prop}

\begin{proof} \!\!We proceed by induction on $\omega(c_{\scriptscriptstyle 3}).$ If $c_{\scriptscriptstyle 3} = 1,$ 
our estimate was established in \eqref{eq: basic-initial-estimate}; \!in other words, \!for every 
$\delta > 0,$ $c_{\scriptscriptstyle 1}, c_{\scriptscriptstyle 2}$ monics such that 
$
c_{\scriptscriptstyle 1}c_{\scriptscriptstyle 2}
$ 
is square-free, and $w$ with $\frac{1}{2} \le \Re(w) \le \frac{4}{5},$ we have 
\begin{equation*}
\big|\tilde{\mathscr{Z}}^{(c_{\scriptscriptstyle 1} c_{\scriptscriptstyle 2})}
(w; \chi_{a_{\scriptscriptstyle 2} c_{\scriptscriptstyle 2}}, \chi_{a_{\scriptscriptstyle 1} c_{\scriptscriptstyle 1}})\big|
\, \le  \, B(\delta, q) \, 20^{\, {\scriptscriptstyle 9} \, \omega(c_{\scriptscriptstyle 1}c_{\scriptscriptstyle 2})}\,
|c_{\scriptscriptstyle 2}|^{\scriptscriptstyle \frac{5}{2}(1 - \Re(w)) + \delta} 
\, |c_{\scriptscriptstyle 1}|^{\, \scriptscriptstyle 3 (1 - \Re(w)) + \delta}
\end{equation*} 
for some positive constant $B(\delta, q).$

Let $c_{\scriptscriptstyle 1}, c_{\scriptscriptstyle 2}$ \!and $c_{\scriptscriptstyle 3}$ be monic polynomials with  
$
c_{\scriptscriptstyle 1}c_{\scriptscriptstyle 2}c_{\scriptscriptstyle 3}
$ 
square-free. \!For 
$
\mathbf{s} = (s_{\scriptscriptstyle 1}, \ldots, s_{\scriptscriptstyle 4}) \in \mathbb{C}^{\scriptscriptstyle 4}
$ 
with $\Re(s_{\scriptscriptstyle i})$ sufficiently large, consider the multiple {D}irichlet series
\begin{equation*} 
Z^{(c_{\scriptscriptstyle 1}c_{\scriptscriptstyle 2}c_{\scriptscriptstyle 3})}(\mathbf{s};  
\chi_{a_{\scriptscriptstyle 2} c_{\scriptscriptstyle 2}}, \chi_{a_{\scriptscriptstyle 1} c_{\scriptscriptstyle 1}}) 
\;\; = \sum_{\substack{m_{\scriptscriptstyle 1}\!,\,  m_{\scriptscriptstyle 2}, \, m_{\scriptscriptstyle 3}, \, d - \mathrm{monic} \\  
d = d_{\scriptscriptstyle 0}^{} d_{\scriptscriptstyle 1}^{2}, \; d_{\scriptscriptstyle 0}^{} \; \mathrm{sq. \; free} \\ 
(m_{\scriptscriptstyle 1} m_{\scriptscriptstyle 2} m_{\scriptscriptstyle 3} d, \, c_{\scriptscriptstyle 1}c_{\scriptscriptstyle 2}c_{\scriptscriptstyle 3}) = 1}}  \, \frac{\chi_{a_{\scriptscriptstyle 1} c_{\scriptscriptstyle 1} d_{\scriptscriptstyle 0}}\!(\widehat{m}_{\scriptscriptstyle 1}) \, 
\chi_{a_{\scriptscriptstyle 1} c_{\scriptscriptstyle 1} d_{\scriptscriptstyle 0}}\!(\widehat{m}_{\scriptscriptstyle 2}) \, 
\chi_{a_{\scriptscriptstyle 1} c_{\scriptscriptstyle 1}d_{\scriptscriptstyle 0}}\!(\widehat{m}_{\scriptscriptstyle 3})\,
\chi_{a_{\scriptscriptstyle 2} c_{\scriptscriptstyle 2}}\!(d_{\scriptscriptstyle 0})}
{|m_{\scriptscriptstyle 1}|^{s_{\scriptscriptstyle 1}} |m_{\scriptscriptstyle 2}|^{s_{\scriptscriptstyle 2}}
|m_{\scriptscriptstyle 3}|^{s_{\scriptscriptstyle 3}} |d|^{s_{\scriptscriptstyle 4}}} \cdot 
A(m_{\scriptscriptstyle 1}, m_{\scriptscriptstyle 2}, m_{\scriptscriptstyle 3}, d).
\end{equation*} 
If $p \in \mathbb{F}[x]$ is a monic irreducible, 
$p \nmid c_{\scriptscriptstyle 1}c_{\scriptscriptstyle 2}c_{\scriptscriptstyle 3},$ we can write (as before): 
\begin{equation*} 
\begin{split}
& Z^{(c_{\scriptscriptstyle 1}c_{\scriptscriptstyle 2}c_{\scriptscriptstyle 3})}(\mathbf{s}; 
\chi_{a_{\scriptscriptstyle 2} c_{\scriptscriptstyle 2}}, \chi_{a_{\scriptscriptstyle 1} c_{\scriptscriptstyle 1}})
= \chi_{a_{\scriptscriptstyle 2} c_{\scriptscriptstyle 2}}\!(p)\,
Z^{(c_{\scriptscriptstyle 1}p \, c_{\scriptscriptstyle 2}c_{\scriptscriptstyle 3})}
(\mathbf{s}; \chi_{a_{\scriptscriptstyle 2} c_{\scriptscriptstyle 2}}, \chi_{a_{\scriptscriptstyle 1} c_{\scriptscriptstyle 1}p}) 
\, f_{\scriptscriptstyle \mathrm{odd}}(\mathbf{t}^{\deg \, p}\!, \, t_{\scriptscriptstyle 4}^{\deg \, p}; \, q^{\deg \, p})\\ 
& + \, 
\frac{\chi_{a_{\scriptscriptstyle 1} c_{\scriptscriptstyle 1}}\!(p)}{2}\,
Z^{(c_{\scriptscriptstyle 1}c_{\scriptscriptstyle 2}p \, c_{\scriptscriptstyle 3})}
(\mathbf{s};  \chi_{a_{\scriptscriptstyle 2} c_{\scriptscriptstyle 2}p}, \chi_{a_{\scriptscriptstyle 1} c_{\scriptscriptstyle 1}})
\left(f_{\scriptscriptstyle \mathrm{even}}(\mathbf{t}^{\deg \, p}\!, \, t_{\scriptscriptstyle 4}^{\deg \, p}; \, q^{\deg \, p}) \, - \,
f_{\scriptscriptstyle \mathrm{even}}(- \, \mathbf{t}^{\deg \, p}\!, \, t_{\scriptscriptstyle 4}^{\deg \, p}; \, q^{\deg \, p}) \right)\\ 
& + \, \frac{1}{2}\, Z^{(c_{\scriptscriptstyle 1} c_{\scriptscriptstyle 2} c_{\scriptscriptstyle 3} p)}
(\mathbf{s};  \chi_{a_{\scriptscriptstyle 2} c_{\scriptscriptstyle 2}}, \chi_{a_{\scriptscriptstyle 1} c_{\scriptscriptstyle 1}})
\left(f_{\scriptscriptstyle \mathrm{even}}(\mathbf{t}^{\deg \, p}\!, \, t_{\scriptscriptstyle 4}^{\deg \, p}; \, q^{\deg \, p}) \, + \,
f_{\scriptscriptstyle \mathrm{even}}(- \, \mathbf{t}^{\deg \, p}\!, \, t_{\scriptscriptstyle 4}^{\deg \, p}; \, q^{\deg \, p}) \right)
\end{split}
\end{equation*} 
where we set $t_{\scriptscriptstyle i} = q^{-s_{\scriptscriptstyle i}}\!,$ and $\pm \, \mathbf{t}^{\deg \, p}$ \!stands for 
$
\big(\! \pm t_{\scriptscriptstyle 1}^{\deg \, p}\!, \pm \, t_{\scriptscriptstyle 2}^{\deg \, p}\!, \pm \, t_{\scriptscriptstyle 3}^{\deg \, p}\big).
$ 
Setting $s_{\scriptscriptstyle i} = \frac{1}{2}$ for $i = 1, 2, 3$ and $s_{\scriptscriptstyle 4} = w,$ 
we have by analytic continuation that 
\begin{equation*} 
\begin{split}
\tilde{\mathscr{Z}}^{(c_{\scriptscriptstyle 1} c_{\scriptscriptstyle 2} c_{\scriptscriptstyle 3})}
(w; \chi_{a_{\scriptscriptstyle 2} c_{\scriptscriptstyle 2}}, \chi_{a_{\scriptscriptstyle 1} c_{\scriptscriptstyle 1}}) 
& = \chi_{a_{\scriptscriptstyle 2} c_{\scriptscriptstyle 2}}\!(p)\, 
\tilde{\mathscr{Z}}^{(c_{\scriptscriptstyle 1}p \, c_{\scriptscriptstyle 2} c_{\scriptscriptstyle 3})}
(w;  \chi_{a_{\scriptscriptstyle 2} c_{\scriptscriptstyle 2}}, 
\chi_{a_{\scriptscriptstyle 1} c_{\scriptscriptstyle 1}p})\, \mathbf{f}_{\scriptscriptstyle p}(w) 
\,  + \, \chi_{a_{\scriptscriptstyle 1} c_{\scriptscriptstyle 1}}\!(p)\, 
\tilde{\mathscr{Z}}^{(c_{\scriptscriptstyle 1} c_{\scriptscriptstyle 2}p \, 
c_{\scriptscriptstyle 3})}(w;  \chi_{a_{\scriptscriptstyle 2} c_{\scriptscriptstyle 2} p}, 
\chi_{a_{\scriptscriptstyle 1} c_{\scriptscriptstyle 1}})\, 
\mathbf{f}_{\scriptscriptstyle p}^{\, -}(w)\\ 
& + \, \tilde{\mathscr{Z}}^{(c_{\scriptscriptstyle 1} c_{\scriptscriptstyle 2} c_{\scriptscriptstyle 3} p)}
(w;  \chi_{a_{\scriptscriptstyle 2} c_{\scriptscriptstyle 2}}, \chi_{a_{\scriptscriptstyle 1} c_{\scriptscriptstyle 1}})\, 
\mathbf{f}_{\scriptscriptstyle p}^{\, +}(w).
\end{split}
\end{equation*} 
Here 
\begin{equation*} 
\mathbf{f}_{\scriptscriptstyle p}(w) : = 
f_{\scriptscriptstyle \mathrm{odd}}
\big(|p|^{\scriptscriptstyle - \frac{1}{2}}\!, \, |p|^{\scriptscriptstyle - \frac{1}{2}}\!, \, |p|^{\scriptscriptstyle - \frac{1}{2}}\!, 
\, |p|^{- w}; \, |p| \big) \;\;\;  \text{and} \;\;\; 
\mathbf{f}_{\scriptscriptstyle p}^{\, \pm}(w) : = f_{\scriptscriptstyle \mathrm{even}}^{\pm}
\big(|p|^{\scriptscriptstyle - \frac{1}{2}}\!,  \, |p|^{\scriptscriptstyle - \frac{1}{2}}\!, \, |p|^{\scriptscriptstyle - \frac{1}{2}}\!, 
\, |p|^{- w}; \, |p| \big).
\end{equation*} 
Applying the inequalities in Lemma \ref{Estimate inverse even part zeven} to 
$|\mathbf{f}_{\scriptscriptstyle p}(w)|$ and $|\mathbf{f}_{\scriptscriptstyle p}^{\, \pm}(w)|,$
it follows that, for $\Re(w) \ge \frac{1}{2},$ 
\begin{equation*}
\begin{split} 
\big|\tilde{\mathscr{Z}}^{(c_{\scriptscriptstyle 1} c_{\scriptscriptstyle 2} c_{\scriptscriptstyle 3} p)}
(w;  \chi_{a_{\scriptscriptstyle 2} c_{\scriptscriptstyle 2}}, 
\chi_{a_{\scriptscriptstyle 1} c_{\scriptscriptstyle 1}})\big|  
& < 20\, \big|\tilde{\mathscr{Z}}^{(c_{\scriptscriptstyle 1} c_{\scriptscriptstyle 2} c_{\scriptscriptstyle 3})}
(w; \chi_{a_{\scriptscriptstyle 2} c_{\scriptscriptstyle 2}}, \chi_{a_{\scriptscriptstyle 1} c_{\scriptscriptstyle 1}}) \big| + 
17\cdot 20\, \big|\tilde{\mathscr{Z}}^{(c_{\scriptscriptstyle 1}p \, c_{\scriptscriptstyle 2} c_{\scriptscriptstyle 3})}
(w;  \chi_{a_{\scriptscriptstyle 2} c_{\scriptscriptstyle 2}}, 
\chi_{a_{\scriptscriptstyle 1} c_{\scriptscriptstyle 1}p})\big|\, |p|^{\scriptscriptstyle - \Re(w)} \\
& + 58\cdot 20\, \big|\tilde{\mathscr{Z}}^{(c_{\scriptscriptstyle 1} c_{\scriptscriptstyle 2}p \, 
c_{\scriptscriptstyle 3})}(w;  \chi_{a_{\scriptscriptstyle 2} c_{\scriptscriptstyle 2}p}, 
\chi_{a_{\scriptscriptstyle 1} c_{\scriptscriptstyle 1}})\big|\, 
|p|^{\scriptscriptstyle -\frac{1}{2}}.
\end{split}
\end{equation*} 
Let $K(c_{\scriptscriptstyle 1}, c_{\scriptscriptstyle 2}, c_{\scriptscriptstyle 3}, w, \delta, q)$ denote the right-hand side of 
\eqref{eq: fundamental-estimate-optimized}, i.e., 
\begin{equation*}
K(c_{\scriptscriptstyle 1}, c_{\scriptscriptstyle 2}, c_{\scriptscriptstyle 3}, w, \delta, q) 
= B(\delta, q) 
\, A_{0}^{\omega(c_{\scriptscriptstyle 1}c_{\scriptscriptstyle 2})} A_{1}^{\omega(c_{\scriptscriptstyle 3})}
|c_{\scriptscriptstyle 1}|^{\scriptscriptstyle 3 (1 - \Re(w)) + \delta} 
\, |c_{\scriptscriptstyle 2}|^{\scriptscriptstyle \frac{5}{2}(1 - \Re(w)) + \delta} 
\, |c_{\scriptscriptstyle 3}|^{\scriptscriptstyle \mathrm{max}\left\{3 \, - \, 4 \Re(w), \, 2 \, - \frac{5 \Re(w)}{2}\right\} + \delta}.
\end{equation*} 
Taking $w$ such that $\frac{1}{2} \le \Re(w) \le \frac{4}{5},$ we have by the induction hypothesis 
\begin{equation*} 
\begin{split}
\big|\tilde{\mathscr{Z}}^{(c p)}(w;  \chi_{a_{\scriptscriptstyle 2} c_{\scriptscriptstyle 2}}, 
\chi_{a_{\scriptscriptstyle 1} c_{\scriptscriptstyle 1}})\big| \, & < \, 
K(c_{\scriptscriptstyle 1}, c_{\scriptscriptstyle 2}, c_{\scriptscriptstyle 3}, w, \delta, q)  
\cdot \Big(20 + 340 \, A_{0} \, |p|^{\scriptscriptstyle 3 \, - \, 4 \Re(w) + \delta} + 1160 \, A_{0} \, |p|^{\scriptscriptstyle 2 \, - \frac{5\Re(w)}{2} + \, \delta}\Big)\\
& < \, K(c_{\scriptscriptstyle 1}, c_{\scriptscriptstyle 2}, c_{\scriptscriptstyle 3}p, w, \delta, q) 
\end{split}
\end{equation*} 
and the proposition follows.
\end{proof}

Using the last proposition, we can now estimate the function
$ 
\mathscr{Z}(w, \chi_{a_{\scriptscriptstyle 2}}\!; h) : = Z\big(\tfrac{1}{2}, \tfrac{1}{2}, \tfrac{1}{2}, w, \chi_{a_{\scriptscriptstyle 2}}; h\big).
$

\vskip10pt
\begin{thm}\label{Main-h-Estimate} --- For $h\in \mathbb{F}[x]$ square-free monic and 
$
a_{\scriptscriptstyle 2} \in \{1, \theta_{\scriptscriptstyle 0} \},
$ 
put 
\begin{equation*}
\tilde{\mathscr{Z}}(w, \chi_{a_{\scriptscriptstyle 2}}\!; h)  = 
\big(1\, - \, q^{3 - 4 w} \big) \big(1 \, - \, q^{2 - 2 w} \big)^{\! \scriptscriptstyle 7}\, 
\mathscr{Z}(w, \chi_{a_{\scriptscriptstyle 2}}\!; h).
\end{equation*} 
Then, for every $\delta > 0,$ \!we have 
\begin{equation} \label{eq: basic-estimate-zh}
\tilde{\mathscr{Z}}(w, \chi_{a_{\scriptscriptstyle 2}}\!; h)\, \ll_{\delta, \, q} \, A^{\omega(h)}  \, 
|h|^{\scriptscriptstyle 2 \, - \frac{9 \Re(w)}{2} + \, \delta} 
\end{equation} 
on the strip $\frac{2}{3} \le \Re(w) \le \frac{4}{5},$ and 
\begin{equation*} 
|h|^{2 w}\,  \tilde{\mathscr{Z}}(w, \chi_{a_{\scriptscriptstyle 2}}\!; h)\, \ll_{\delta, \, q} \, 
A^{\omega(h)}  \, |h|^{\scriptscriptstyle \delta} 
\end{equation*} 
on the strip $\frac{4}{5} \le \Re(w) \le 1+ \delta,$ where $A$ is an explicitly computable constant. 
\end{thm}

\begin{proof} \!By \eqref{eq: fundamental-eq-h} we have 
\begin{equation*} 
\begin{split}
|\tilde{\mathscr{Z}}(w, \chi_{a_{\scriptscriptstyle 2}}\!; h)| \, \le \, 
|h|^{- 2 \Re(w)} \!\sum_{h = c c'} \; \sum_{\varepsilon = (\varepsilon_{\scriptscriptstyle p'})_{p' \, \mid \, c'}} 
\, \big|\tilde{\mathscr{Z}}^{(h)}(w; \chi_{a_{\scriptscriptstyle 2} c_{\varepsilon}'}, \chi_{c}) \big|\, 
&\prod_{p \, \mid \, c}\, \big|F(|p|^{\, - \frac{1}{2}}\!, \ldots, \, |p|^{\, - w}; \, 
|p|)\big|\;  |p|^{\, \scriptscriptstyle - \, \Re(w)}\\
& \cdot \prod_{p' \, \mid \, c'} \big|G^{(\varepsilon_{\scriptscriptstyle p'})}
(|p'|^{\, - \frac{1}{2}}\!, \ldots, \, |p'|^{\, - w}; \, |p'|)\big|.
\end{split}
\end{equation*} 
It follows from Proposition \ref{key-proposition} and Lemma \ref{Estimate Zloc} that, for every $\delta > 0$ 
and $w \in \mathbb{C}$ with $\frac{2}{3} \le \Re(w) \le \frac{4}{5},$
\begin{equation*} 
\tilde{\mathscr{Z}}(w, \chi_{a_{\scriptscriptstyle 2}}\!; h)\, \ll_{\delta,\,  q} \, B^{\omega(h)}  \, 
|h|^{\scriptscriptstyle 2 \, - \frac{9 \Re(w)}{2} + \, \delta} \!
\underbrace{\sum_{h = c c'} \;
\sum_{\varepsilon = (\varepsilon_{\scriptscriptstyle p'})_{p' \, \mid \, c'}}  1}_{\textrm{\textcolor{MyDarkBlue}
{${\scriptstyle{= 3^{\omega(h)}}}$}}} \, 
\ll_{\delta, \, q} \, (3B)^{\omega(h)}  \, |h|^{\scriptscriptstyle 2 \, - \frac{9 \Re(w)}{2} + \, \delta}
\end{equation*} 
for some explicitly computable constant $B.$ 
In particular, if $\Re(w) = \frac{4}{5},$ \!we have  
\begin{equation*} 
|h|^{2 w}\,  \tilde{\mathscr{Z}}(w, \chi_{a_{\scriptscriptstyle 2}}\!; h)\,  \ll_{\delta, \, q} \, (3B)^{\omega(h)}  \, 
|h|^{\scriptscriptstyle \delta}.
\end{equation*}

On the other hand, if $\Re(w) = 1 + \delta$ we have by \eqref{eq: EstimateZ1} that 
\begin{equation*} 
|h|^{2 w}\,  \tilde{\mathscr{Z}}(w, \chi_{a_{\scriptscriptstyle 2}}\!; h)\,  \ll_{\delta, \, q} \,  
(11\,  B)^{\omega(h)}  \, |h|^{\scriptscriptstyle \delta}.
\end{equation*} 
The function 
$
|h|^{2 w}\,  \tilde{\mathscr{Z}}(w, \chi_{a_{\scriptscriptstyle 2}}\!; h)
$ 
is holomorphic on an open neighborhood of the strip $\frac{4}{5} \le \Re(w) \le 1 + \delta.$ This function is also 
of finite order on the strip and thus, the second estimate follows from the Phragmen-Lindel\"of principle. 

This completes the proof of the theorem. 
\end{proof}

To establish the analytic continuation of 
$
Z_{\scriptscriptstyle 0}\big(\tfrac{1}{2}, \tfrac{1}{2}, \tfrac{1}{2}, w, \, \chi_{a_{\scriptscriptstyle 2}}\big)
$ 
to the half-plane $\Re(w) > \frac{2}{3},$ \!we shall need the following elementary lemma:

\vskip10pt
\begin{lem}\label{Elem-Lemma} --- For any $A > 1,$ the {D}irichlet series 
\begin{equation*} 
\mathscr{D}_{A}(s) : = \sum_{h - \mathrm{monic \; \& \; sq. \; free}} 
A^{\omega(h)}\, |h|^{\scriptscriptstyle - s}
\end{equation*} 
is absolutely convergent in the half-plane $\Re(s) > 1.$
\end{lem} 

\begin{proof} \!First the series is absolutely convergent for $\Re(s)$ sufficiently large. To see this, 
choose $n \ge 1$ such that $A < q^{n}.$ Since $\omega(h) \le \deg \, h$ for any square-free 
polynomial, we have $A^{\omega(h)} \le |h|^{n}.$ Thus, for $\Re(s) = \sigma > n + 1,$
\begin{equation*} 
\sum_{\substack{h - \mathrm{monic \; \& \; sq. \; free}\\ \deg h\,  \le \, k}} A^{\omega(h)}\, |h|^{\scriptscriptstyle - \sigma} \, 
 <  \sum_{h - \mathrm{monic}} |h|^{\scriptscriptstyle n - \sigma} 
\end{equation*} 
the last series being obviously convergent. 

Now $\mathscr{D}_{A}(s)$ has the Euler product expression 
\begin{equation*} 
\mathscr{D}_{A}(s) = \prod_{m = 1}^{\infty} \, (1 + A \, q^{\scriptscriptstyle - m s})^{\mathrm{Irr}_{\scriptscriptstyle q}(m)} \; 
\qquad \; \text{(for $\Re(s) > n + 1$)}
\end{equation*} 
where $\mathrm{Irr}_{\scriptscriptstyle q}(m)$ is the number of (monic) irreducible polynomials of degree $m$ over 
$\mathbb{F}.$ From the well-known formula 
$
\sum_{d \mid m} \, d \, \mathrm{Irr}_{\scriptscriptstyle q}(d) = q^{\scriptscriptstyle m}
$ 
for $m \ge 1,$ we have 
$
\mathrm{Irr}_{\scriptscriptstyle q}(m) \le q^{\scriptscriptstyle m} \slash m.
$ 
By using this estimate and the familiar inequality $\log(1 + y) < y$ for $y > 0,$ we have, 
for $s = \sigma > n + 1,$ 
\begin{equation*} 
\log \, \mathscr{D}_{A}(\sigma) = \sum_{m = 1}^{\infty} \, \mathrm{Irr}_{\scriptscriptstyle q}(m)
\log(1 + A \, q^{\scriptscriptstyle - m \sigma}) \, < \, 
A\sum_{m = 1}^{\infty} \, \frac{q^{\scriptscriptstyle m(1 - \sigma)}}{m}.
\end{equation*} 
Thus the Euler product expression of $\mathscr{D}_{A}(\sigma)$ converges when $\sigma > 1,$ 
from which the lemma follows.
\end{proof}

\section{Proof of Theorem \ref{Main Theorem A}} \label{proof Thm A}
The function 
$ 
\tilde{\mathscr{Z}}_{\scriptscriptstyle 0}(w, \chi_{a_{\scriptscriptstyle 2}}) : = 
\big(1\, - \, q^{3 - 4w} \big) \big(1 \, - \, q^{2 - 2 w} \big)^{\! \scriptscriptstyle 7} 
Z_{\scriptscriptstyle 0}\big(\tfrac{1}{2}, \tfrac{1}{2}, \tfrac{1}{2}, w, \, \chi_{a_{\scriptscriptstyle 2}}\big)
$ 
is holomorphic in the half-plane $\Re(w) > 1,$ and in this region, we have 
(by Lemma \ref{sieving} and analytic continuation) that
\begin{equation} \label{eq: sqfree-non-sqfree1/2}
\tilde{\mathscr{Z}}_{\scriptscriptstyle 0}(w, \chi_{a_{\scriptscriptstyle 2}})\; = \sum_{h - \mathrm{monic}}\,  \mu(h) 
\tilde{\mathscr{Z}}(w, \chi_{a_{\scriptscriptstyle 2}}\!; h).
\end{equation} 
By Theorem \ref{Main-h-Estimate} and Lemma \ref{Elem-Lemma}, the series in the right-hand side 
converges uniformly on every compact subset of the half-plane $\Re(w) > 2 \slash 3,$ and the 
meromorphic continuation of 
$ 
Z_{\scriptscriptstyle 0}\big(\tfrac{1}{2}, \tfrac{1}{2}, \tfrac{1}{2}, w, \, \chi_{a_{\scriptscriptstyle 2}}\big)
$ 
now follows from Weierstrass Theorem. \!The values $w \in \mathbb{C}$ for which $q^{- w} = \pm \, q^{\scriptscriptstyle - 1}\!,$ or
$q^{- w} = \pm \, q^{\scriptscriptstyle - 3\slash 4}\!, \, \pm\,  i \, q^{\scriptscriptstyle - 3\slash 4}$ are the only possible 
poles of this function. \!The principal parts of 
$ 
Z_{\scriptscriptstyle 0}\big(\tfrac{1}{2}, \tfrac{1}{2}, \tfrac{1}{2}, w, \, \chi_{a_{\scriptscriptstyle 2}}\big)
$ 
at $q^{- w} = \pm \, q^{\scriptscriptstyle - 1}$ can be computed following the arguments in \cite[Section~3.2]{DGH}. 

To compute the residues at the remaining poles, fix $\vartheta' \in \{1, \, \theta_{\scriptscriptstyle 0} \},$ and let 
$\rho(\vartheta')$ be such that $\rho(\vartheta') \in \{\pm 1\}$ if $\vartheta' \! = 1$ or $\rho(\vartheta') \in \{\pm i\}$ if 
$\vartheta' \!= \theta_{\scriptscriptstyle 0}.$ Letting
\begin{equation*}
\begin{split}
P(x) & = (1 - x)^{\scriptscriptstyle 5} (1 + x)  (1 + 4 \, x + 11 x^{\scriptscriptstyle 2} + 10 \, x^{\scriptscriptstyle 3} 
- 11 x^{\scriptscriptstyle 4} + 11 x^{\scriptscriptstyle 6} - 4 \, x^{\scriptscriptstyle 7} - x^{\scriptscriptstyle 8}) \\
& = 1 - 14 \, x^{\scriptscriptstyle 3} - x^{\scriptscriptstyle 4} + 78 \, x^{\scriptscriptstyle 5} + \cdots
\end{split}
\end{equation*} 
we have that 
\begin{equation} \label{eq: generic-residue-3/4}
\left.\Big(\! 1 - \rho(\vartheta') q^{{\scriptscriptstyle \frac{3}{4}} - w} \!\Big)\, 
Z_{\scriptscriptstyle 0}\big(\tfrac{1}{2}, \tfrac{1}{2}, \tfrac{1}{2}, w, \, \chi_{a_{\scriptscriptstyle 2}}\big)\right\vert_{q^{ - w} \, = \; \overline{\rho(\vartheta')}\, 
q^{\scriptscriptstyle - \frac{3}{4}}} = \, \frac{1}{8}\,  \Gamma(a_{\scriptscriptstyle 2}, \vartheta'; \rho(\vartheta'))\, 
L\big(\tfrac{1}{2}, \chi_{{\scriptscriptstyle \vartheta'}}\big)^{\! \scriptscriptstyle 7}
\cdot \prod_{p}  P\left(\frac{\chi_{{\scriptscriptstyle \vartheta'}}(p)}{\sqrt{|p|}}\right).
\end{equation} 
the product in the right-hand side being over all monic irreducibles in $\mathbb{F}[x].$ Note that $P(x)$ is precisely the polynomial appearing 
in the analogous calculation of Zhang \cite{Zha} in the context of the cubic moment of quadratic {D}irichlet {$L$}-series over the rationals.

To justify \eqref{eq: generic-residue-3/4}, we first apply Proposition \ref{MDS-residue-3/4} and \eqref{eq: fundamental-eq-h}. Indeed, let 
\begin{equation*}
\begin{split}
&\mathscr{P}_{\! \scriptscriptstyle 1}(c) = \rho(\vartheta')^{\deg \, c}\,  |c|^{\scriptscriptstyle -  1 \slash 4} 
\prod_{p \, \mid \, c} 
\!\big(1 - \chi_{{\scriptscriptstyle \vartheta'}}(p) |p|^{\scriptscriptstyle - 1 \slash 2} \big)^{\! \scriptscriptstyle 8}
\big(1 + \chi_{{\scriptscriptstyle \vartheta'}}(p) |p|^{\scriptscriptstyle - 1 \slash 2} \big)^{\! \scriptscriptstyle 2} 
\big(1 + 6 \, \chi_{{\scriptscriptstyle \vartheta'}}(p) |p|^{\scriptscriptstyle - 1 \slash 2}  +   |p|^{\scriptscriptstyle -1} \big) \\
&\mathscr{P}_{\! \scriptscriptstyle 2}(c_{\varepsilon}') = |c_{\varepsilon}'|^{\scriptscriptstyle -  1 \slash 2} 
\prod_{p \, \mid \, c_{\varepsilon}'}\!\big(1 - \chi_{{\scriptscriptstyle \vartheta'}}(p) |p|^{\scriptscriptstyle - 1 \slash 2} \big)^{\! \scriptscriptstyle 8}
\big(1 + \chi_{{\scriptscriptstyle \vartheta'}}(p) |p|^{\scriptscriptstyle - 1 \slash 2} \big)\, 
\big(3 + 7 \chi_{{\scriptscriptstyle \vartheta'}}(p) |p|^{\scriptscriptstyle - 1 \slash 2} + 3 \, |p|^{\scriptscriptstyle -1} \big) \\
&\mathscr{P}_{\! \scriptscriptstyle 3}\left(\frac{c'}{c_{\varepsilon}'} \right) = \prod_{p \, \mid \, \frac{c'}{c_{\varepsilon}'}} 
\!\big(1 - \chi_{{\scriptscriptstyle \vartheta'}}(p) |p|^{\scriptscriptstyle - 1 \slash 2} \big)^{\! \scriptscriptstyle 8}
\big(1 + \chi_{{\scriptscriptstyle \vartheta'}}(p) |p|^{\scriptscriptstyle - 1 \slash 2} \big)\, 
\big(1 + 7 \chi_{{\scriptscriptstyle \vartheta'}}(p) |p|^{\scriptscriptstyle - 1 \slash 2} 
+ 13 \, |p|^{\scriptscriptstyle -1} + 7 \chi_{{\scriptscriptstyle \vartheta'}}(p) |p|^{\scriptscriptstyle - 3 \slash 2}  + |p|^{\scriptscriptstyle - 2}\big).
\end{split}
\end{equation*} 
In \eqref{eq: fundamental-eq-h} set $s_{\scriptscriptstyle 1} \! = s_{\scriptscriptstyle 2} \! = s_{\scriptscriptstyle 3} \! = \!\frac{1}{2}$ and 
$s_{\scriptscriptstyle 4} \! = w.$ Multiplying the resulting equality by $1 - \rho(\vartheta') q^{{\scriptscriptstyle \frac{3}{4}} - w}$ and then 
taking the value $q^{ - w} \! = \overline{\rho(\vartheta')}\, q^{\scriptscriptstyle - \frac{3}{4}},$ it follows from Proposition \ref{MDS-residue-3/4} that   
\begin{equation*}
\begin{split} 
&\left.\Big(\! 1 - \rho(\vartheta') q^{{\scriptscriptstyle \frac{3}{4}} - w} \!\Big)\, 
Z\big(\tfrac{1}{2}, \tfrac{1}{2}, \tfrac{1}{2}, w, \, \chi_{a_{\scriptscriptstyle 2}}; h\big)
\right\vert_{q^{ - w} \, = \; \overline{\rho(\vartheta')}\, q^{\scriptscriptstyle - \frac{3}{4}}}\\ 
& = \frac{1}{8}\,  \Gamma(a_{\scriptscriptstyle 2}, \vartheta'; \rho(\vartheta'))\, 
L\big(\tfrac{1}{2}, \chi_{{\scriptscriptstyle \vartheta'}}\!\big)^{\! \scriptscriptstyle 7} \\
& \cdot  \chi_{{\scriptscriptstyle \vartheta'}}(h)\, |h|^{\scriptscriptstyle - 3\slash 2}
\!\sum_{h = c c'}  \mathscr{P}_{\! \scriptscriptstyle 1}(c) 
\prod_{p \, \mid \, c} F(|p|^{\, \scriptscriptstyle - 1\slash 2}\!, \, |p|^{\, \scriptscriptstyle - 1\slash 2}\!, \,
|p|^{\, \scriptscriptstyle - 1\slash 2}\!, \, \overline{\rho(\vartheta')}^{\, \deg \, p} |p|^{\scriptscriptstyle - 3\slash 4}; \, |p|) 
\, \overline{\rho(\vartheta')}^{\, \deg \, p} |p|^{\scriptscriptstyle - 3\slash 4}\\
&\sum_{\varepsilon = (\varepsilon_{\scriptscriptstyle p'})_{p' \, \mid \, c'}} 
\mathscr{P}_{\! \scriptscriptstyle 2}(c_{\varepsilon}')\, 
\mathscr{P}_{\! \scriptscriptstyle 3}\left(\frac{c'}{c_{\varepsilon}'} \right)
\prod_{p' \, \mid \, c'} G^{(\varepsilon_{\scriptscriptstyle p'})}
(|p'|^{\, \scriptscriptstyle - 1\slash 2}\!, \, |p'|^{\, \scriptscriptstyle - 1\slash 2}\!, \,
|p'|^{\, \scriptscriptstyle - 1\slash 2}\!, \, \overline{\rho(\vartheta')}^{\, \deg \, p'}\! |p'|^{\scriptscriptstyle - 3 \slash 4}; \, |p'|).
\end{split}
\end{equation*} 
Recalling the explicit expressions of $F$ and $G^{(\varepsilon_{\scriptscriptstyle p'})}$ 
\!(see the proof of Lemma \ref{Estimate Zloc}), the equality \eqref{eq: generic-residue-3/4} follows now from 
\eqref{eq: sqfree-non-sqfree1/2} and a routine computation. This completes the proof of the theorem.

\section{Proof of Theorem \ref{Main Theorem B}}  \label{proof Thm B}
The proof is a standard application of the residue theorem. First replace $q^{- w}$ in 
$ 
Z_{\scriptscriptstyle 0}\big(\tfrac{1}{2}, \tfrac{1}{2}, \tfrac{1}{2}, w, 1\big)
$ 
by $\xi,$ and denote the resulting function by $\mathscr{W}(\xi).$ Thus 
\begin{equation*}
\mathscr{W}(\xi)\, = \sum_{D \ge 0}
\; \Bigg(\, \sum_{\substack{d_{\scriptscriptstyle 0} - \mathrm{monic \; \& \; sq. \; free} \\ \deg \, d_{\scriptscriptstyle 0} \, = \, D}} \, 
L\big(\tfrac{1}{2}, \chi_{\scriptscriptstyle d_{\scriptscriptstyle 0}}\!\big)^{\!\scriptscriptstyle 3} \Bigg)\, \xi^{\scriptscriptstyle D}.
\end{equation*} 
By Theorem \ref{Main Theorem A}, this function is meromorphic in the open disk $|\xi| < q^{\scriptscriptstyle - 2\slash 3}.$ For small positive $\delta,$ let 
$A_{\scriptscriptstyle \delta} = \{\xi \in \mathbb{C} : q^{\scriptscriptstyle - 2} \le |\xi| \le q^{\scriptscriptstyle - 2\slash 3 - \delta} \},$ and for $D \ge 0,$ 
consider the contour integral 
\begin{equation*} 
I(D) = \frac{1}{2 \, \pi \, i} \!\int_{\partial A_{\scriptscriptstyle \delta}} \frac{\mathscr{W}(\xi)}{\xi^{\scriptscriptstyle D + 1}}\, d\xi.
\end{equation*} 
We have 
\begin{equation*} 
\sum_{\substack{d_{\scriptscriptstyle 0} - \mathrm{monic \; \& \; sq. \; free} \\ \deg \, d_{\scriptscriptstyle 0} \, = \, D}} \, 
L\big(\tfrac{1}{2}, \chi_{\scriptscriptstyle d_{\scriptscriptstyle 0}}\!\big)^{\!\scriptscriptstyle 3}
= \frac{1}{2 \, \pi \, i}\;  \int\limits_{|\xi| \, = \, q^{\scriptscriptstyle - 2}} \frac{\mathscr{W}(\xi)}{\xi^{\scriptscriptstyle D + 1}}\, d\xi
\end{equation*} 
and by applying \eqref{eq: sqfree-non-sqfree1/2} and \eqref{eq: basic-estimate-zh},    
\begin{equation*} 
\int\limits_{|\xi| \, = \, q^{\scriptscriptstyle - 2\slash 3 - \delta}} \frac{\mathscr{W}(\xi)}{\xi^{\scriptscriptstyle D + 1}}\, d\xi
\, \ll_{\delta, \, q} \,  q^{\scriptscriptstyle D \left(\! \frac{2}{3} + \delta \!\right)} 
\end{equation*} 
giving the error term in the asymptotic formula. By the residue theorem, $I(D)$ is the sum of the residues at the poles of the function 
$\mathscr{W}(\xi) \slash \xi^{\scriptscriptstyle D + 1}$ in the annulus $A_{\scriptscriptstyle \delta},$ i.e., 
$\xi = \pm \, q^{\scriptscriptstyle - 1}$ \!and 
$\xi = \pm \, q^{\scriptscriptstyle - 3\slash 4}\!, \, \pm\,  i \, q^{\scriptscriptstyle - 3\slash 4}.$ 
The sum corresponding to the poles at $\xi = \pm \, q^{\scriptscriptstyle - 1}$ gives the main contribution to 
the asymptotic formula, and can be computed as in \cite{DGH}; see also \cite[Section~8 (a)]{RW} 
and \cite[Section~5.3]{AK}. \!Now, from the proof of Theorem \ref{Main Theorem A}, the sum of the residues at 
$\xi = \pm \, q^{\scriptscriptstyle - 3\slash 4}, \, \pm\,  i \, q^{\scriptscriptstyle - 3\slash 4}$ of the integrand is given by
\begin{equation*} 
\begin{split}
& - \tfrac{1}{4} \small{(1 + q^{\scriptscriptstyle 1\slash 4} + 10 \, q^{\scriptscriptstyle 1\slash 2} + 7  q^{\scriptscriptstyle 3 \slash 4} 
\, + 20 \, q + 7  q^{\scriptscriptstyle 5\slash 4} + 10 \, q^{\scriptscriptstyle 3\slash 2} + q^{\scriptscriptstyle 7\slash 4} + q^{\scriptscriptstyle 2})} \, 
q^{\scriptscriptstyle \frac{3}{4} D}\,  \zeta\big(\tfrac{1}{2}\big)^{\! \scriptscriptstyle 7}
\cdot \, \prod_{p}  P\, \Big(1\slash \sqrt{|p|} \, \Big) \\
& - \tfrac{(- 1)^{\scriptscriptstyle D}}{4} \small{(1 - q^{\scriptscriptstyle 1\slash 4} + 10 \, q^{\scriptscriptstyle 1\slash 2} 
- 7  q^{\scriptscriptstyle 3 \slash 4} \, + 20 \, q - 7  q^{\scriptscriptstyle 5\slash 4} + 10 \, q^{\scriptscriptstyle 3\slash 2} - q^{\scriptscriptstyle 7\slash 4} 
+ q^{\scriptscriptstyle 2})} \, q^{\scriptscriptstyle \frac{3}{4} D}\,  \zeta\big(\tfrac{1}{2}\big)^{\! \scriptscriptstyle 7}
\cdot \, \prod_{p}  P\, \Big(1 \slash \sqrt{|p|} \, \Big) \\
& - \tfrac{1}{2}\, \Re\big(i^{\, \scriptscriptstyle D} \small{(1 - i \, q^{\scriptscriptstyle 1\slash 4}\! 
- 4 \, q^{\scriptscriptstyle 1\slash 2} 
+ 7 i  \, q^{\scriptscriptstyle 3 \slash 4} + 6 \, q - 7 i \, q^{\scriptscriptstyle 5\slash 4}\! - 4 \, q^{\scriptscriptstyle 3\slash 2} 
+ i \, q^{\scriptscriptstyle 7\slash 4} + q^{\scriptscriptstyle 2})}\big)\, 
q^{\scriptscriptstyle \frac{3}{4} D}  
L\big(\tfrac{1}{2}, \chi_{{\scriptscriptstyle \theta_{\scriptscriptstyle 0}}}\big)^{\! \scriptscriptstyle 7}
\cdot \prod_{p}  P\, \Big((- 1)^{\deg \, p}\slash \sqrt{|p|} \, \Big).
\end{split}
\end{equation*} 
Thus, letting $R(D, q)$ denote the last expression times $ - q^{\scriptscriptstyle - \frac{3}{4} D},$ we have that   
\begin{equation*} 
\sum_{\substack{d_{\scriptscriptstyle 0} - \mathrm{monic \; \& \; sq. \; free} \\ \deg \, d_{\scriptscriptstyle 0} \, = \, D}} \, 
L\big(\tfrac{1}{2}, \chi_{\scriptscriptstyle d_{\scriptscriptstyle 0}}\!\big)^{\!\scriptscriptstyle 3}
= \, \frac{q^{\scriptscriptstyle D}}{\zeta(2)} Q(D, q)\, + \, q^{\scriptscriptstyle \frac{3}{4} D} R(D, q)
\, + \, O_{{\scriptstyle \delta}, \, q}\Big(q^{\scriptscriptstyle D \left(\! \frac{2}{3} + \delta \!\right)}\Big)
\end{equation*} 
which completes the proof.

\appendix
\section{Appendix} \label{Appendix A} 
To obtain the estimate \eqref{eq: EstimateZ1}, we have used the Lindel\"of-type bound established in the following

\vskip10pt
\begin{thm}\label{estimate-L-realpart-w=1} --- Let $\mathbb{F}_{\! q}$ be a finite field of odd characteristic, and let $d$ be a square-free polynomial over $\mathbb{F}_{\! q}$ of degree $D \ge 3.$ Then, for any $t\in \mathbb{R},$ we have
\begin{equation*}
\big|L\big(\tfrac{1}{2} + i t, \chi_{d} \big)\big| \, < \, 4 \, |d|^{\scriptscriptstyle \frac{10}{\log D}}. 
\end{equation*} 
\end{thm}

\begin{proof} \!We shall follow closely the argument in the proof of \cite[Theorem~5.1]{BCDGL-D}. Let $C_{d}$ denote the 
(elliptic\slash hyperelliptic) curve corresponding to $d,$ and consider the numerator $P_{\! \scriptscriptstyle d}(u)$ 
of the zeta function of $C_{d}.$ Then 
\begin{equation*}
L(s, \chi_{d}) \, = \, (1 \pm q^{- s})^{\scriptscriptstyle \epsilon(D)} P_{\! \scriptscriptstyle d}(q^{ - s})
\end{equation*}
with $\epsilon(D) = (1 + ( - 1)^{\scriptscriptstyle D})\slash 2$ and the $+$ or $-$ sign is determined according to whether 
the leading coefficient of $d$ is a square in $\mathbb{F}_{\! q}^{\times}$ or not. We estimate the factor 
$
(1 \pm q^{- s})^{\scriptscriptstyle \epsilon(D)}
$ 
(for $s = \frac{1}{2} + i t$) trivially:
\begin{equation*}
|\, 1 \pm q^{- s} \, |^{\scriptscriptstyle \epsilon(D)} \le 1 + q^{- 1 \slash 2} \le 1 + \frac{1}{\sqrt{3}}. 
\end{equation*} 
It is well-known (see \cite{Weil}) that
\begin{equation*}
P_{\! \scriptscriptstyle d}(u) = \prod_{m = 1}^{2 g}(1 \, - \, \sqrt{q} \, e^{i \omega_{\scriptstyle m}} u) \in \mathbb{Z}[u] 
\end{equation*} 
with $\omega_{\scriptstyle m}\in \mathbb{R}$ for all $m;$ the genus $g$ of the curve $C_{d}$ is obtained from the degree $D$ of 
the polynomial $d$ by: \!$2 g = D - 1$ if $D$ is odd, and $2 g = D - 2$ if $D$ is even.

Now, by \cite[Theorem~8.1]{CaVaa}, for every non-negative integer $N$ and every monic polynomial  
\begin{equation*}
F(z) = \prod_{m = 1}^{M} (z - \alpha_{m}) \;\;\; \qquad \;\;\; \text{($\alpha_{\scriptstyle1}, \ldots, \alpha_{\scriptscriptstyle M} \in \mathbb{C}$ 
with $|\alpha_{m}| \le 1$ for all $1\le m \le M$)}
\end{equation*} 
we have the estimate 
\begin{equation*}
\underset{|z| \le 1}{\mathrm{sup}}\; \log |F(z)| \, \le \, M\, (N + 1)^{\scriptscriptstyle -1} \log 2 \; + \, 
\sum_{n = 1}^{N} n^{\scriptscriptstyle -1} \bigg| \sum_{m = 1}^{M} \alpha_{\scriptstyle m}^{n} \bigg|.
\end{equation*} 
Fix an algebraic closure $\overline{\mathbb{F}}_{q}$ of $\mathbb{F}_{q}.$ Let $\chi_{\scriptscriptstyle n}$ ($n \ge 1$) denote the 
non-trivial real character of $\mathbb{F}_{\! q^{\scriptscriptstyle n}}^{\times},$ \!extended to $\mathbb{F}_{\! q^{\scriptscriptstyle n}} 
\subset \overline{\mathbb{F}}_{q}$ by setting $\chi_{\scriptscriptstyle n}(0) : = 0.$ 
Applying this bound to $P_{\! \scriptscriptstyle d}(u),$ 
we have  
\begin{equation*}
\log |P_{\! \scriptscriptstyle d}(u)| \, < \, D\, (N + 1)^{\scriptscriptstyle - 1} \log 2 \; + \, 
\sum_{n = 1}^{N} n^{\scriptscriptstyle -1} \bigg|\sum_{m = 1}^{2 g} e^{i n \omega_{\scriptstyle m}}\bigg|
\end{equation*} 
for every $u \in \mathbb{C}$ with $|u| = 1\slash \sqrt{q}.$ Recalling that, for a prime $\ell$ different from the characteristic of 
$\mathbb{F}_{q}$ and $n \ge 1,$
\begin{equation*}
q^{n \slash 2}\! \sum_{m = 1}^{2 g} e^{i n \omega_{\scriptstyle m}} 
=  \, \Tr(F^{*n} \, \vert \; \mathrm{H}_{\text{\'et}}^{\scriptscriptstyle 1}(\bar{C}_{d}, \mathbb{Q}_{\ell}))\,
= \, - \sum_{\theta \in  {\bf{\mathrm{P}}}^{\scriptscriptstyle 1}(\mathbb{F}_{\! q^{\scriptscriptstyle n}}\!)} 
\!\!\chi_{\scriptscriptstyle n} (d(\theta))
\end{equation*} 
where $\bar{C}_{d}: = C_{d} \otimes_{_{\mathbb{F}_{q}}} \! \overline{\mathbb{F}}_{q}$ and $F^{*}$ \!is the endomorphism of the 
$\ell$-adic \'etale cohomology induced by the Frobenius morphism $F : \bar{C}_{d} \rightarrow \bar{C}_{d},$ we have trivially 
\begin{equation*}
\bigg|\sum_{m = 1}^{2 g} e^{i n \omega_{\scriptstyle m}}\bigg| \, \le \, q^{ - n \slash 2} + \, q^{n \slash 2}.
\end{equation*} 
Consequently, if $N \ge 1$ we have 
\begin{equation*} 
\begin{split}
\sum_{n = 1}^{N} n^{\scriptscriptstyle -1} \bigg|\sum_{m = 1}^{2 g} e^{i n \omega_{\scriptstyle m}}\bigg| \, & \le \, 
\sum_{n = 1}^{N} n^{\scriptscriptstyle -1} \big(q^{ - n \slash 2} + \, q^{n \slash 2} \big) \\
&< \, \sum_{n = 1}^{\infty} n^{\scriptscriptstyle -1} q^{ - n \slash 2} \, + \, 2\!\sum_{N\slash 2 \le n \le N} n^{\scriptscriptstyle -1} q^{n \slash 2}\\
& < \, \log\big(1 - 3^{- 1 \slash 2}\big)^{\! \scriptscriptstyle - 1} + \; 4 \, \big(1 - 3^{- 1 \slash 2}\big)^{\! \scriptscriptstyle - 1} 
\cdot\,  \frac{q^{N \slash 2}}{N}.
\end{split}
\end{equation*} 
Thus, for all $N \ge 1,$ we obtain the estimate
\begin{equation*}
\log |P_{\! \scriptscriptstyle d}(u)| \, < \, D\, (N + 1)^{\scriptscriptstyle - 1} \log 2 \; + 
\, \log\big(1 - 3^{- 1 \slash 2}\big)^{\! \scriptscriptstyle - 1} + \; 8 \, \big(1 - 3^{- 1 \slash 2}\big)^{\! \scriptscriptstyle - 1} 
q^{N \slash 2}\,(N + 1)^{\scriptscriptstyle - 1}.
\end{equation*}
Choosing $N = \left \lfloor{\frac{2 \log D}{\log q}}\right \rfloor,$ we see that
\begin{equation*}
|P_{\! \scriptscriptstyle d}(u)| \, < \, \big(1 - 3^{- 1 \slash 2}\big)^{\! \scriptscriptstyle - 1} \cdot \, 
|d|^{\scriptscriptstyle \left[\! \frac{\log 2}{2} \, + \; 4\, \left(1 \, - \, 3^{\scriptscriptstyle - 1 \slash 2}\right)^{\scriptscriptstyle - 1}\right]
\frac{1}{\log D}}
\qquad \text{(if $D \, \ge \sqrt{q}$)}
\end{equation*} 
and 
\begin{equation*}
|P_{\! \scriptscriptstyle d}(u)| \, < \, |d|^{\scriptscriptstyle \frac{\log 2}{2 \log D}}
\qquad \text{(if $D \, < \sqrt{q}$)}
\end{equation*}
from which the theorem follows. 
\end{proof}

\section{Appendix} \label{Appendix B} 
Let $\mathbb{F}_{\! q}$ be a finite field with $q$ elements of odd characteristic. \!In \cite{BD} we have constructed a 
multiple {D}irichlet series associated to the fourth moment of quadratic {D}irichlet {$L$}-series over $\mathbb{F}_{\! q}(x).$ 
By setting one of the first four variables of this multiple {D}irichlet series to zero and by applying the recurrence relations in the 
proof of Theorem 3.7 of loc. cit., one obtains the explicit expression of the multiple {D}irichlet series associated to the cubic moment. Explicitly, this series is, in fact, a rational function            
\begin{equation} \label{rat-funct-Z-D4}
Z(z_{\scriptscriptstyle 1}^{}\!, \, z_{\scriptscriptstyle 2}^{}, \, z_{\scriptscriptstyle 3}^{}, \, z_{\scriptscriptstyle 4}^{}; q) 
= \frac{N(z_{\scriptscriptstyle 1}^{}\!, \, z_{\scriptscriptstyle 2}^{}, \, z_{\scriptscriptstyle 3}^{}, \, z_{\scriptscriptstyle 4}^{}; q)}
{D(z_{\scriptscriptstyle 1}^{}\!, \, z_{\scriptscriptstyle 2}^{}, \, z_{\scriptscriptstyle 3}^{}, \, z_{\scriptscriptstyle 4}^{}; q)}
\end{equation} 
with numerator given by    
\begin{equation*}
\begin{split}
& N(z_{\scriptscriptstyle 1}^{}\!, \, z_{\scriptscriptstyle 2}^{}, \, z_{\scriptscriptstyle 3}^{}, \, z_{\scriptscriptstyle 4}^{}; q) \, = 
1 -  q^{\scriptscriptstyle 2} z_{\scriptscriptstyle 1}^{} z_{\scriptscriptstyle 4}^{} - q^{\scriptscriptstyle 2} z_{\scriptscriptstyle 2}^{} z_{\scriptscriptstyle 4}^{} +  q^{\scriptscriptstyle 3} z_{\scriptscriptstyle 1}^{} z_{\scriptscriptstyle 2}^{} z_{\scriptscriptstyle 4}^{} - q^{\scriptscriptstyle 2} z_{\scriptscriptstyle 3}^{} z_{\scriptscriptstyle 4}^{} + q^{\scriptscriptstyle 3} z_{\scriptscriptstyle 1}^{} z_{\scriptscriptstyle 3}^{} z_{\scriptscriptstyle 4}^{} 
+ q^{\scriptscriptstyle 3} z_{\scriptscriptstyle 2}^{} z_{\scriptscriptstyle 3}^{} z_{\scriptscriptstyle 4}^{}  -  q^{\scriptscriptstyle 4} z_{\scriptscriptstyle 1}^{} z_{\scriptscriptstyle 2}^{} z_{\scriptscriptstyle 3}^{} z_{\scriptscriptstyle 4}^{} \\
& + q^{\scriptscriptstyle 3} z_{\scriptscriptstyle 1}^{} z_{\scriptscriptstyle 2}^{} z_{\scriptscriptstyle 4}^{\scriptscriptstyle 2} - q^{\scriptscriptstyle 4} z_{\scriptscriptstyle 1}^{\scriptscriptstyle 2} z_{\scriptscriptstyle 2}^{} z_{\scriptscriptstyle 4}^{\scriptscriptstyle 2} 
- q^{\scriptscriptstyle 4} z_{\scriptscriptstyle 1}^{} z_{\scriptscriptstyle 2}^{\scriptscriptstyle 2} z_{\scriptscriptstyle 4}^{\scriptscriptstyle 2} 
+ q^{\scriptscriptstyle 3} z_{\scriptscriptstyle 1}^{} z_{\scriptscriptstyle 3}^{} z_{\scriptscriptstyle 4}^{\scriptscriptstyle 2} - q^{\scriptscriptstyle 4} z_{\scriptscriptstyle 1}^{\scriptscriptstyle 2} z_{\scriptscriptstyle 3}^{} z_{\scriptscriptstyle 4}^{\scriptscriptstyle 2} 
+ q^{\scriptscriptstyle 3} z_{\scriptscriptstyle 2}^{} z_{\scriptscriptstyle 3}^{} z_{\scriptscriptstyle 4}^{\scriptscriptstyle 2} 
- 2 q^{\scriptscriptstyle 4} z_{\scriptscriptstyle 1}^{} z_{\scriptscriptstyle 2}^{} z_{\scriptscriptstyle 3}^{} z_{\scriptscriptstyle 4}^{\scriptscriptstyle 2} 
+ q^{\scriptscriptstyle 4} z_{\scriptscriptstyle 1}^{\scriptscriptstyle 2} z_{\scriptscriptstyle 2}^{} z_{\scriptscriptstyle 3}^{} 
z_{\scriptscriptstyle 4}^{\scriptscriptstyle 2} \\ 
&+ q^{\scriptscriptstyle 5} z_{\scriptscriptstyle 1}^{\scriptscriptstyle 2} z_{\scriptscriptstyle 2}^{} z_{\scriptscriptstyle 3}^{} 
z_{\scriptscriptstyle 4}^{\scriptscriptstyle 2} - q^{\scriptscriptstyle 4} z_{\scriptscriptstyle 2}^{\scriptscriptstyle 2} 
z_{\scriptscriptstyle 3}^{} z_{\scriptscriptstyle 4}^{\scriptscriptstyle 2} + q^{\scriptscriptstyle 4} 
z_{\scriptscriptstyle 1}^{} z_{\scriptscriptstyle 2}^{\scriptscriptstyle 2} z_{\scriptscriptstyle 3}^{} 
z_{\scriptscriptstyle 4}^{\scriptscriptstyle 2} 
+ q^{\scriptscriptstyle 5} z_{\scriptscriptstyle 1}^{} z_{\scriptscriptstyle 2}^{\scriptscriptstyle 2} 
z_{\scriptscriptstyle 3}^{} z_{\scriptscriptstyle 4}^{\scriptscriptstyle 2} - q^{\scriptscriptstyle 5} 
z_{\scriptscriptstyle 1}^{\scriptscriptstyle 2} z_{\scriptscriptstyle 2}^{\scriptscriptstyle 2} z_{\scriptscriptstyle 3}^{} 
z_{\scriptscriptstyle 4}^{\scriptscriptstyle 2} - q^{\scriptscriptstyle 4} z_{\scriptscriptstyle 1}^{} 
z_{\scriptscriptstyle 3}^{\scriptscriptstyle 2} z_{\scriptscriptstyle 4}^{\scriptscriptstyle 2} 
- q^{\scriptscriptstyle 4} z_{\scriptscriptstyle 2}^{} z_{\scriptscriptstyle 3}^{\scriptscriptstyle 2} 
z_{\scriptscriptstyle 4}^{\scriptscriptstyle 2} + q^{\scriptscriptstyle 4} z_{\scriptscriptstyle 1}^{} 
z_{\scriptscriptstyle 2}^{} z_{\scriptscriptstyle 3}^{\scriptscriptstyle 2} z_{\scriptscriptstyle 4}^{\scriptscriptstyle 2} 
+ q^{\scriptscriptstyle 5} z_{\scriptscriptstyle 1}^{} z_{\scriptscriptstyle 2}^{} z_{\scriptscriptstyle 3}^{\scriptscriptstyle 2} 
z_{\scriptscriptstyle 4}^{\scriptscriptstyle 2} \\ 
&- q^{\scriptscriptstyle 5} z_{\scriptscriptstyle 1}^{\scriptscriptstyle 2} z_{\scriptscriptstyle 2}^{} 
z_{\scriptscriptstyle 3}^{\scriptscriptstyle 2} z_{\scriptscriptstyle 4}^{\scriptscriptstyle 2} 
- q^{\scriptscriptstyle 5} z_{\scriptscriptstyle 1}^{} z_{\scriptscriptstyle 2}^{\scriptscriptstyle 2} 
z_{\scriptscriptstyle 3}^{\scriptscriptstyle 2} z_{\scriptscriptstyle 4}^{\scriptscriptstyle 2} 
+ q^{\scriptscriptstyle 6} z_{\scriptscriptstyle 1}^{\scriptscriptstyle 2} z_{\scriptscriptstyle 2}^{\scriptscriptstyle 2} 
z_{\scriptscriptstyle 4}^{\scriptscriptstyle 3} 
- q^{\scriptscriptstyle 5} z_{\scriptscriptstyle 1}^{} z_{\scriptscriptstyle 2}^{} z_{\scriptscriptstyle 3}^{} 
z_{\scriptscriptstyle 4}^{\scriptscriptstyle 3} + q^{\scriptscriptstyle 6} z_{\scriptscriptstyle 1}^{\scriptscriptstyle 2} 
z_{\scriptscriptstyle 2}^{} z_{\scriptscriptstyle 3}^{} z_{\scriptscriptstyle 4}^{\scriptscriptstyle 3} 
- q^{\scriptscriptstyle 6} z_{\scriptscriptstyle 1}^{\scriptscriptstyle 3} z_{\scriptscriptstyle 2}^{} 
z_{\scriptscriptstyle 3}^{} z_{\scriptscriptstyle 4}^{\scriptscriptstyle 3} 
+ q^{\scriptscriptstyle 6} z_{\scriptscriptstyle 1}^{} z_{\scriptscriptstyle 2}^{\scriptscriptstyle 2} 
z_{\scriptscriptstyle 3}^{} z_{\scriptscriptstyle 4}^{\scriptscriptstyle 3} - q^{\scriptscriptstyle 6} 
z_{\scriptscriptstyle 1}^{\scriptscriptstyle 2} z_{\scriptscriptstyle 2}^{\scriptscriptstyle 2} 
z_{\scriptscriptstyle 3}^{} z_{\scriptscriptstyle 4}^{\scriptscriptstyle 3} \\ 
& - q^{\scriptscriptstyle 7} z_{\scriptscriptstyle 1}^{\scriptscriptstyle 2} 
z_{\scriptscriptstyle 2}^{\scriptscriptstyle 2} z_{\scriptscriptstyle 3}^{} 
z_{\scriptscriptstyle 4}^{\scriptscriptstyle 3} + q^{\scriptscriptstyle 7} 
z_{\scriptscriptstyle 1}^{\scriptscriptstyle 3} z_{\scriptscriptstyle 2}^{\scriptscriptstyle 2} 
z_{\scriptscriptstyle 3}^{} z_{\scriptscriptstyle 4}^{\scriptscriptstyle 3} - q^{\scriptscriptstyle 6} 
z_{\scriptscriptstyle 1}^{} z_{\scriptscriptstyle 2}^{\scriptscriptstyle 3} z_{\scriptscriptstyle 3}^{} 
z_{\scriptscriptstyle 4}^{\scriptscriptstyle 3} 
+ q^{\scriptscriptstyle 7} z_{\scriptscriptstyle 1}^{\scriptscriptstyle 2} 
z_{\scriptscriptstyle 2}^{\scriptscriptstyle 3} z_{\scriptscriptstyle 3}^{} 
z_{\scriptscriptstyle 4}^{\scriptscriptstyle 3} + q^{\scriptscriptstyle 6} 
z_{\scriptscriptstyle 1}^{\scriptscriptstyle 2} z_{\scriptscriptstyle 3}^{\scriptscriptstyle 2} 
z_{\scriptscriptstyle 4}^{\scriptscriptstyle 3} + q^{\scriptscriptstyle 6} z_{\scriptscriptstyle 1}^{} 
z_{\scriptscriptstyle 2}^{} z_{\scriptscriptstyle 3}^{\scriptscriptstyle 2} z_{\scriptscriptstyle 4}^{\scriptscriptstyle 3} 
- q^{\scriptscriptstyle 6} z_{\scriptscriptstyle 1}^{\scriptscriptstyle 2} z_{\scriptscriptstyle 2}^{} 
z_{\scriptscriptstyle 3}^{\scriptscriptstyle 2} z_{\scriptscriptstyle 4}^{\scriptscriptstyle 3} 
- q^{\scriptscriptstyle 7} z_{\scriptscriptstyle 1}^{\scriptscriptstyle 2} z_{\scriptscriptstyle 2}^{} 
z_{\scriptscriptstyle 3}^{\scriptscriptstyle 2} z_{\scriptscriptstyle 4}^{\scriptscriptstyle 3} \\ 
& + q^{\scriptscriptstyle 7} z_{\scriptscriptstyle 1}^{\scriptscriptstyle 3} z_{\scriptscriptstyle 2}^{} 
z_{\scriptscriptstyle 3}^{\scriptscriptstyle 2} z_{\scriptscriptstyle 4}^{\scriptscriptstyle 3} 
+ q^{\scriptscriptstyle 6} z_{\scriptscriptstyle 2}^{\scriptscriptstyle 2} z_{\scriptscriptstyle 3}^{\scriptscriptstyle 2} 
z_{\scriptscriptstyle 4}^{\scriptscriptstyle 3} - q^{\scriptscriptstyle 6} z_{\scriptscriptstyle1}^{} 
z_{\scriptscriptstyle 2}^{\scriptscriptstyle 2} z_{\scriptscriptstyle 3}^{\scriptscriptstyle 2} z_{\scriptscriptstyle 4}^{\scriptscriptstyle 3} 
- q^{\scriptscriptstyle 7} z_{\scriptscriptstyle 1}^{} z_{\scriptscriptstyle 2}^{\scriptscriptstyle 2} 
z_{\scriptscriptstyle 3}^{\scriptscriptstyle 2} z_{\scriptscriptstyle 4}^{\scriptscriptstyle 3} 
+ 3 q^{\scriptscriptstyle 7} z_{\scriptscriptstyle 1}^{\scriptscriptstyle 2} 
z_{\scriptscriptstyle 2}^{\scriptscriptstyle 2} z_{\scriptscriptstyle 3}^{\scriptscriptstyle 2} 
z_{\scriptscriptstyle 4}^{\scriptscriptstyle 3} - q^{\scriptscriptstyle 8} z_{\scriptscriptstyle 1}^{\scriptscriptstyle 3} 
z_{\scriptscriptstyle 2}^{\scriptscriptstyle 2} z_{\scriptscriptstyle 3}^{\scriptscriptstyle 2} z_{\scriptscriptstyle 4}^{\scriptscriptstyle 3} 
+ q^{\scriptscriptstyle 7} z_{\scriptscriptstyle 1}^{} z_{\scriptscriptstyle 2}^{\scriptscriptstyle 3} 
z_{\scriptscriptstyle 3}^{\scriptscriptstyle 2} z_{\scriptscriptstyle 4}^{\scriptscriptstyle 3} 
- q^{\scriptscriptstyle 8} z_{\scriptscriptstyle 1}^{\scriptscriptstyle 2} 
z_{\scriptscriptstyle 2}^{\scriptscriptstyle 3} z_{\scriptscriptstyle 3}^{\scriptscriptstyle 2} 
z_{\scriptscriptstyle 4}^{\scriptscriptstyle 3} \\ 
& - q^{\scriptscriptstyle 6} z_{\scriptscriptstyle 1}^{} z_{\scriptscriptstyle 2}^{} 
z_{\scriptscriptstyle 3}^{\scriptscriptstyle 3} z_{\scriptscriptstyle 4}^{\scriptscriptstyle 3} 
+ q^{\scriptscriptstyle 7} z_{\scriptscriptstyle 1}^{\scriptscriptstyle 2} z_{\scriptscriptstyle 2}^{} 
z_{\scriptscriptstyle 3}^{\scriptscriptstyle 3} z_{\scriptscriptstyle 4}^{\scriptscriptstyle 3} 
+ q^{\scriptscriptstyle 7} z_{\scriptscriptstyle 1}^{} z_{\scriptscriptstyle 2}^{\scriptscriptstyle 2} 
z_{\scriptscriptstyle 3}^{\scriptscriptstyle 3} z_{\scriptscriptstyle 4}^{\scriptscriptstyle 3}
 -  q^{\scriptscriptstyle 8} z_{\scriptscriptstyle 1}^{\scriptscriptstyle 2} 
 z_{\scriptscriptstyle 2}^{\scriptscriptstyle 2} z_{\scriptscriptstyle 3}^{\scriptscriptstyle 3} 
 z_{\scriptscriptstyle 4}^{\scriptscriptstyle 3} - q^{\scriptscriptstyle 8} 
 z_{\scriptscriptstyle 1}^{\scriptscriptstyle 3} z_{\scriptscriptstyle 2}^{\scriptscriptstyle 3} 
 z_{\scriptscriptstyle 3}^{} z_{\scriptscriptstyle4}^{\scriptscriptstyle 4} + q^{\scriptscriptstyle 8} 
 z_{\scriptscriptstyle 1}^{\scriptscriptstyle 2} z_{\scriptscriptstyle 2}^{\scriptscriptstyle 2} 
 z_{\scriptscriptstyle 3}^{\scriptscriptstyle 2} z_{\scriptscriptstyle 4}^{\scriptscriptstyle 4}
- q^{\scriptscriptstyle 8} z_{\scriptscriptstyle 1}^{\scriptscriptstyle 3} 
z_{\scriptscriptstyle 2}^{\scriptscriptstyle 2} z_{\scriptscriptstyle 3}^{\scriptscriptstyle 2} 
z_{\scriptscriptstyle 4}^{\scriptscriptstyle 4} + q^{\scriptscriptstyle 9} 
z_{\scriptscriptstyle 1}^{\scriptscriptstyle 4} z_{\scriptscriptstyle 2}^{\scriptscriptstyle 2} 
z_{\scriptscriptstyle 3}^{\scriptscriptstyle 2} z_{\scriptscriptstyle 4}^{\scriptscriptstyle 4} \\ 
& - q^{\scriptscriptstyle 8} z_{\scriptscriptstyle 1}^{\scriptscriptstyle 2} 
z_{\scriptscriptstyle 2}^{\scriptscriptstyle 3} z_{\scriptscriptstyle 3}^{\scriptscriptstyle 2} 
z_{\scriptscriptstyle 4}^{\scriptscriptstyle 4} + q^{\scriptscriptstyle 9} 
z_{\scriptscriptstyle 1}^{\scriptscriptstyle 3} z_{\scriptscriptstyle 2}^{\scriptscriptstyle 3} 
z_{\scriptscriptstyle 3}^{\scriptscriptstyle 2} z_{\scriptscriptstyle 4}^{\scriptscriptstyle 4} 
+ q^{\scriptscriptstyle 9} z_{\scriptscriptstyle 1}^{\scriptscriptstyle 2} 
z_{\scriptscriptstyle 2}^{\scriptscriptstyle 4} z_{\scriptscriptstyle 3}^{\scriptscriptstyle 2} 
z_{\scriptscriptstyle 4}^{\scriptscriptstyle 4} 
- q^{\scriptscriptstyle 8} z_{\scriptscriptstyle 1}^{\scriptscriptstyle 3} z_{\scriptscriptstyle 2}^{} 
z_{\scriptscriptstyle 3}^{\scriptscriptstyle 3} z_{\scriptscriptstyle 4}^{\scriptscriptstyle 4} 
- q^{\scriptscriptstyle 8} z_{\scriptscriptstyle 1}^{\scriptscriptstyle 2} 
z_{\scriptscriptstyle 2}^{\scriptscriptstyle 2} z_{\scriptscriptstyle 3}^{\scriptscriptstyle 3} 
z_{\scriptscriptstyle 4}^{\scriptscriptstyle 4} + q^{\scriptscriptstyle 9} 
z_{\scriptscriptstyle 1}^{\scriptscriptstyle 3} z_{\scriptscriptstyle 2}^{\scriptscriptstyle 2} 
z_{\scriptscriptstyle 3}^{\scriptscriptstyle 3} z_{\scriptscriptstyle 4}^{\scriptscriptstyle 4} 
- q^{\scriptscriptstyle 8} z_{\scriptscriptstyle 1}^{} z_{\scriptscriptstyle 2}^{\scriptscriptstyle 3} 
z_{\scriptscriptstyle 3}^{\scriptscriptstyle 3} z_{\scriptscriptstyle 4}^{\scriptscriptstyle 4} 
+ q^{\scriptscriptstyle 9} z_{\scriptscriptstyle 1}^{\scriptscriptstyle 2} 
z_{\scriptscriptstyle 2}^{\scriptscriptstyle 3} z_{\scriptscriptstyle 3}^{\scriptscriptstyle 3} 
z_{\scriptscriptstyle 4}^{\scriptscriptstyle 4} \\ 
& - q^{\scriptscriptstyle 9} z_{\scriptscriptstyle 1}^{\scriptscriptstyle 3} 
z_{\scriptscriptstyle 2}^{\scriptscriptstyle 3} z_{\scriptscriptstyle 3}^{\scriptscriptstyle 3} 
z_{\scriptscriptstyle 4}^{\scriptscriptstyle 4} + q^{\scriptscriptstyle 9} 
z_{\scriptscriptstyle 1}^{\scriptscriptstyle 2} z_{\scriptscriptstyle 2}^{\scriptscriptstyle 2} 
z_{\scriptscriptstyle 3}^{\scriptscriptstyle 4} z_{\scriptscriptstyle 4}^{\scriptscriptstyle 4} 
+ q^{\scriptscriptstyle 9} z_{\scriptscriptstyle 1}^{\scriptscriptstyle 3} 
z_{\scriptscriptstyle 2}^{\scriptscriptstyle 3} z_{\scriptscriptstyle 3}^{} 
z_{\scriptscriptstyle 4}^{\scriptscriptstyle 5} 
- q^{\scriptscriptstyle 9} z_{\scriptscriptstyle 1}^{\scriptscriptstyle 2} 
z_{\scriptscriptstyle 2}^{\scriptscriptstyle 2} z_{\scriptscriptstyle 3}^{\scriptscriptstyle 2} 
z_{\scriptscriptstyle 4}^{\scriptscriptstyle 5} + q^{\scriptscriptstyle 9} 
z_{\scriptscriptstyle 1}^{\scriptscriptstyle 3} z_{\scriptscriptstyle 2}^{\scriptscriptstyle 2} 
z_{\scriptscriptstyle 3}^{\scriptscriptstyle 2} z_{\scriptscriptstyle 4}^{\scriptscriptstyle 5} 
- q^{\scriptscriptstyle 10} z_{\scriptscriptstyle 1}^{\scriptscriptstyle 4} z_{\scriptscriptstyle 2}^{\scriptscriptstyle 2} 
z_{\scriptscriptstyle 3}^{\scriptscriptstyle 2} z_{\scriptscriptstyle 4}^{\scriptscriptstyle 5} 
+ q^{\scriptscriptstyle 9} z_{\scriptscriptstyle 1}^{\scriptscriptstyle 2} 
z_{\scriptscriptstyle 2}^{\scriptscriptstyle 3} z_{\scriptscriptstyle 3}^{\scriptscriptstyle 2} 
z_{\scriptscriptstyle 4}^{\scriptscriptstyle 5} - q^{\scriptscriptstyle 10} z_{\scriptscriptstyle 1}^{\scriptscriptstyle 3} 
z_{\scriptscriptstyle 2}^{\scriptscriptstyle 3} z_{\scriptscriptstyle 3}^{\scriptscriptstyle 2} z_{\scriptscriptstyle 4}^{\scriptscriptstyle 5} \\ 
& - q^{\scriptscriptstyle 10} z_{\scriptscriptstyle 1}^{\scriptscriptstyle 2} z_{\scriptscriptstyle 2}^{\scriptscriptstyle 4} 
z_{\scriptscriptstyle 3}^{\scriptscriptstyle 2} z_{\scriptscriptstyle 4}^{\scriptscriptstyle 5} 
+ q^{\scriptscriptstyle 9} z_{\scriptscriptstyle 1}^{\scriptscriptstyle 3} z_{\scriptscriptstyle 2}^{} 
z_{\scriptscriptstyle 3}^{\scriptscriptstyle 3} z_{\scriptscriptstyle 4}^{\scriptscriptstyle 5} 
+ q^{\scriptscriptstyle 9} z_{\scriptscriptstyle 1}^{\scriptscriptstyle 2} 
z_{\scriptscriptstyle 2}^{\scriptscriptstyle 2} z_{\scriptscriptstyle 3}^{\scriptscriptstyle 3} 
z_{\scriptscriptstyle 4}^{\scriptscriptstyle 5} 
- q^{\scriptscriptstyle 10} z_{\scriptscriptstyle 1}^{\scriptscriptstyle 3} z_{\scriptscriptstyle 2}^{\scriptscriptstyle 2} 
z_{\scriptscriptstyle 3}^{\scriptscriptstyle 3} z_{\scriptscriptstyle 4}^{\scriptscriptstyle 5} 
+ q^{\scriptscriptstyle 9} z_{\scriptscriptstyle 1}^{} z_{\scriptscriptstyle 2}^{\scriptscriptstyle 3} 
z_{\scriptscriptstyle 3}^{\scriptscriptstyle 3} z_{\scriptscriptstyle 4}^{\scriptscriptstyle 5} 
- q^{\scriptscriptstyle 10} z_{\scriptscriptstyle 1}^{\scriptscriptstyle 2} z_{\scriptscriptstyle 2}^{\scriptscriptstyle 3} 
z_{\scriptscriptstyle 3}^{\scriptscriptstyle 3} z_{\scriptscriptstyle 4}^{\scriptscriptstyle 5} 
+ q^{\scriptscriptstyle 10} z_{\scriptscriptstyle 1}^{\scriptscriptstyle 3} z_{\scriptscriptstyle 2}^{\scriptscriptstyle 3} 
z_{\scriptscriptstyle 3}^{\scriptscriptstyle 3} z_{\scriptscriptstyle 4}^{\scriptscriptstyle 5} 
- q^{\scriptscriptstyle 10} z_{\scriptscriptstyle 1}^{\scriptscriptstyle 2} z_{\scriptscriptstyle 2}^{\scriptscriptstyle 2} 
z_{\scriptscriptstyle 3}^{\scriptscriptstyle 4} z_{\scriptscriptstyle 4}^{\scriptscriptstyle 5} \\ 
& - q^{\scriptscriptstyle 10} z_{\scriptscriptstyle 1}^{\scriptscriptstyle 3} z_{\scriptscriptstyle 2}^{\scriptscriptstyle 3} 
z_{\scriptscriptstyle 3}^{\scriptscriptstyle 2} z_{\scriptscriptstyle 4}^{\scriptscriptstyle 6} 
+ q^{\scriptscriptstyle 11} z_{\scriptscriptstyle 1}^{\scriptscriptstyle 4} z_{\scriptscriptstyle 2}^{\scriptscriptstyle 3} 
z_{\scriptscriptstyle 3}^{\scriptscriptstyle 2} z_{\scriptscriptstyle 4}^{\scriptscriptstyle 6} 
+ q^{\scriptscriptstyle11} z_{\scriptscriptstyle 1}^{\scriptscriptstyle 3} z_{\scriptscriptstyle 2}^{\scriptscriptstyle 4} 
z_{\scriptscriptstyle 3}^{\scriptscriptstyle 2} z_{\scriptscriptstyle 4}^{\scriptscriptstyle 6} 
- q^{\scriptscriptstyle 12} z_{\scriptscriptstyle 1}^{\scriptscriptstyle 4} z_{\scriptscriptstyle 2}^{\scriptscriptstyle 4} 
z_{\scriptscriptstyle 3}^{\scriptscriptstyle 2} z_{\scriptscriptstyle 4}^{\scriptscriptstyle 6} 
- q^{\scriptscriptstyle 10} z_{\scriptscriptstyle 1}^{\scriptscriptstyle 3} z_{\scriptscriptstyle 2}^{\scriptscriptstyle 2} 
z_{\scriptscriptstyle 3}^{\scriptscriptstyle 3} z_{\scriptscriptstyle 4}^{\scriptscriptstyle 6}
+ q^{\scriptscriptstyle 11} z_{\scriptscriptstyle 1}^{\scriptscriptstyle 4} z_{\scriptscriptstyle 2}^{\scriptscriptstyle 2} 
z_{\scriptscriptstyle 3}^{\scriptscriptstyle 3} z_{\scriptscriptstyle 4}^{\scriptscriptstyle 6} 
- q^{\scriptscriptstyle 10} z_{\scriptscriptstyle 1}^{\scriptscriptstyle 2} z_{\scriptscriptstyle 2}^{\scriptscriptstyle 3} 
z_{\scriptscriptstyle 3}^{\scriptscriptstyle 3} z_{\scriptscriptstyle 4}^{\scriptscriptstyle 6} + 
3 q^{\scriptscriptstyle 11} z_{\scriptscriptstyle 1}^{\scriptscriptstyle 3} z_{\scriptscriptstyle 2}^{\scriptscriptstyle 3} 
z_{\scriptscriptstyle 3}^{\scriptscriptstyle 3} z_{\scriptscriptstyle 4}^{\scriptscriptstyle 6} \\ 
& - q^{\scriptscriptstyle 11} z_{\scriptscriptstyle 1}^{\scriptscriptstyle 4} z_{\scriptscriptstyle 2}^{\scriptscriptstyle 3} 
z_{\scriptscriptstyle 3}^{\scriptscriptstyle 3} z_{\scriptscriptstyle 4}^{\scriptscriptstyle 6} 
- q^{\scriptscriptstyle 12} z_{\scriptscriptstyle 1}^{\scriptscriptstyle 4} z_{\scriptscriptstyle 2}^{\scriptscriptstyle 3} 
z_{\scriptscriptstyle 3}^{\scriptscriptstyle 3} z_{\scriptscriptstyle 4}^{\scriptscriptstyle 6} + 
q^{\scriptscriptstyle 12} z_{\scriptscriptstyle 1}^{\scriptscriptstyle 5} z_{\scriptscriptstyle 2}^{\scriptscriptstyle 3} 
z_{\scriptscriptstyle 3}^{\scriptscriptstyle 3} z_{\scriptscriptstyle 4}^{\scriptscriptstyle 6} 
+ q^{\scriptscriptstyle 11} z_{\scriptscriptstyle 1}^{\scriptscriptstyle 2} z_{\scriptscriptstyle 2}^{\scriptscriptstyle 4} 
z_{\scriptscriptstyle 3}^{\scriptscriptstyle 3} z_{\scriptscriptstyle 4}^{\scriptscriptstyle 6} - 
q^{\scriptscriptstyle 11} z_{\scriptscriptstyle 1}^{\scriptscriptstyle 3} z_{\scriptscriptstyle 2}^{\scriptscriptstyle 4} 
z_{\scriptscriptstyle 3}^{\scriptscriptstyle 3} z_{\scriptscriptstyle 4}^{\scriptscriptstyle 6} 
- q^{\scriptscriptstyle 12} z_{\scriptscriptstyle 1}^{\scriptscriptstyle 3} z_{\scriptscriptstyle 2}^{\scriptscriptstyle 4} 
z_{\scriptscriptstyle 3}^{\scriptscriptstyle 3} z_{\scriptscriptstyle 4}^{\scriptscriptstyle 6} + 
q^{\scriptscriptstyle 12} z_{\scriptscriptstyle 1}^{\scriptscriptstyle 4} z_{\scriptscriptstyle 2}^{\scriptscriptstyle 4} 
z_{\scriptscriptstyle 3}^{\scriptscriptstyle 3} z_{\scriptscriptstyle 4}^{\scriptscriptstyle 6} 
+ q^{\scriptscriptstyle 12} z_{\scriptscriptstyle 1}^{\scriptscriptstyle 3} z_{\scriptscriptstyle 2}^{\scriptscriptstyle 5} 
z_{\scriptscriptstyle 3}^{\scriptscriptstyle 3} z_{\scriptscriptstyle 4}^{\scriptscriptstyle 6} \\ 
& + q^{\scriptscriptstyle 11} z_{\scriptscriptstyle 1}^{\scriptscriptstyle 3} z_{\scriptscriptstyle 2}^{\scriptscriptstyle 2} 
z_{\scriptscriptstyle 3}^{\scriptscriptstyle 4} z_{\scriptscriptstyle 4}^{\scriptscriptstyle 6} 
- q^{\scriptscriptstyle 12} z_{\scriptscriptstyle 1}^{\scriptscriptstyle 4} z_{\scriptscriptstyle 2}^{\scriptscriptstyle 2} 
z_{\scriptscriptstyle 3}^{\scriptscriptstyle 4} z_{\scriptscriptstyle 4}^{\scriptscriptstyle 6} 
+ q^{\scriptscriptstyle 11} z_{\scriptscriptstyle 1}^{\scriptscriptstyle 2} z_{\scriptscriptstyle 2}^{\scriptscriptstyle 3}
 z_{\scriptscriptstyle 3}^{\scriptscriptstyle 4} z_{\scriptscriptstyle 4}^{\scriptscriptstyle 6} 
 - q^{\scriptscriptstyle 11} z_{\scriptscriptstyle 1}^{\scriptscriptstyle 3} z_{\scriptscriptstyle 2}^{\scriptscriptstyle 3} 
 z_{\scriptscriptstyle 3}^{\scriptscriptstyle 4} z_{\scriptscriptstyle 4}^{\scriptscriptstyle 6} 
- q^{\scriptscriptstyle 12} z_{\scriptscriptstyle 1}^{\scriptscriptstyle 3} z_{\scriptscriptstyle 2}^{\scriptscriptstyle 3} 
z_{\scriptscriptstyle 3}^{\scriptscriptstyle 4} z_{\scriptscriptstyle 4}^{\scriptscriptstyle 6} 
+ q^{\scriptscriptstyle 12} z_{\scriptscriptstyle 1}^{\scriptscriptstyle 4} z_{\scriptscriptstyle 2}^{\scriptscriptstyle 3} 
z_{\scriptscriptstyle 3}^{\scriptscriptstyle 4} z_{\scriptscriptstyle 4}^{\scriptscriptstyle 6} 
- q^{\scriptscriptstyle 12} z_{\scriptscriptstyle 1}^{\scriptscriptstyle 2} z_{\scriptscriptstyle 2}^{\scriptscriptstyle 4} 
z_{\scriptscriptstyle 3}^{\scriptscriptstyle 4} z_{\scriptscriptstyle 4}^{\scriptscriptstyle 6} 
+ q^{\scriptscriptstyle 12} z_{\scriptscriptstyle 1}^{\scriptscriptstyle 3} z_{\scriptscriptstyle 2}^{\scriptscriptstyle 4} 
z_{\scriptscriptstyle 3}^{\scriptscriptstyle 4} z_{\scriptscriptstyle 4}^{\scriptscriptstyle 6} \\ 
& - q^{\scriptscriptstyle13} z_{\scriptscriptstyle 1}^{\scriptscriptstyle 4} z_{\scriptscriptstyle 2}^{\scriptscriptstyle 4} 
z_{\scriptscriptstyle 3}^{\scriptscriptstyle 4} z_{\scriptscriptstyle 4}^{\scriptscriptstyle 6} 
+ q^{\scriptscriptstyle 12} z_{\scriptscriptstyle 1}^{\scriptscriptstyle 3} z_{\scriptscriptstyle 2}^{\scriptscriptstyle 3} 
z_{\scriptscriptstyle 3}^{\scriptscriptstyle 5} z_{\scriptscriptstyle 4}^{\scriptscriptstyle 6} 
- q^{\scriptscriptstyle 13} z_{\scriptscriptstyle 1}^{\scriptscriptstyle 4} z_{\scriptscriptstyle 2}^{\scriptscriptstyle 3} 
z_{\scriptscriptstyle 3}^{\scriptscriptstyle 3} z_{\scriptscriptstyle 4}^{\scriptscriptstyle 7} 
- q^{\scriptscriptstyle 13} z_{\scriptscriptstyle 1}^{\scriptscriptstyle 3} z_{\scriptscriptstyle 2}^{\scriptscriptstyle 4} 
z_{\scriptscriptstyle 3}^{\scriptscriptstyle 3} z_{\scriptscriptstyle 4}^{\scriptscriptstyle 7} 
+ q^{\scriptscriptstyle 13} z_{\scriptscriptstyle 1}^{\scriptscriptstyle 4} z_{\scriptscriptstyle 2}^{\scriptscriptstyle 4} 
z_{\scriptscriptstyle 3}^{\scriptscriptstyle 3} z_{\scriptscriptstyle 4}^{\scriptscriptstyle 7} 
+ q^{\scriptscriptstyle 14} z_{\scriptscriptstyle 1}^{\scriptscriptstyle 4} z_{\scriptscriptstyle 2}^{\scriptscriptstyle 4} 
z_{\scriptscriptstyle 3}^{\scriptscriptstyle 3} z_{\scriptscriptstyle 4}^{\scriptscriptstyle 7} 
- q^{\scriptscriptstyle 14} z_{\scriptscriptstyle 1}^{\scriptscriptstyle 5} z_{\scriptscriptstyle 2}^{\scriptscriptstyle 4} 
z_{\scriptscriptstyle 3}^{\scriptscriptstyle 3} z_{\scriptscriptstyle 4}^{\scriptscriptstyle 7} 
- q^{\scriptscriptstyle 14} z_{\scriptscriptstyle 1}^{\scriptscriptstyle 4} z_{\scriptscriptstyle 2}^{\scriptscriptstyle 5} 
z_{\scriptscriptstyle 3}^{\scriptscriptstyle 3} z_{\scriptscriptstyle 4}^{\scriptscriptstyle 7} \\
& - q^{\scriptscriptstyle 13} z_{\scriptscriptstyle 1}^{\scriptscriptstyle 3} z_{\scriptscriptstyle 2}^{\scriptscriptstyle 3} 
z_{\scriptscriptstyle 3}^{\scriptscriptstyle 4} z_{\scriptscriptstyle 4}^{\scriptscriptstyle 7} 
+ q^{\scriptscriptstyle 13} z_{\scriptscriptstyle 1}^{\scriptscriptstyle  4} z_{\scriptscriptstyle 2}^{\scriptscriptstyle 3} 
z_{\scriptscriptstyle 3}^{\scriptscriptstyle 4} z_{\scriptscriptstyle 4}^{\scriptscriptstyle 7} 
+ q^{\scriptscriptstyle 14} z_{\scriptscriptstyle 1}^{\scriptscriptstyle 4} z_{\scriptscriptstyle 2}^{\scriptscriptstyle 3} 
z_{\scriptscriptstyle 3}^{\scriptscriptstyle 4} z_{\scriptscriptstyle 4}^{\scriptscriptstyle 7} 
- q^{\scriptscriptstyle 14} z_{\scriptscriptstyle 1}^{\scriptscriptstyle 5} z_{\scriptscriptstyle 2}^{\scriptscriptstyle 3} 
z_{\scriptscriptstyle 3}^{\scriptscriptstyle 4} z_{\scriptscriptstyle 4}^{\scriptscriptstyle 7} 
+ q^{\scriptscriptstyle 13} z_{\scriptscriptstyle 1}^{\scriptscriptstyle 3} z_{\scriptscriptstyle 2}^{\scriptscriptstyle 4} 
z_{\scriptscriptstyle 3}^{\scriptscriptstyle 4} z_{\scriptscriptstyle 4}^{\scriptscriptstyle 7} 
+ q^{\scriptscriptstyle 14} z_{\scriptscriptstyle 1}^{\scriptscriptstyle 3} z_{\scriptscriptstyle 2}^{\scriptscriptstyle 4} 
z_{\scriptscriptstyle 3}^{\scriptscriptstyle 4} z_{\scriptscriptstyle 4}^{\scriptscriptstyle 7} 
- 2 q^{\scriptscriptstyle 14} z_{\scriptscriptstyle 1}^{\scriptscriptstyle 4} z_{\scriptscriptstyle 2}^{\scriptscriptstyle 4} 
z_{\scriptscriptstyle 3}^{\scriptscriptstyle 4} z_{\scriptscriptstyle 4}^{\scriptscriptstyle 7} 
+ q^{\scriptscriptstyle 15} z_{\scriptscriptstyle 1}^{\scriptscriptstyle 5} z_{\scriptscriptstyle 2}^{\scriptscriptstyle 4} 
z_{\scriptscriptstyle 3}^{\scriptscriptstyle 4} z_{\scriptscriptstyle 4}^{\scriptscriptstyle 7} \\ 
& - q^{\scriptscriptstyle 14} z_{\scriptscriptstyle 1}^{\scriptscriptstyle 3} z_{\scriptscriptstyle 2}^{\scriptscriptstyle 5} 
z_{\scriptscriptstyle 3}^{\scriptscriptstyle 4} z_{\scriptscriptstyle 4}^{\scriptscriptstyle 7} 
+ q^{\scriptscriptstyle 15} z_{\scriptscriptstyle 1}^{\scriptscriptstyle 4} z_{\scriptscriptstyle 2}^{\scriptscriptstyle 5} 
z_{\scriptscriptstyle 3}^{\scriptscriptstyle 4} z_{\scriptscriptstyle 4}^{\scriptscriptstyle 7} 
- q^{\scriptscriptstyle 14} z_{\scriptscriptstyle 1}^{\scriptscriptstyle 4} z_{\scriptscriptstyle 2}^{\scriptscriptstyle 3} 
z_{\scriptscriptstyle 3}^{\scriptscriptstyle 5} z_{\scriptscriptstyle 4}^{\scriptscriptstyle 7} 
- q^{\scriptscriptstyle 14} z_{\scriptscriptstyle 1}^{\scriptscriptstyle 3} z_{\scriptscriptstyle 2}^{\scriptscriptstyle 4} 
z_{\scriptscriptstyle 3}^{\scriptscriptstyle 5} z_{\scriptscriptstyle 4}^{\scriptscriptstyle 7} 
+ q^{\scriptscriptstyle 15} z_{\scriptscriptstyle 1}^{\scriptscriptstyle 4} z_{\scriptscriptstyle 2}^{\scriptscriptstyle 4} 
z_{\scriptscriptstyle 3}^{\scriptscriptstyle 5} z_{\scriptscriptstyle 4}^{\scriptscriptstyle 7} 
- q^{\scriptscriptstyle 14} z_{\scriptscriptstyle 1}^{\scriptscriptstyle 4} z_{\scriptscriptstyle 2}^{\scriptscriptstyle 4} 
z_{\scriptscriptstyle 3}^{\scriptscriptstyle 4} z_{\scriptscriptstyle 4}^{\scriptscriptstyle 8} 
+ q^{\scriptscriptstyle 15} z_{\scriptscriptstyle 1}^{\scriptscriptstyle 5} z_{\scriptscriptstyle 2}^{\scriptscriptstyle 4} 
z_{\scriptscriptstyle 3}^{\scriptscriptstyle 4} z_{\scriptscriptstyle 4}^{\scriptscriptstyle 8} 
+ q^{\scriptscriptstyle 15} z_{\scriptscriptstyle 1}^{\scriptscriptstyle 4} z_{\scriptscriptstyle 2}^{\scriptscriptstyle 5} 
z_{\scriptscriptstyle 3}^{\scriptscriptstyle 4} z_{\scriptscriptstyle 4}^{\scriptscriptstyle 8} \\ 
& - q^{\scriptscriptstyle 16} z_{\scriptscriptstyle 1}^{\scriptscriptstyle 5} z_{\scriptscriptstyle 2}^{\scriptscriptstyle 5} 
z_{\scriptscriptstyle 3}^{\scriptscriptstyle 4} z_{\scriptscriptstyle 4}^{\scriptscriptstyle 8} + 
q^{\scriptscriptstyle 15} z_{\scriptscriptstyle 1}^{\scriptscriptstyle 4} z_{\scriptscriptstyle 2}^{\scriptscriptstyle 4} 
z_{\scriptscriptstyle 3}^{\scriptscriptstyle 5} z_{\scriptscriptstyle 4}^{\scriptscriptstyle 8} - 
q^{\scriptscriptstyle 16} z_{\scriptscriptstyle 1}^{\scriptscriptstyle 5} z_{\scriptscriptstyle 2}^{\scriptscriptstyle 4}
z_{\scriptscriptstyle 3}^{\scriptscriptstyle 5} z_{\scriptscriptstyle 4}^{\scriptscriptstyle 8} - 
q^{\scriptscriptstyle 16} z_{\scriptscriptstyle 1}^{\scriptscriptstyle 4} z_{\scriptscriptstyle 2}^{\scriptscriptstyle 5} 
z_{\scriptscriptstyle 3}^{\scriptscriptstyle 5} z_{\scriptscriptstyle 4}^{\scriptscriptstyle 8} + 
q^{\scriptscriptstyle 18} z_{\scriptscriptstyle 1}^{\scriptscriptstyle 5} z_{\scriptscriptstyle 2}^{\scriptscriptstyle 5} 
z_{\scriptscriptstyle 3}^{\scriptscriptstyle 5} z_{\scriptscriptstyle 4}^{\scriptscriptstyle 9}
\end{split}
\end{equation*} 
and denominator
\begin{equation*}
\begin{split}
D(z_{\scriptscriptstyle 1}^{}\!, \, z_{\scriptscriptstyle 2}^{}, \, z_{\scriptscriptstyle 3}^{}, \, z_{\scriptscriptstyle 4}^{}; q) = \, 
&\left(1 - q^{} z_{\scriptscriptstyle 1}^{} \!\right) \left(1 - q^{}  z_{\scriptscriptstyle 2}^{} \right)  \left(1 - q^{} z_{\scriptscriptstyle 3}^{} \right) 
\left(1 - q^{} z_{\scriptscriptstyle 4}^{} \right) 
\big(1 - q^{\scriptscriptstyle 3} z_{\scriptscriptstyle 1}^{\scriptscriptstyle 2} z_{\scriptscriptstyle 4}^{\scriptscriptstyle 2} \big) 
\big(1 - q^{\scriptscriptstyle 3} z_{\scriptscriptstyle 2}^{\scriptscriptstyle 2} z_{\scriptscriptstyle 4}^{\scriptscriptstyle 2} \big) 
\big(1 - q^{\scriptscriptstyle 3} z_{\scriptscriptstyle 3}^{\scriptscriptstyle 2} z_{\scriptscriptstyle 4}^{\scriptscriptstyle 2} \big) \\
& \cdot \big(1 - q^{\scriptscriptstyle 4} z_{\scriptscriptstyle 1}^{\scriptscriptstyle 2} z_{\scriptscriptstyle 2}^{\scriptscriptstyle 2} 
z_{\scriptscriptstyle 4}^{\scriptscriptstyle 2} \big) \big(1 - q^{\scriptscriptstyle 4} z_{\scriptscriptstyle 1}^{\scriptscriptstyle 2} 
z_{\scriptscriptstyle 3}^{\scriptscriptstyle 2} z_{\scriptscriptstyle 4}^{\scriptscriptstyle 2} \big) 
\big(1 - q^{\scriptscriptstyle 4} z_{\scriptscriptstyle 2}^{\scriptscriptstyle 2} z_{\scriptscriptstyle 3}^{\scriptscriptstyle 2} 
z_{\scriptscriptstyle 4}^{\scriptscriptstyle 2} \big) 
\big(1 - q^{\scriptscriptstyle 5} z_{\scriptscriptstyle 1}^{\scriptscriptstyle 2} z_{\scriptscriptstyle 2}^{\scriptscriptstyle 2} 
z_{\scriptscriptstyle 3}^{\scriptscriptstyle 2} z_{\scriptscriptstyle 4}^{\scriptscriptstyle 2} \big) 
\big(1 - q^{\scriptscriptstyle 6} z_{\scriptscriptstyle 1}^{\scriptscriptstyle 2} 
z_{\scriptscriptstyle 2}^{\scriptscriptstyle 2} z_{\scriptscriptstyle 3}^{\scriptscriptstyle 2} z_{\scriptscriptstyle 4}^{\scriptscriptstyle 4} \big).
\end{split}
\end{equation*} 
In other words, with notation as in Section \ref{Three expressions of Multiple {D}irichlet series}, Eq. \!\eqref{eq: MDS-vers1}, 
we have that  
\begin{equation*} 
Z\big(q^{- s_{\scriptscriptstyle 1}}\!, \, q^{- s_{\scriptscriptstyle 2}}\!, \, q^{- s_{\scriptscriptstyle 3}}\!, \, q^{- s_{\scriptscriptstyle 4}}; q \big) 
= Z^{(1)}(\mathbf{s}; 1, 1) 
\; = \sum_{d = d_{\scriptscriptstyle 0}^{} d_{\scriptscriptstyle 1}^{2}}  \, 
\frac{\prod_{i = 1}^{3}\, L(s_{\scriptscriptstyle i}, \chi_{d_{\scriptscriptstyle 0}}) \cdot 
P_{\scriptscriptstyle d}(s_{\scriptscriptstyle 1}\!, \, s_{\scriptscriptstyle 2}, \, s_{\scriptscriptstyle 3}; 
\chi_{d_{\scriptscriptstyle 0}})}{|d|^{s_{\scriptscriptstyle 4}}}.
\end{equation*} 
Then the function  
\begin{equation} \label{rational-function-D4} 
f(z_{\scriptscriptstyle 1}^{}\!, \, z_{\scriptscriptstyle 2}^{}, \, z_{\scriptscriptstyle 3}^{}, \, z_{\scriptscriptstyle 4}^{}; q) = 
f_{\scriptscriptstyle D_{\scriptscriptstyle 4}}\!(z_{\scriptscriptstyle 1}^{}\!, \, z_{\scriptscriptstyle 2}^{}, \, z_{\scriptscriptstyle 3}^{}, \, z_{\scriptscriptstyle 4}^{}; q): 
= Z(q z_{\scriptscriptstyle 1}^{}\!, \, q z_{\scriptscriptstyle 2}^{}, \, q z_{\scriptscriptstyle 3}^{}, \, q z_{\scriptscriptstyle 4}^{}; 
1\slash q)
\end{equation} 
is precisely the rational function obtained by considering the Chinta-Gunnells average \eqref{eq: CG-average} 
for the root system $D_{4}$ with central node corresponding to $z_{\scriptscriptstyle 4}^{}.$ 
\!This fact can be checked either by a direct computation of the Chinta-Gunnells average, 
or by simply verifying that the rational function $f$ is $W$-invariant with respect to the Weyl group action defined in 
\ref{Chi-Gunn}, $f(0, \ldots, 0; q) = 1,$ and that it satisfies the condition 
\eqref{eq: limiting condition}. Expanding $f$ in a power series  
\begin{equation*} 
f(z_{\scriptscriptstyle 1}^{}\!, \, z_{\scriptscriptstyle 2}^{}, \, z_{\scriptscriptstyle 3}^{}, \, z_{\scriptscriptstyle 4}^{}; q) \;\; =  
\sum_{k_{\scriptscriptstyle 1}\!, \, k_{\scriptscriptstyle 2}, \, k_{\scriptscriptstyle 3}, \, l \ge 0}  
\, a(k_{\scriptscriptstyle 1}\!,  \, k_{\scriptscriptstyle 2},  \, k_{\scriptscriptstyle 3}, \,  l; q) z_{\scriptscriptstyle 1}^{\scriptscriptstyle k_{\scriptscriptstyle 1}} 
\! \cdots \, z_{\scriptscriptstyle 4}^{\scriptscriptstyle l} 
\end{equation*} 
we see easily that 
\begin{equation} \label{initial-conditions-f} 
a(k_{\scriptscriptstyle 1}\!,  \, k_{\scriptscriptstyle 2},  \, k_{\scriptscriptstyle 3}, \,  0; q) = a(0,  \, 0,  \, 0, \,  l; q) = 1
\qquad \text{(for all $k_{\scriptscriptstyle 1}\!,  \, k_{\scriptscriptstyle 2}, \, k_{\scriptscriptstyle 3}, \, l \ge 0$).}
\end{equation} 
When $l = 1$ these coefficients vanish, unless 
$
k_{\scriptscriptstyle 1} \! = k_{\scriptscriptstyle 2} \! = k_{\scriptscriptstyle 3} \! = 0. 
$ 
Moreover, if $\sum k_{\scriptscriptstyle i} \equiv l \equiv 1 \!\!\pmod 2$ then 
\begin{equation} \label{odd-l-ranked-coeff} 
a(k_{\scriptscriptstyle 1}\!,  \, k_{\scriptscriptstyle 2},  \, k_{\scriptscriptstyle 3}, \,  l; q) = 0. 
\end{equation} 
Now define  
\begin{equation*}
f_{\scriptscriptstyle \mathrm{odd}}(z_{\scriptscriptstyle 1}^{}\!, \, z_{\scriptscriptstyle 2}^{}, \, z_{\scriptscriptstyle 3}^{}, \, z_{\scriptscriptstyle 4}^{}; q) = (f(z_{\scriptscriptstyle 1}^{}\!, \, z_{\scriptscriptstyle 2}^{}, \, z_{\scriptscriptstyle 3}^{}, \, z_{\scriptscriptstyle 4}^{}; q) \, - \, f(z_{\scriptscriptstyle 1}^{}\!, \, z_{\scriptscriptstyle 2}^{}, \, z_{\scriptscriptstyle 3}^{}, \, - z_{\scriptscriptstyle 4}^{}; q)) \slash 2
\end{equation*} 
and 
\begin{equation*}
f_{\scriptscriptstyle \mathrm{even}}(z_{\scriptscriptstyle 1}^{}\!, \, z_{\scriptscriptstyle 2}^{}, \, z_{\scriptscriptstyle 3}^{}, \, z_{\scriptscriptstyle 4}^{}; q) = (f(z_{\scriptscriptstyle 1}^{}\!, \, z_{\scriptscriptstyle 2}^{}, \, z_{\scriptscriptstyle 3}^{}, \, z_{\scriptscriptstyle 4}^{}; q) \, + \, f(z_{\scriptscriptstyle 1}^{}\!, \, z_{\scriptscriptstyle 2}^{}, \, z_{\scriptscriptstyle 3}^{}, \, - z_{\scriptscriptstyle 4}^{}; q)) \slash 2.
\end{equation*} 
One checks that the numerator of 
$
f_{\scriptscriptstyle \mathrm{odd}}(z_{\scriptscriptstyle 1}^{}\!, \, z_{\scriptscriptstyle 2}^{}, \, z_{\scriptscriptstyle 3}^{}, \, z_{\scriptscriptstyle 4}^{}; q)
$ 
is divisible by 
$
(1 - z_{\scriptscriptstyle 1}^{}) (1 - z_{\scriptscriptstyle 2}^{}) (1 - z_{\scriptscriptstyle 3}^{}), 
$ 
and thus we can write 
\begin{equation*}
\begin{split}
f(z_{\scriptscriptstyle 1}^{}\!, \, z_{\scriptscriptstyle 2}^{}, \, z_{\scriptscriptstyle 3}^{}, \, z_{\scriptscriptstyle 4}^{}; q)
& = f_{\scriptscriptstyle \mathrm{even}}(z_{\scriptscriptstyle 1}^{}\!, \, z_{\scriptscriptstyle 2}^{}, \, z_{\scriptscriptstyle 3}^{}, \, z_{\scriptscriptstyle 4}^{}; q) \, + \, f_{\scriptscriptstyle \mathrm{odd}}(z_{\scriptscriptstyle 1}^{}\!, \, z_{\scriptscriptstyle 2}^{}, \, z_{\scriptscriptstyle 3}^{}, \, z_{\scriptscriptstyle 4}^{}; q)  \\ 
& = (1 - z_{\scriptscriptstyle 1}^{})^{\scriptscriptstyle - 1} (1 - z_{\scriptscriptstyle 2}^{})^{\scriptscriptstyle - 1}  
(1 - z_{\scriptscriptstyle 3}^{})^{\scriptscriptstyle - 1} 
\!\!\sum_{l - \text{even}} 
P_{\scriptscriptstyle l}(z_{\scriptscriptstyle 1}\!, \, z_{\scriptscriptstyle 2}, \, z_{\scriptscriptstyle 3}; q) \,
z_{\scriptscriptstyle 4}^{\scriptscriptstyle l} \; + 
\sum_{l - \text{odd}} P_{\scriptscriptstyle l}(z_{\scriptscriptstyle 1}\!, \, z_{\scriptscriptstyle 2}, \, z_{\scriptscriptstyle 3}; q) \,
z_{\scriptscriptstyle 4}^{\scriptscriptstyle l}
\end{split}
\end{equation*} 
for $|z_{\scriptscriptstyle 4}|$ sufficiently small (depending on the other variables). \!The symmetric polynomials 
$
P_{\scriptscriptstyle l}(z_{\scriptscriptstyle 1}\!, \, z_{\scriptscriptstyle 2}, \, z_{\scriptscriptstyle 3}; q) 
$ 
defined by this expression of $f$ were used in Section \ref{Three expressions of Multiple {D}irichlet series} 
to define the {D}irichlet polynomial \eqref{eq: polyPd}. \!Similarly, the polynomials 
$
Q_{\scriptscriptstyle \underline{k}}(z_{\scriptscriptstyle 4}; q)
$ 
are defined by the expansion 
\begin{equation*}
\begin{split}
f(z_{\scriptscriptstyle 1}^{}\!, \, z_{\scriptscriptstyle 2}^{}, \, z_{\scriptscriptstyle 3}^{}, \, z_{\scriptscriptstyle 4}^{}; q)
= (1 - z_{\scriptscriptstyle 4}^{})^{\scriptscriptstyle - 1} \!\!\!\sum_{\substack{\underline{k} 
= (k_{\scriptscriptstyle 1}\!, \, k_{\scriptscriptstyle 2}, \, k_{\scriptscriptstyle 3}) \\ |\underline{k}| - \text{even}}} 
Q_{\scriptscriptstyle \underline{k}}(z_{\scriptscriptstyle 4}; q)\,
z_{\scriptscriptstyle 1}^{\scriptscriptstyle k_{\scriptscriptstyle 1}} 
z_{\scriptscriptstyle 2}^{\scriptscriptstyle k_{\scriptscriptstyle 2}} 
z_{\scriptscriptstyle 3}^{\scriptscriptstyle k_{\scriptscriptstyle 3}}  \;\;  + 
\sum_{\substack{\underline{k} 
= (k_{\scriptscriptstyle 1}\!, \, k_{\scriptscriptstyle 2}, \, k_{\scriptscriptstyle 3}) \\ |\underline{k}| - \text{odd}}} 
Q_{\scriptscriptstyle \underline{k}}(z_{\scriptscriptstyle 4}; q)\,
z_{\scriptscriptstyle 1}^{\scriptscriptstyle k_{\scriptscriptstyle 1}} 
z_{\scriptscriptstyle 2}^{\scriptscriptstyle k_{\scriptscriptstyle 2}} 
z_{\scriptscriptstyle 3}^{\scriptscriptstyle k_{\scriptscriptstyle 3}}. 
\end{split}
\end{equation*} 
From the $W$-invariance of $f$ we deduce the functional equations:       
\begin{equation} \label{eq: polynomials-P-Q-func-eq}
P_{\scriptscriptstyle l}(z_{\scriptscriptstyle 1}, \, z_{\scriptscriptstyle 2}, \, z_{\scriptscriptstyle 3}; q)
= (\sqrt{q} \, z_{\scriptscriptstyle 1})^{l - a_{\scriptscriptstyle l}} \, 
P_{\scriptscriptstyle l}\bigg(\frac{1}{q \, z_{\scriptscriptstyle 1}}, \, z_{\scriptscriptstyle 2}, \, z_{\scriptscriptstyle 3}; q\bigg)
\;\;\;\; \mathrm{and} \;\;\;\; 
Q_{\scriptscriptstyle \underline{k}}(z_{\scriptscriptstyle 4}; q)
= (\sqrt{q} \, z_{\scriptscriptstyle 4})^{|\underline{k}| - a_{\scriptscriptstyle |\underline{k}|}} \, 
Q_{\scriptscriptstyle \underline{k}}\bigg(\frac{1}{q \, z_{\scriptscriptstyle 4}}; q\bigg)
\end{equation} 
with $a_{\scriptscriptstyle n} \!= 0$ or $1$ according as $n$ is even or odd.

For the reader's convenience, we include here the explicit expressions of some specializations of the rational functions 
introduced in this appendix. We have  
\begin{equation*}
f_{\scriptscriptstyle \mathrm{odd}}\big(q^{\scriptscriptstyle -\frac{1}{2}}\!, \, q^{\scriptscriptstyle -\frac{1}{2}}\!, \, 
q^{\scriptscriptstyle -\frac{1}{2}}\!, \, z  ; \, q \big) \, = \, 
\frac{z \, (1 + 7 z^{\scriptscriptstyle 2} + 7 z^{\scriptscriptstyle 4} + z^{\scriptscriptstyle 6})}
{(1 - z^{\scriptscriptstyle 2})^{\scriptscriptstyle 7}  (1 - q \, z^{\scriptscriptstyle 4})} 
\end{equation*} 
\begin{equation*} 
\begin{split}
& f_{\scriptscriptstyle \mathrm{even}}\big(q^{\scriptscriptstyle -\frac{1}{2}}\!, \, q^{\scriptscriptstyle -\frac{1}{2}}\!, \, 
q^{\scriptscriptstyle -\frac{1}{2}}\!, \, z  ; \, q \big) \\
& = \, \big(1 - q^{\scriptscriptstyle - \frac{1}{2}} \big)^{\scriptscriptstyle - 3} \cdot
\frac{1 + \big(7  - 14 \, q^{\scriptscriptstyle - \frac{1}{2}} + 6 \, q^{\scriptscriptstyle - 1} 
- q^{\scriptscriptstyle - \frac{3}{2}} \big)\, z^{\scriptscriptstyle 2} + 
7\big(1 - 4\, q^{\scriptscriptstyle - \frac{1}{2}} + 4\, q^{\scriptscriptstyle -1} 
- q^{\scriptscriptstyle - \frac{3}{2}} \big)\, z^{\scriptscriptstyle 4} 
+ \big(1 - 6 \, q^{\scriptscriptstyle - \frac{1}{2}} + 14\, q^{\scriptscriptstyle -1} 
- 7\, q^{\scriptscriptstyle - \frac{3}{2}} \big)\, z^{\scriptscriptstyle 6} 
- q^{\scriptscriptstyle - \frac{3}{2}} z^{\scriptscriptstyle 8}}
{(1 - z^{\scriptscriptstyle 2})^{\scriptscriptstyle 7} (1 - q\,  z^{\scriptscriptstyle 4})}
\end{split}
\end{equation*} 
and 
\begin{equation*} 
\begin{split}
& f_{\scriptscriptstyle \mathrm{even}}\big( - q^{\scriptscriptstyle -\frac{1}{2}}\!, \, - \, q^{\scriptscriptstyle -\frac{1}{2}}\!, \, 
- \, q^{\scriptscriptstyle -\frac{1}{2}}\!, \, z  ; \, q \big) \\
& = \, \big(1 + q^{\scriptscriptstyle - \frac{1}{2}} \big)^{\scriptscriptstyle - 3} \cdot
\frac{1 + \big(7  + 14 \, q^{\scriptscriptstyle - \frac{1}{2}} + 6 \, q^{\scriptscriptstyle - 1} 
+ q^{\scriptscriptstyle - \frac{3}{2}} \big)\, z^{\scriptscriptstyle 2} + 
7\big(1 + 4\, q^{\scriptscriptstyle - \frac{1}{2}} + 4\, q^{\scriptscriptstyle -1} 
+ q^{\scriptscriptstyle - \frac{3}{2}} \big)\, z^{\scriptscriptstyle 4} 
+ \big(1 + 6 \, q^{\scriptscriptstyle - \frac{1}{2}} + 14\, q^{\scriptscriptstyle -1} 
+ 7\, q^{\scriptscriptstyle - \frac{3}{2}} \big)\, z^{\scriptscriptstyle 6} 
+ q^{\scriptscriptstyle - \frac{3}{2}} z^{\scriptscriptstyle 8}}
{(1 - z^{\scriptscriptstyle 2})^{\scriptscriptstyle 7} (1 - q\,  z^{\scriptscriptstyle 4})}.
\end{split}
\end{equation*} 
One can use these formulas to estimate 
$
P_{\scriptscriptstyle l}\big(\pm  q^{\scriptscriptstyle -\frac{1}{2}}\!, \, \pm \, q^{\scriptscriptstyle -\frac{1}{2}}\!, \, 
\pm \, q^{\scriptscriptstyle -\frac{1}{2}}; q\big).
$ 
Indeed, taking $|z| = q^{\scriptscriptstyle -\frac{1}{4} - \eta}$ for small $\eta > 0,$ we have the inequalities  
\begin{equation*}
\big|f_{\scriptscriptstyle \mathrm{odd}}\big(q^{\scriptscriptstyle -\frac{1}{2}}\!, \, q^{\scriptscriptstyle -\frac{1}{2}}\!, \, 
q^{\scriptscriptstyle -\frac{1}{2}}\!, \, z  ; \, q \big) \, z^{\scriptscriptstyle -1}\big| \, \le \, 
\frac{1 + 7 |z|^{\scriptscriptstyle 2} + 7 |z|^{\scriptscriptstyle 4} + |z|^{\scriptscriptstyle 6}}
{(1 - |z|^{\scriptscriptstyle 2})^{\scriptscriptstyle 7}  (1 - q \, |z|^{\scriptscriptstyle 4})} \, <  \, 
\bigg(\frac{1 + |z|^{\scriptscriptstyle 2}}{1 - |z|^{\scriptscriptstyle 2}}\bigg)^{\!\! \scriptscriptstyle 7} 
\cdot \,  \frac{1}{1 - q \, |z|^{\scriptscriptstyle 4}} \, < \, 
\bigg(\frac{\sqrt{q} + 1}{\sqrt{q} - 1}\bigg)^{\!\! \scriptscriptstyle 7} 
\cdot \,  \frac{1}{1 - q^{\scriptscriptstyle - 4 \, \eta}}.
\end{equation*} 
The same bound holds for 
$
\big(1 \mp q^{\scriptscriptstyle - \frac{1}{2}} \big)^{\scriptscriptstyle 3}
\big|f_{\scriptscriptstyle \mathrm{even}}\big(\pm q^{\scriptscriptstyle -\frac{1}{2}}\!, \, \pm \, q^{\scriptscriptstyle -\frac{1}{2}}\!, \, 
\pm \, q^{\scriptscriptstyle -\frac{1}{2}}\!, \, z  ; \, q \big)\big|,
$ 
since by the maximum principle we have 
\begin{equation*}
\begin{split} 
&\big(1 \mp q^{\scriptscriptstyle - \frac{1}{2}} \big)^{\scriptscriptstyle 3}
\big|(1 - q\,  z^{\scriptscriptstyle 4}) f_{\scriptscriptstyle \mathrm{even}}\big(\pm q^{\scriptscriptstyle -\frac{1}{2}}\!, \, \pm \, q^{\scriptscriptstyle -\frac{1}{2}}\!, \, 
\pm \, q^{\scriptscriptstyle -\frac{1}{2}}\!, \, z  ; \, q \big)\big|  \\ 
& < \, \big(1 \mp q^{\scriptscriptstyle - \frac{1}{2}} \big)^{\scriptscriptstyle 3} \cdot \, 
\underset{|u| = q^{\scriptscriptstyle - 1 \slash 4}}{\mathrm{max}}\; \big|(1 - q\,  u^{\scriptscriptstyle 4}) 
f_{\scriptscriptstyle \mathrm{even}}\big(\pm q^{\scriptscriptstyle -\frac{1}{2}}\!, \, \pm \, q^{\scriptscriptstyle -\frac{1}{2}}\!, \, 
\pm \, q^{\scriptscriptstyle -\frac{1}{2}}\!, \, u  ; \, q \big)\big| \\
& \le \, \bigg(\frac{\sqrt{q} + 1}{\sqrt{q} - 1}\bigg)^{\!\! \scriptscriptstyle 7}. 
\end{split}
\end{equation*} 
If $q \ge 5,$ we have 
\begin{equation*} 
\bigg(\frac{\sqrt{q} + 1}{\sqrt{q} - 1}\bigg)^{\!\! \scriptscriptstyle 7} 
\cdot \,  \frac{1}{1 - q^{\scriptscriptstyle - 4 \, \eta}}  <  \frac{843}{1 - 5^{\scriptscriptstyle - 4 \, \eta}}.
\end{equation*} 
Then by applying Cauchy's inequality we obtain:

\vskip10pt
\begin{prop}\label{poly-P-estimate} --- For every small positive $\eta,$ $q\ge 5$ and $l \ge 1$ we have 
\begin{equation*} 
\big|P_{\scriptscriptstyle l}\big(\pm  q^{\scriptscriptstyle -\frac{1}{2}}\!, \, \pm \, q^{\scriptscriptstyle -\frac{1}{2}}\!, \, 
\pm \, q^{\scriptscriptstyle -\frac{1}{2}}; q\big)\big|  \, <  \frac{843}{1 - 5^{\scriptscriptstyle - 4 \, \eta}} 
\, q^{(l - a_{\scriptscriptstyle l})\left(\scriptscriptstyle \frac{1}{4} + \eta \right)}
\end{equation*} 
where $a_{\scriptscriptstyle l} \!= 0$ or $1$ according as $l$ is even or odd.
\end{prop}

%


\end{document}